\documentclass[final]{siamltex}  

\usepackage[dvips]{graphicx}	
\usepackage[dvipsnames,usenames]{color}

\usepackage{cancel}
\usepackage[loose,nice]{units}

\usepackage{myStyle}		
\usepackage{myAliases}

\title{Radiative Decay of Bubble Oscillations in a Compressible Fluid}

\author{A.~M.~Shapiro\thanks{School of Natural Sciences, 
					University of California, Merced, CA }	
					\and M.~I.~Weinstein\thanks{Department of Applied Physics and 							Applied Mathematics, Columbia University, New York, NY}}

\begin{document}

\maketitle 

\begin{abstract}
Consider the  dynamics of a gas bubble in an inviscid,  compressible liquid with surface tension. Kinematic and dynamic boundary conditions couple the bubble surface deformation dynamics with the dynamics of waves in the fluid. 
This system has a spherical equilibrium state, resulting from the balance of the pressure at infinity and the gas pressure within the bubble.
We study the linearized dynamics about this equilibrium state in a center of mass frame.
 We prove that the velocity potential and bubble surface perturbation satisfy point-wise in space exponential time-decay estimates. The time-decay rate is governed by the imaginary parts {\it scattering resonances}. These are characterized by  a non-selfadjoint spectral problem or as  pole singularities in the lower half plane of the analytic continuation of a resolvent operator from the upper half plane, across the real axis  into the lower half plane. 
The time-decay estimates are a consequence of resonance mode expansions for the velocity potential and bubble surface perturbations.
The weakly compressible case (small Mach number, $\epsilon$),  is a singular perturbation of the incompressible limit. The scattering resonances which govern the remarkably slow time-decay, are {\it Rayleigh resonances}, associated with capillary waves, due to surface tension, on the bubble surface, which impart their energy slowly to the unbounded fluid. Rigorous results,  asymptotics and high-precision numerical studies, indicate that the Rayleigh resonances which are closest to the real axis satisfy
 ${\bigl| \frac{\Im \lambda_\star(\epsilon)}{\Re \lambda_\star(\epsilon)}\bigr|}=\cO\!\left(\exp(-\kappa\ \We\ \epsilon^{-2})\right),\ \kappa>0$. Here, $\We$ denotes the Weber number, a dimensionless ratio comparing inertia and surface tension.
  To obtain the above results we prove a general result estimating the Neumann to Dirichlet map for the wave equation, exterior to a sphere.
\end{abstract}



\pagestyle{myheadings}
\thispagestyle{plain}
\markboth{A.~M.~SHAPIRO AND M.~I.~WEINSTEIN}{RADIATIVE DECAY OF BUBBLE OSCILLATIONS IN A COMPRESSIBLE FLUID}
%
%
\section{Introduction} \label{sec:intro}

Consider a gas bubble, surrounded by an unbounded, inviscid and incompressible fluid with surface tension. This system has a family of equilibria, consisting of translates of a   spherical gas bubble, whose radius is set by a balance of pressure at infinity with pressure inside the bubble.  
In this paper we prove pointwise in space  time-decay estimates for the linearized evolution near this family of equilibria. We  also obtain very precise information on the rate of decay.\medskip
 
\noindent {\bf Background:}\ The dynamics of gas bubbles in a liquid play an important role in many fields.  Examples include underwater explosion bubbles~\cite{Keller:1956gc}  ($\unit[15]{cm}$), seismic wave-producing bubbles in magma~\cite{Ripepe:1999fu}, bubbly flows behind ships and propellors ($\unit[1]{cm}$), bubbles at the ocean surface~\cite{Longuet-Higgins:1989ux} ($\unit[0.015-0.5]{cm}$), microfluidics~\cite{Xu:2007fk} ($\unit[50]{\mu m}$), bubbles used as contrast agents in medical imaging~\cite{Ferrara:2007xr} ($\unit[2]{\mu m}$), and sonoluminescence~\cite{Brenner:1995yu,Putterman:1998gi} ($\unit[0.1-10]{\mu m}$). For a discussion of these and other applications of bubble dynamics, see the  excellent review articles~\cite{Leighton:2004bi,Prosperetti:2004cz} and  references cited therein.

The dynamics of a bubble in a liquid are governed by the  compressible Navier-Stokes equations in the liquid external to the bubble, a description of the gas within the bubble, and boundary conditions (kinematic and dynamic) which couple the fluid and gas. We assume the gas inside the bubble to be at a uniform pressure throughout and to satisfy a thermodynamic law relating the bubble pressure to the bubble volume.
Rayleigh initiated  the study of bubble dynamics 
and derived an equation for the radial oscillations of a spherically symmetric gas bubble in an incompressible, inviscid liquid with surface tension~\cite{Rayleigh:1917zm,Lamb:1993mu}.
This problem has a  spherical equilibrium, balancing the pressure at infinity and the pressure within the bubble. A general (asymmetric) perturbation of this spherical equilibrium will excite all harmonics and Rayleigh showed, in the linear approximation, that each harmonic executes {\it undamped} time-periodic oscillations~\cite{Lamb:1993mu}. 
 With viscosity present, the pulsating bubble oscillations are  damped and  the spherical equilibrium is approached.

A second very important damping mechanism, energy preserving in nature, is due to compressibility of the fluid. In a non-viscous and compressible fluid, a perturbation of the equilibrium bubble will, due to coupling  at the gas-liquid interface, generate acoustic waves in the fluid which, in an unbounded region, propagate to infinity. The bubble dynamics acquire an effective damping, due to energy transfer to the fluid, and its propagation to infinity. If surface tension is included asymmetric modes should damp  and the bubble shape should approach that of a sphere as time advances.

Keller and co-workers~\cite{Keller:1956gc,Epstein:1972uc,Keller:1980rm} modeled slight compressibility of the fluid by the linear wave equation, exterior to the bubble. Both the bubble interior pressure law and the boundary conditions are kept as in the incompressible case.  In the spherically symmetric setting this leads to an ODE, which captures acoustic radiation damping. A systematic derivation of the model of Keller \etal, in the spherically symmetric setting, was presented by Prosperretti-Lezzi~\cite{Prosperetti:1986el,Lezzi:1987qr}.

\medskip

 We expect, for the inviscid, compressible problem  with surface tension, that the family of translates of the spherical equilibrium is (locally) nonlinearly asymptotically stable. Specifically, a small perturbation of the equilibrium spherical bubble, will induce  translational motion of the bubble and deformation of its surface. We expect that, {\it in a frame of reference which  moves with the bubble center of mass}, $\com(t)$, a small perturbation of the spherical bubble will damp toward a spherical equilibrium shape. 
   As in many studies of asymptotic stability for coherent structures in  spatially-extended conservative nonlinear PDEs, linear decay estimates of the type established in this work can be expected to play a role in the estimation of the convergence to equilibrium for the nonlinear dynamics.

\subsection{PDEs for a gas bubble in a compressible liquid} 
\label{sec:linprob}
Consider a gas bubble, occupying a bounded region $B(t)$, surrounded by an  inviscid and  compressible fluid with surface tension. We consider the boundary of the bubble, $\D B(t)$, to be  parametrized by a function $\vect{R}: (\alpha,t)\mapsto\vect{R}(\alpha,t)\in\D B(t)\subset\R^3$, with parameter $\alpha$, \eg\ spherical coordinates.

The time-evolution of the fluid is governed by the system:
\begin{subequations}  \label{eq:euler1}
\begin{align}
\D_{t} \fixed{\vect{u}} + \left( \fixed{\vect{u}} \cdot \grad \right) \fixed{\vect{u}} + \frac{1}{\fixed{\rho}} \grad \fixed{p}(\rho) &= \vect{0}, && \fixed{\vect{x}}\in \R^{3}\setminus \fixed{B}(t) \label{momentum} \\
 \D_{\fixed{t}} \fixed{\rho} + \div (\fixed{\rho}\fixed{\vect{u}}) &= \vect{0}, && \fixed{\vect{x}}\in \R^{3}\setminus \fixed{B}(t) \label{mass} \\
(\fixed{\vect{u}}\circ\vect{R}) \cdot \unitn &= \D_{t}  \fixed{\vect{R}} \cdot \unitn, &&   \D\fixed{B}(\fixed{t}) \label{kinematic}\\
\left. p_{\rm bubble}\right|_{\D B(t)}\ -\ \left. p_{\rm fluid}\right|_{\D B(t)} &= 2\sigma\, H[\vect{R}], &&  \D\fixed{B}(\fixed{t}),
\label{dynamic} \end{align}
\end{subequations}
where $H[\vect{R}]=\frac{1}{2}\,\div\unitn$ denotes the mean curvature at location $\vect{R}\in\D B(t)$; see, for example,~\cite{Lamb:1993mu} (Article~275, Equation~5), or~\cite{Carmo:1976ve} (Section 3-3).

The pressure within the fluid is assumed to obey an equation of state:
$p=p_{\rm fluid}=p(\rho)$.
%
Equations~\eqref{momentum} and~\eqref{mass} express conservation of momentum and mass.
Equation~\eqref{kinematic} is the kinematic boundary condition, {\it i.e.} the normal velocity, $\D_t\vect{R}\cdot {\unitn}$,   of the material point on the bubble surface moves with the normal velocity of the fluid. The Young-Laplace boundary condition, also called the dynamic boundary condition,~\eqref{dynamic}, expresses that the jump in pressure at the fluid bubble interface is proportional to the mean curvature~\cite{Leal:2007fk}:
 \begin{equation}
\left. p_{\rm bubble}\right|_{\D B(t)}\ -\ \left. p_{\rm fluid}\right|_{\D B(t)}\ =\ {\rm Surface\ Tension}\ \times \ (2 \times {\rm Mean\ Curvature}).
 \label{p-jump}\end{equation}
 
\noindent We assume that the   pressure within the bubble is spatially uniform and given by the polytropic gas law~\cite{Longuet-Higgins:1989ux}:
\beq
 \left. p_{\rm bubble}\right|_{\D B(t)} = \fixed{P}_{\fixed{B}}= \frac{k}{\abs{\fixed{B}(\fixed{t})}^{\gamma}} =P_{\text{eq}} \left( \frac{\frac{4\pi}{3}a^{3}}{\abs{\fixed{B}(\fixed{t})}} \right)^{\gamma},\ \ \gamma>1.\label{adiabatic}
\eeq
The pressure within the fluid is assumed to satisfy a general smooth
 relation of the form: $p=p(\rho)$, so that the system is determined by only one state variable.

Finally we assume that the initial velocity is irrotational, $\grad\wedge \vect{u}_0=\vect{0}.$ It follows that the velocity field remains irrotational for all $t\ge0$; see, for example,~\cite{Lamb:1993mu}, Article 33. Thus, there is a single-valued velocity potential $\Phi$, such that
$
 \vect{u}(\vect{x},t) = \grad \Phi.
$\\
 
\noindent {\bf Equilibrium solutions:}\ Equations~\eqref{eq:euler1} have a spherically symmetric equilibrium solution:
\begin{align}
\vect{u} &= \vect{0},\ \ \vect{R} = a \,\unitr =\ a\ \frac{\vect{x}-\vect{x_0}}{|\vect{x}-\vect{x_0}|}\ \  \ p = p_{\infty},\ \ \ \ \rho = \rho_{\infty}.
\end{align}
The equilibrium bubble radius, $a$, is uniquely determined via the dynamic boundary condition%
\begin{equation}
\frac{k}{\left( \frac{4\pi}{3}a^{3} \right)^{\gamma}} - p_{\infty} = \frac{2\sigma}{a}\qquad {\rm or}\qquad
P_{\text{eq}}=p_{\infty} + \frac{2\sigma}{a}.
\label{eq:Peq}
\end{equation}

We consider these dynamics in the linear approximation. Introduce spherical coordinates, $(r,\theta,\phi)=(r,\Omega)$,  with the origin chosen to be the bubble center of mass, $\com(t)$. 
We express a small perturbation of the spherical bubble as:
\begin{align}
\Phi(r,\Omega,t)\ &=\ {\rm constant}\ +\ \Psi(r,\Omega,t),\ \ \ \ \ r=|\vect{x}-\com(t)|,\label{Psi-def}\\
\vect{R}(t,\Omega)\ &=\ \bigl(\ 1+\beta(\Omega,t)\ \bigr)\ 
\frac{\vect{x}-\com(t)}{|\vect{x}-\com(t)|}.\label{beta-def}
\end{align}
In Appendix~\ref{app:derivations} we re-write the system~\eqref{eq:euler1}, relative to coordinates centered at $\com(t)$. 
In  the linear approximation, 
$\com(t)=\com(0)$\footnote{At nonlinear order, $\com(t)$ will typically evolve.} and the nondimensional system of equations, linearized  about the spherical equilibrium bubble are:
\begin{subequations} \label{eq:n3linear}
\begin{alignat}{2}
\epsilon^{2} {\partial_t}^2 \Psi - \Delta\Psi &= 0 &\qquad&  r > 1, \label{eq:lin3wave}\\
\Psi_r &= \beta_t &&  r=1, \label{eq:lin3kinematic}\\
\Psi_t &= 3\gamma \left( \frac{\Ca}{2} + \frac{2}{\We} \right) \left\langle\ \beta,Y_{0}^{0}\ \right\rangle\ Y_{0}^{0}
 - \frac{1}{\We} \left(2+ \Delta_S\right) \beta&& r=1, \label{eq:lin3bernoulli} \\
\left\langle\ \beta\ , Y_{1}^{m} \ \right\rangle_{L^2(S^2)} &= 0.&& |m|\le1\label{com-constraint}
\end{alignat}
\end{subequations}
Here, $\epsilon$ denotes the Mach number ($\Ma=\epsilon$\, is used in the derivation of the non-dimensional equations in Appendix~\ref{app:derivations}), $\Ca$, the Cavitation number, and $\We$, the Weber number.

We shall focus on the initial-boundary value problem for~\eqref{eq:n3linear} with data corresponding to an initial perturbation of only the bubble surface:
\begin{align}
&\Psi(r,\Omega,t=0)= \D_t\Psi(r,\Omega,t=0) = 0,\ \ \ \ \beta(t=0,\Omega)\ \text{ given and sufficiently smooth}.
\label{ib-data}
\end{align}
\noindent $\Delta_S$ denotes the Laplacian on $S^2$, in spherical coordinates given by~\eqref{Delta_S}. The spherical harmonics, $Y_l^m(\Omega)$,
are eigenfunctions: $-\Delta_S Y_l^m = l(l+1)Y_l^m, \ l\ge0,\ |m|\le l$, forming a complete orthonormal set in $L^2(S^2)$ with respect to the inner product $\langle \alpha_1,\alpha_2 \rangle_{L^2(S^2)} $; see Section~\ref{app:hankel}.
 The orthogonality conditions~\eqref{com-constraint}, derived in Appendix~\ref {app:derivations}, express our choice of coordinates (in the linearized approximation)  placing the origin at the bubble center of mass.

\subsection{Overview of results and discussion}\label{section:overview}
We conclude this section with an overview and discussion of results.
\begin{remunerate} 
%
\item {\bf Time-decay of solutions to wave equation on $\R^3\setminus\{|\bx|\le1\}$ with time-dependent Neumann data:}\ Theorem~\ref{theorem:NtD} is a general result, of independent interest, on the time-decay and resonance expansion of solutions to the initial-boundary value problem for the wave equation on $\R^3-S^2$. The data prescribed on $S^2$ are assumed sufficiently smooth and exponentially decaying with time. Theorem~\ref{theorem:NtD} generalizes the results of~\cite{Tokita:1972lr,Wilcox:1959qy}; see also~\cite{LP:89,TZ:00}.  
\item {\bf Exponential time-decay estimates for the bubble surface perturbation (Theorem~\ref{thm:linsol}):} 
Theorem~\ref{theorem:NtD} on the Neumann to Dirichlet map, together with the detailed information we obtain on the locations of scattering resonances in the lower half plane,  is used  to prove that the solution of the initial value problem~\eqref{eq:n3linear} with initial data~\eqref{ib-data}  decays exponentially to zero, pointwise, at a rate $\cO\!\left(e^{-\Gamma(\epsilon) t}\right) $, $\Gamma(\epsilon) >0$, as $t$ tends to infinity.  Moreover, the linearized velocity potential, $\Psi(r,\Omega,t)$ and bubble surface perturbation, $\beta(\Omega,t)$,  satisfy resonance expansions in terms of outwardly radiating states of the scattering resonance problem. The expansion converges  in $C^2(K\times\R_+)$, where $K$ denotes  any compact subset of $|\bx|\ge 1$. 
\item {\bf Scattering resonances and radiation damping:}  
The exponential rate of decay, $\Gamma(\epsilon)=|\Im\lambda^{\pm}_\star(\epsilon)|$, is determined via the {\it scattering resonance problem},~\eqref{eq:srp}, a non-selfadjoint spectral problem associated with the 
time-harmonic solutions of the linearized compressible Euler equations + boundary conditions on $|\bx|=1$, with {\it outgoing radiation conditions} at infinity. Two families of scattering resonances, associated with the Helmholtz equation 
 $\Delta \Psi+\omega^2\Psi=0$ for $|\bx|>1$,  appear in the linearized problem.\ 
{\it  Rigid resonances, $\{\omega_{l,k}(\epsilon)\}_{l\ge0, 1\le k\le l+1}$} (Theorem~\ref{thm:rigid-summary}), are associated with  Neumann ({\it sound soft}) boundary conditions imposed on the sphere, and outgoing radiation conditions at infinity and 
{\it deformation resonances, $\{\lambda_{l,k}(\epsilon)\}_{l\ge0, 1\le k\le l+2}$} (Theorem~\ref{thm:def-res-general}), associated with the  hydrodynamic boundary conditions on the sphere, responsible for the deformation of the bubble-fluid interface,  and outgoing radiation conditions at infinity.  Theorem~\ref{thm:big-l} implies that for each fixed $\epsilon>0$, there is a strip containing the real axis, in which there are no scattering resonances; hence $|\Im\lambda^{\pm}_\star(\epsilon)|>0$. 
\item{\bf Asymptotics of deformation resonances and Fermi's Golden Rule:}  In the incompressible limit, $\epsilon\downarrow0$, 
 all resonances (rigid and deformation) have imaginary parts which tend to minus infinity, except for the sub-family of deformation resonances called  {\it Rayleigh resonances}, $\{\lambda_l^\pm(\epsilon)\}_{l\ge0}$.  Rayleigh resonances are scattering resonances in the lower half plane to which the real Rayleigh eigenfrequencies (of undamped oscillations in the incompressible problem, $\epsilon=0$)  perturb upon inclusion of small compressibility, $\epsilon>0$. 
 Their detailed asymptotics for $\epsilon$ small is given  in Theorem~\ref{thm:2smallep}. In particular, we find for the\\ 
 {\it  Imaginary parts of the Rayleigh Deformation Resonances:}
\begin{align}
\Im \lambda_{l=0}^\pm(\epsilon)\ &=\ -\epsilon \left[\ \frac{1}{2} \left( \frac{3\gamma}{2}\Ca + 2(3\gamma-1)\frac{1}{\We} \right) \right],\nn\\
\Im \lambda_l^\pm(\epsilon)\ &=\  -\frac{1}{\epsilon} \left[  \frac{1}{2} \left[(l+2)(l-1) \right]^{l+1} (l+1)^{l} \left[\frac{2^{l}l!}{(2l)!}\right]^{2}  \left( \frac{\epsilon^{2}}{{\We}} \right)^{l+1} \!+ \bigO_{l}\!\left( \left( \frac{\epsilon^{2}}{\We}\right)^{l+2} \right) 
\right]
 \label{Im-part_l}
\end{align}
for $l=2,3,\dotsc.$
The proof of~\eqref{Im-part_l} relies on use of detailed properties of spherical Hankel function and, in particular, a subtle result on the Taylor expansion of the function 
$
 z\ \mapsto\ G_l(z)\ =\  z\,\D h_l^{(1)}(z)/h_l^{(1)}(z)
 $
  in a neighborhood of $z=0$; see Proposition~\ref{thm:imagterm}.
 \\ \\
The infinite set of pairs of  (real) Rayleigh eigenfrequencies $\{\lambda_l^\pm(0)\}_{l=0,2,\dots}$ may be viewed as: {\it embedded eigenvalues in the continuous spectrum of the unperturbed ($\epsilon=0$, incompressible) spectral problem}. The negative imaginary part in~\eqref{Im-part_l} is an instance of the {\it Fermi Golden Rule}. An expression coined originally in the context quantum electrodynamics~\cite{Cohen-Tannoudji}, it refers to the 
 induced damping of an ``excited state'' (the bubble perturbation), due to coupling of an ``atom'' (the deforming bubble) to a field (wave equation); see also, for example,~\cite{Weisskopf-Wigner:30,RS4,SW:98}.
\item {\bf Scaling behavior of the decay rate, $\abs{\Im \lambda^{\pm}_\star(\epsilon)}$  (Section~\ref{sec:longlivedres-thm}):}\  Theorem~\ref{thm:2smallep},  asymptotics and high-precision numerics show that the exponential {\it rate} of decay is given by 
$|\Im \lambda^{+}_\star(\epsilon)|=|\Im \lambda^{-}_\star(\epsilon)| >0$, where $\lambda^\pm_\star(\epsilon)$ are the {\it scattering resonance energies} closest to the real axis, and in the lower half plane. 
These resonances
$
\lambda^{\pm}_\star(\epsilon)\ \sim\ \lambda_{l_\star(\epsilon)}^\pm,\ 
 {\rm where}\ \ l_\star(\epsilon)=
 \cO\left( \epsilon^{-2}\We\right).
$
Moreover, an asymptotic study of the results of Theorem~\ref{thm:2smallep} yields:
\begin{equation}
\Re\lambda^{\pm}_\star(\epsilon) = \frac{1}{\epsilon}\,\cO\!\left( \frac{\We}{\epsilon^2}\right),\ \ \ \ \ \Im\lambda^{\pm}_\star(\epsilon)= \frac{1}{\epsilon}
\,\cO\!\left( \frac{\We}{\epsilon^2}\ e^{-\kappa
\frac{\We}{\epsilon^2}}\right),\ \ \kappa>0.
\nn\end{equation}

In contrast, the monopole (spherically symmetric) resonance has an imaginary part which is $\cO(\epsilon)$. 
This remarkably slow rate of decay is related to the scattering resonance problem being a singular perturbation problem in the  (incompressible) limit of small $\epsilon$; the wave equation $\epsilon^2\D_t^2\Phi=\Delta\Phi$ reduces to Laplace's equation $\Delta\Phi=0$ as $\epsilon\to0$. Figure~\ref{fig:resonances} displays, for a particular small choice of $\epsilon$, a range of resonance energies including those located closest to the real axis. Physically, these very slowly decaying bubble shape modes are associated with capillary waves on the bubble surface, are excited only by asymmetric perturbations, and very slowly transfer their energy to the infinite fluid. 
\footnote{We remark on an interesting example, where scattering resonances converge to the real axis. Stefanov and Vodev~\cite{Stefanov:1994fr} show, for the system of linear elasticity
 on the exterior of a ball  in $\R^3$, that 
 the scattering resonances converge to the real axis exponentially fast, as the spherical harmonic index, $l$, tends to infinity.  Consequently, the time-decay of solutions is not exponentially fast;
see~\cite{Stefanov:1994fr,Stefanov-Vodev:95,Stefanov-Vodev:96,Burq-Zworski:01,Stefanov:01} and references cited therein.
The modes associated with the resonances which come ever closer to the real axis are called
 {\it Rayleigh surface waves}, a decay rate limiting mechanism analogous to the capillary surface waves ({\it Rayleigh resonance modes}) of the bubble problem we consider in this article. }
\begin{figure}[!htb] 
\hspace{-1.5em}   \includegraphics[width=\textwidth]{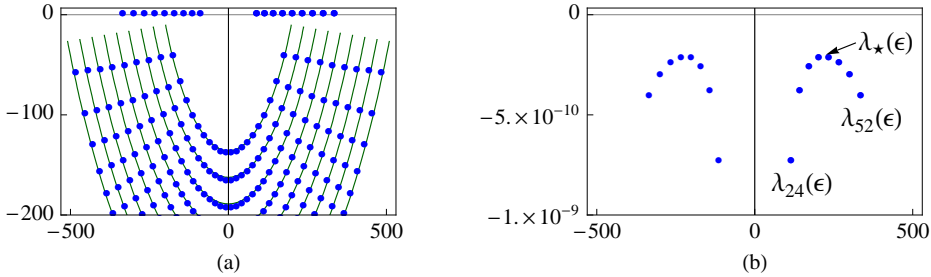}
   \caption{\small (a)\ Numerically computed scattering resonances associated with spherical harmonics $l=20,24,28,\dotsc,52$ for $\epsilon=0.1,\,\We=1$.
 {\it Rayleigh resonances} are those located just below the real axis. 
(b)  Vertical axis is scaled to show the resonances  closest to the real axis.}
   \label{fig:resonances}
\end{figure}
%
%
\item {\bf Monopole \emph{vs.}~Multipole radiation:}\ A discussion of emission of acoustic radiation in the physics literature (see~\cite{Longuet-Higgins:1989ux} and references cited therein) is based on a heuristic energy balance argument, which assumes that the dominant emission of acoustic radiation is through the monopole ($l=0$, purely radial) mode. This assumption implies decay of the bubble perturbation on the time scale of order 
$ \tau_{\rm monopole}(\epsilon) \sim \epsilon^{-1}$.
Our results demonstrate that there are indeed non-symmetric vibrational modes of a bubble, which have the much longer lifetime
 \begin{equation*}
 \tau(\epsilon) \sim \epsilon\ \smash{(\epsilon^{-2} \We)}\ \exp\left(\kappa\ \epsilon^{-2} \We\right) \gg \tau_{\rm monopole}(\epsilon),\ \ \kappa>0.
 \end{equation*}
  Of course, in many systems one would expect viscosity (not included in the modeling) to contribute a dominant correction to the imaginary part.

We note a relation of these results to those in~\cite{Lighthill:62} and~\cite{Moore-Spiegel:64}, who used energy flux arguments to study the radiation of acoustic energy in compressible atmospheres.
%
%
\end{remunerate}

\medskip 
\noindent{\bf Some future directions:} In addition to the question of nonlinear asymptotic stability of the family of spherical equilibrium bubbles, we remark that there is a rich class of solutions in the radially symmetric setting. In the case of radial symmetry, Rayleigh showed that the dynamics of a spherical bubble in an incompressible liquid exactly reduce to a nonlinear ordinary differential equation for the time-evolution of the radius. This equation admits the spherical equilibrium (constant radius) state as well as time-periodic (radially pulsating) states. Plesset~\cite{Plesset:1949vo} extended the analysis in the spherically symmetric case to include viscosity. 
It is of interest to study the dynamics near the spherically symmetric time-periodic states of the Rayleigh-Plesset equations.
We expect that these would be unstable; periodic oscillations would couple to continuous spectral modes, resulting in radiation-damped ``breather'' oscillations; see, for example,~\cite{SW:99} and references cited therein. The period map or monodromy operator associated with the linearization about such a state, for $\epsilon=0$, would have an embedded Floquet multiplier on the unit circle, which would perturb to an unstable  Floquet multiplier outside the unit circle~\cite{Sigal:93,MSW:00}.

Finally, we remark that the problem we consider is one of a very large class, involving an infinite dimensional conservative system comprised of two coupled subsystems: one subsystem has a discrete number of degrees of freedom (here the mode amplitudes of the spherical harmonics of the bubble surface) and one has a continuum of degrees of freedom, the velocity potential governed by the wave equation.  While there has been significant progress on such systems, where the number of ``soliton'' degrees of freedom of the discrete subsystem is finite (for example, see~\cite{SW:99,KKS:99,W:07} and references cited therein), systems like the bubble-fluid system, which involve the coupling of infinite dimensional systems, are central in physical and engineering science and are an important future direction.
\medskip

\noindent{\bf Acknowledgements:} The authors thank D. Attinger, J.B. Keller, A. Soffer,  E. Spiegel, J. Xu and M. Zworski for stimulating discussions. This research was supported in part by US National Science Foundation grant DMS-07-07850 and NSF grant DGE-0221041. The Ph.D. thesis research of the first author was supported, in part,  by the IGERT Joint Program in Applied
Mathematics and Earth and Environmental Science at Columbia University. Some of the discussion of this paper, appears in greater detail in~\cite{ASK-thesis:10}.

\subsection{Some oft used notation}
\begin{itemize}
\item $\langle \alpha ,\beta \rangle = \langle \alpha ,\beta \rangle_{L^2(S^2)}\ =\ $ inner product on $L^2(S^2)$;\quad  Section~\ref{app:hankel}.
\item $\Delta_S = \text{Laplacian on } S^2$;\quad Equation~(\ref {Delta_S}),\ \ $Y_{l}^{m}(\theta,\phi) = Y_{l}^{m}(\Omega)$, spherical harmonic;\quad Appendix~\ref {app:Ylm}.
\item Dimensionless quantities: 
	$\We=$ Weber number,  $\Ca=$ Cavitation number, $\epsilon = \Ma = $ Mach number; \quad Appendix~\ref {app:derivations} 
\end{itemize}
\beq \label{eq:rlhat-notation} 
\rlhat = \begin{cases} 
			\frac{3\gamma}{2}\Ca + 2(3\gamma-1)\frac{1}{\We}, & l=0 \\
			\frac{1}{\We}\ (l+2)(l-1), & l\ge 1,\\
		\end{cases} \mspace{300mu}
		\eeq
\section{Conservation of Energy} 
\label{sec:consenergy}
An important role in our analysis of time-decay of solutions of the initial-boundary value problem for~\eqref{eq:n3linear} is played by the following conservation law. We use it to establish the scattering resonance frequencies as lying in the lower half plane; see the proof of Proposition~\ref{thm:def-res-general}. 

\begin{proposition} \label{thm:energy}
Let $\left(\ \Psi(\bx,t),\ \beta(\Omega,t)\ \right)$ denote a smooth complex-valued solution of the linearized equations.
\begin{romannum}
\item Then if $\Psi$ decays sufficiently rapidly as $|x|\to\infty$, then the following functional is time-invariant
\begin{equation}
{\cal E} = \int_{\abs{x}\ge{1}} \hspace{-.68em}\left(\epsilon^{2} \abs{\D_t\Psi}^{2} + \abs{\grad\Psi}^{2}\right) \,dx + 3\gamma \left( \tfrac{\Ca}{2} + \tfrac{2}{\We} \right) \bigabs{\langle\beta,Y_0^0\rangle}^2+\tfrac{1}{\We}\langle (-\Delta_S-2)\beta,\beta\rangle,
\label{CofE}
\end{equation}
where $\langle\alpha,\beta\rangle$ denotes the inner product on $L^2(S^2)$.
\item
 The energy, ${\cal E}$,  can also be expressed as:
 \begin{multline}
 {\cal E} =
\int_{\abs{x}\ge{1}} \left(\epsilon^{2} \abs{\D_t\Psi}^{2} + \abs{\grad\Psi}^{2}\right) dx\ +\ \left( \tfrac{3\gamma}{2}\Ca + 2(3\gamma-1)\tfrac{1}{\We} \right)\;\bigabs{\langle\beta,Y_0^0\rangle}^2\ \\
+\
 \frac{1}{\We}\ \sum_{l\ge2}\sum_{|m|\le l} (l+2)(l-1) \; \bigabs{\langle\beta,Y_l^m\rangle}^2.
 \label{CofE-pos1}
 \end{multline}
Furthermore, 
${\cal E}$ is positive definite on solutions of the linearized equations~\eqref{eq:n3linear} since $\gamma>1$.
\end{romannum}
\end{proposition}
 
\begin{proof}
{\noindent \bf Part (i):}\ Multiplying  the wave equation~\eqref{eq:lin3wave} by 
$\overline{\D_t\Psi}$ and taking the real part of the resulting equation yields
\begin{equation}
\D_t e(\Psi,\D_t\Psi)\ +\ \div\left(\ -\overline{\D_t\Psi}\grad\Psi\ -\ \D_t\Psi\overline{\grad\Psi}\right)\ =\ 0,
\label{local-cons}
\end{equation} 
where
$e(\Psi,\D_t\Psi) = \epsilon^{2} \abs{\D_t\Psi}^{2} + \abs{\grad\Psi}^{2}.
$
Integration of the local conservation law~\eqref{local-cons} over the region $|x|\ge1$ yields
\begin{equation}
\frac{d}{dt}\ \int_{|x|\ge1} e(\Psi,\D_t\Psi)\ dx + \int_{|x|=1} 
\left(\ \overline{\D_t\Psi}\grad\Psi\cdot\unitn + \ 
 \D_t\Psi\overline{\grad\Psi}\cdot\unitn \right)\ dS = 0,
\label{cons1}
\end{equation}
where $\unitn$ denotes the unit normal to $S^2$, pointing into the ball. 
The theorem no follows from re-expressing the second term in~\eqref{cons1} using the boundary conditions:
\begin{multline}
\int_{|x|=1} 
\left(\, \overline{\D_t\Psi}\grad\Psi\cdot\unitn +\D_t\Psi\overline{\grad\Psi}\cdot\unitn\, \right)\, dS\\
= \frac{d}{dt}\left[ 3\gamma\left( \tfrac{\Ca}{2} + \tfrac{2}{\We} \right) \bigabs{\langle\beta,Y_0^0\rangle}^2 + \tfrac{1}{\We}\, \langle (-\Delta_S-2)\beta,\beta\rangle \right]
\end{multline}

\noindent{\bf Part (ii):} We show now, using the 
center of mass constraint,
\eqref{com-constraint}, that the energy is positive definite on solutions of the linearized system,~\eqref{eq:n3linear}.

By~\eqref{com-constraint} we can expand $\beta$ in spherical harmonics as
$\beta \,=\, \langle \beta,Y_0^0\rangle\, Y_0^0\ +\ \beta_{l\ge2}$, where
$\beta_{l\ge2}\ =\ \sum_{l\ge2}\sum_{|m|\le l}\, \langle \beta,Y_l^m\rangle\ Y_l^m.
$
Consider the expression $\langle (-\Delta_S-2)\beta,\beta\rangle$. We have, by orthogonality
\begin{align}
\langle (-\Delta_S-2)\beta,\beta\rangle &= -2\left| \bigave{\beta,Y_0^0} \right|^2+ \left\langle (-\Delta_S-2)\beta_{l\ge2},\beta_{l\ge2}\right\rangle\nonumber\\
&= -2\left| \bigave{\beta,Y_0^0}\right|^2+  \sum_{l\ge2}\sum_{|m|\le l} (l+2)(l-1) \left| \langle\beta,Y_l^m\rangle\right|^2
\label{DeltaSm2}
\end{align}
Use of~\eqref{DeltaSm2} in the expression for the energy~\eqref{CofE} yields~\eqref{CofE-pos1}, 
which is clearly positive definite.
\qquad\end{proof}

\section{The Scattering Resonance Problem}
Our results on time-decay of the velocity potential and bubble surface perturbations 
(Theorems~\ref{theorem:NtD} and~\ref{thm:linsol}) are proved using their representation as 
 inverse Laplace transforms of a resolvent operator applied to the initial data; see Equations~\eqref{eq:ulm} and~\eqref{eq:betauhp}, 
  in which the contour of integration is a horizontal line in the upper half of the complex frequency plane. Time decay 
   (and a resonance mode expansion) is deduced by analytically continuing the resolvent kernel (Green's function)  and  deforming the integration contour across the real axis (essential spectrum) into the lower half plane. The time-decay is determined  by poles of this analytically continued resolvent kernel. These poles are called {\it scattering resonances} or  {\it scattering poles}. An alternative characterization of scattering resonances is as  complex frequencies in the lower half plane for which there are non-trivial time-harmonic solutions of the linearized equations,  satisfying an {\it outgoing radiation} (non-self-adjoint) boundary condition at infinity.

 In this section we use the alternative (spectral) characterization as solutions of a non-selfadjoint eigenvalue problem.
We consider two classes of scattering resonances:
\begin{remunerate}
\item {\it Rigid resonances}, associated with the wave equation
on the exterior of the (rigid) unit sphere with Neumann boundary conditions, and 
\item {\it Deformation resonances}, associated with the linearized
 system~\eqref{eq:n3linear}.
\end{remunerate}

Both families of resonances, for $\epsilon>0$, lie in the lower half plane. Of particular interest is the resonance in each family with the smallest (in magnitude) imaginary part. Denote by $\omega_\star(\epsilon)$ and  $\lambda_\star(\epsilon)$\footnote{Here, and in later sections, $\lambda^{\pm}_\star(\epsilon)$ may be denoted simply $\lambda_\star(\epsilon)$.} the rigid and deformation resonances of smallest imaginary parts. We find for $\epsilon$ small:
\begin{equation}
\left| \Im\omega_\star(\epsilon) \right| = 
{\mathcal O}\left(\epsilon^{-1}\right),\
\left| \Im\lambda_\star(\epsilon) \right| = 
{\cal O}\bigl(\epsilon^{-3}e^{-\frac{\kappa}{\epsilon^{2}}}\bigr),\ \kappa>0\,\,\,
\Rightarrow\,\,\left| \Im\lambda_\star(\epsilon) \right|  \ll \left| \Im\omega_\star(\epsilon) \right|.
\end{equation}
\subsection{Rigid\ -\ Neumann scattering resonances}
We consider the scattering resonance problem, associated with  the wave equation $\epsilon^2\D_t^2\Psi=\Delta\Psi$ in the exterior of the unit sphere with Neumann boundary conditions. Seeking time-harmonic solutions, $\Psi(r,\Omega,t)=e^{-i\omega t}\Psi_\omega$, which are outgoing as $r\to\infty$ , we arrive  the following eigenvalue problem
 \begin{align}
 \left(\ \Delta\ +\ (\epsilon\omega)^2\ \right)\Psi_\omega(r,\Omega)\ &=0,\ \ \ \  r>1;\ \ \ 
 \D_r\Psi_\omega(r,\Omega)\ =0,\ \ \ \ r=1.
 \label{srp-rigid}
 \end{align}
Outgoing solutions are spanned by functions of the form
 $
 h_l^{(1)}(\epsilon\omega r)\ Y_l^m(\Omega).
 $
 Thus, $\omega$ is a Neumann scattering resonance if and only if
$\D h_l(\epsilon\omega)\ =\ 0$\ . The following result summarizes results in~\cite{Tokita:1972lr,Olver:1954lr}; see also~\cite{Abramowitz:1965zr}.

\begin{theorem}\label{thm:rigid-summary}
Fix an arbitrary $\epsilon>0$ and arbitrary.
\begin{romannum}
	\item For each $l\ge0$, the equation $\D h_{l}^{(1)}(\epsilon\omega)=0$ has a family of solutions
$
		\{\omega_{l,k}(\epsilon)=
		{\epsilon}^{-1}\omega_{l,k}:\ \ k=1,\dots,l+1\}.$
	\item The set of all Neumann scattering resonances is a discrete subset of the lower half complex plane and is uniformly bounded away from the real axis. In particular,  there exists $l_*\ge0, 1\le k_*\le l_*+1$ such that $\omega_*(\epsilon)=_{\rm def} \omega_{l_*,k_*}(\epsilon)$ is a resonance whose imaginary part is of minimal magnitude, {\it i.e.} 
		\begin{equation}
		\Im \omega_{l,k}(\epsilon)\le \Im\omega_*(\epsilon) < 0,\ \ \ \text{for all}\ l\ge0,\ \ k=1,\dots,l+1
		\label{im-omega-lk-lt0}
		\end{equation}
	\item There exist constants $C_1, C_2>0$, such that for all 
	$l\ge 0$
		\begin{align*}
		C_1\ \epsilon^{-1}\ l^{1/3} \le \bigabs{\Im \omega_{l,k}(\epsilon)},\  \ l\gg 1\qquad \textrm{and}\qquad    
		\left| \omega_{l,k}(\epsilon) \right|\ =\ \epsilon^{-1} |\omega_{l,k}|\ \le\ C_2 \epsilon^{-1}\ l\ .
		\end{align*}
\end{romannum}
\end{theorem}
This result is essentially due to Tokita~\cite{Tokita:1972lr} and uses fundamental results on the asymptotics of Hankel and Airy functions; see  Olver~\cite{Olver:1954lr,Abramowitz:1965zr}. Explicit approximations to zeros of $\D H_{\nu}$ are given in Equation~(\ref {eq:dH-zero}).


\subsection{Deformation resonances}\label{section:deformation}
We seek time-harmonic solutions of the linearized perturbation equations~\eqref{eq:n3linear}\ :
$
\Psi = e^{-i\lambda t}\Psi_{\lambda}(r,\Omega),\ \beta = e^{-i\lambda t} \beta_{\lambda}(\Omega).$
%
Substitution into~\eqref{eq:n3linear} yields the following Helmholtz  eigenvalue problem:
\begin{subequations}\label{eq:srp}
\begin{align}
\left(\ \Delta +\left(\epsilon \lambda\right)^2\ \right) \Psi_{\lambda} &= 0, && r>1\\
\partial_{r}\Psi_{\lambda} &= -i\lambda \beta_{\lambda}, && r=1\\
-i\lambda \Psi_{\lambda} &= 3\gamma\left( \tfrac{\Ca}{2} + \tfrac{2}{\We} \right) \ 
\langle \beta_{\lambda},Y_0^0 \rangle Y_0^0 -
 \tfrac{1}{\We}\ (2+\Delta_{S})\beta_{\lambda}\ , && r=1\\
\Psi_{\lambda} &\phantom{=} \text{outgoing} && r\to\infty.
\end{align}
\end{subequations}
If $\lambda$ is such that~\eqref{eq:srp} has a non-trivial solution,
then we call $\lambda$ a (deformation) scattering resonance energy or scattering frequency, and $(\Psi_\lambda,\beta_\lambda)$ a corresponding scattering resonance mode.

Since outgoing solutions of the three-dimensional Helmholtz equation are linear combinations of solutions of the form $h_l^{(1)}(r) Y_l^m(\Omega)$, $\ |m|\le l$, where $h_l^{(1)}$ denotes the {\it outgoing} spherical Hankel function of order $l$,  we seek solutions of the boundary value problem~\eqref{eq:srp} of the form:
$
\Psi_{\lambda}(r,\Omega) = A\ Y_l^m(\Omega)\ h_{l}^{(1)}\left(\epsilon \lambda r\right)$, $\ \beta_{\lambda}(\Omega) = B\ Y_l^m(\Omega)$, $\ r\ge1$, $\ \Omega\in S^2$,
%
where $A$ and $B$ are constants to be determined.
This Ansatz  automatically solves the Helmholtz equation and satisfies
the outgoing radiation condition. To impose the boundary conditions at $r=1$ we 
substitute we substitute the expressions for $\Psi_\lambda$ and $\beta_\lambda$ into~\eqref{eq:srp} and obtain the following two linear homogeneous equations for the constants $A$ and $B$.
This system has a non-trivial solution if and only if $\lambda$ solves the equation:
\begin{equation}
\rlhat\ \epsilon\ \lambda \D h_{l}^{(1)}(\epsilon\lambda) + \lambda^{2}  h_{l}^{(1)}\left(\epsilon\lambda\right) = 0,
\label{srp-en-eqn}
\end{equation} 
where $\rlhat$ is given by~\eqref{eq:rlhat-notation}.
We shall be interested in the character of solutions to~\eqref{srp-en-eqn} for small, positive and fixed $\epsilon$. By~\eqref{eq:useful-limit}
\begin{equation}
\lim_{z\to0} \frac{z\D h_l^{(1)}(z)}{h_l^{(1)}(z)}\ =\ \lim_{z\to0}\  
 \frac{p_{l} ( z)}{r^{l+1}\left[l p_{l}(z) - p_{l+1}( z)\right]}\ =\ 
  -(l+1)\ .  \label{hankrat-lim}
\end{equation}
Thus,  we write~\eqref{srp-en-eqn} in the form:
\begin{equation}
\lambda^2 \ +\  \rlhat\ G_l(\epsilon\lambda)\ =\ 0,\ \ \  G_l(z)\ \equiv\ 
\frac{z\D h_l^{(1)}(z)}{h_l^{(1)}(z)} .
 \label{srp-en-eqn-A}
 \end{equation}
we call the solutions of~\eqref{srp-en-eqn-A} {\it deformation resonances}. Their properties are summarized in the following result, proved in Section~\ref{sec:resonances}.
\begin{theorem}  [\textrm{Deformation resonances}] \label{thm:deformation-summary}
 Fix $\epsilon>0$ and arbitrary. 
\begin{romannum}
\item There are $l+2$ solutions of Equation~\eqref{srp-en-eqn-A}  denoted $\{\lambda_{l,j}(\epsilon)\}$, for $l=0,2,3,\dotsc$ and $j=1,\dotsc,l+2$. These are the  deformation resonance energies.
\item The set of deformation resonance energies is a discrete subset of the lower half complex plane and is uniformly bounded away from the real axis. That is, for some $l_{\star}(\epsilon), k_{\star}(\epsilon)$, 
\begin{equation}
\Im\lambda_{l,j}(\epsilon) \le\ \Im\lambda_{{l_\star},{j_\star}}(\epsilon)\ \equiv\  \Im\lambda_\star(\epsilon) < 0,\ \ \ {\rm all}\ \  l\ge0,\ |j|\le l+2.
\end{equation}
\item $\lambda_{l,j}(\epsilon)= {\cal O}( l)$ as $l\to\infty$.
\item In the incompressible limit,  $\epsilon\to0^{+}$, the imaginary parts of all resonances tend to $-\infty$ \underline{except} for the family of 
\underline{Rayleigh resonances} with real frequencies:
\begin{align}
\lambda^\pm_{0}(0) &= \pm\sqrt{ \tfrac{3\gamma}{2}\Ca + \tfrac{2}{\We} (3\gamma-1)},\;\;
\lambda^\pm_{l}(0) = \pm\sqrt{\tfrac{1}{\We} (l+2)(l+1)(l-1)},\; \;l\ge2.  \label{leps0-res}
\end{align}
\end{romannum}
\end{theorem}
%
%
\section{A theorem on resonance expansions and time-decay for the exterior Neumann problem for the wave equation} \label{sec:extneumann}
Our strategy for solving the initial-boundary value problem~\eqref{eq:n3linear},~\eqref{ib-data} is to
(1) construct the solution to the wave equation~\eqref{eq:lin3wave} with kinematic boundary condition~\eqref{eq:lin3kinematic}, which specifies Neumann boundary data on the unit sphere. This is the {\it Neumann to Dirichlet map}\  $\D_t\beta\mapsto\Psi={\rm NtD}\left[\D_t\beta\right]$. Then,  (2) substitute $\Psi=\D_t\beta\to{\rm NtD}\left[\D_t\beta\right]$ into the dynamic boundary condition~\eqref{eq:lin3bernoulli} and then study the closed nonlocal equation for $\beta$ on the unit sphere.

Denote by $U$ the exterior of the unit sphere in $\R^{3}$:
\begin{equation}
U\ =\ \{\bx:\ |\bx|>1\}.
\nn\end{equation}
 In particular, we consider the general initial-boundary value problem
%
\begin{subequations} \label{eq:prob}
\begin{alignat}{2} 
\left( c^{-2} \D_t^2 - \Delta\right) u &= 0 , &\qquad& \text{in } U \times (0,\infty), \label{eq:uwave} \\
u &= 0,\ \  \dt u = 0, && \text{on } U \times \set{t=0}, \\
\dr u &= \dt f, && \text{on } \partial U\times (0,\infty) \label{eq:uBC}
\end{alignat}
\end{subequations}
where  the function $\dt f: \partial U \times (0,\infty) \to \R$  is a prescribed boundary forcing function, whose properties (anticipating those of the bubble shape perturbation, $\D_t\beta$) are prescribed below. 

We introduce a norm of functions, $\D_tf=\D_tf(\Omega,t)$, defined on the sphere, $S^2$,  which encodes smoothness in $\Omega\in S^2$ and decay in time, $t$.

\begin{definition}
\begin{romannum}
\item For a function $J:[0,\infty)\to \R$, define for any $\kp\in\R$.
\begin{equation}
[J]_{\kp}\ =\ \sup_{t\ge0} e^{\kp t}|J(t)|.
\label{inf-eta-norm}
\end{equation}
\item For a function $g:\R_t\times S^2_\Omega\to\R$, with 
spherical harmonic expansion:
$g(t,\Omega) = \sum_{l\ge0}\sum_{|m|\le l} g_l^m(t)Y_l^m(\Omega)$,
define, for $q, \kp\in\R$, the norm
\begin{align}
[g]_{q,\kp} &=   \sum_{l\ge0}\sum_{|m|\le l} 
(1+l)^q\ \sup_{t\ge0}\  e^{\kp t}\left(\ |g^m_l(t)|+|\D_t g^m_l(t)|\ \right)\ \nonumber \\
&= \sum_{l\ge0}\sum_{|m|\le l} (1+l)^q\ \left(\ [g^m_l ]_{\kp}+[\D_t g^m_l ]_{\kp}\ \right).
\label{norm-q}
\end{align}
\end{romannum}
\end{definition}

\noindent The  main result of this section is the following theorem, a generalization~\cite{Tokita:1972lr,Wilcox:1959qy}.

\begin{theorem}\label{theorem:NtD}
Consider the Neumann initial-boundary value problem for the wave equation~\eqref{eq:prob} on the exterior of the unit sphere. 
\begin{romannum}
\item  \emph{Time Decay Estimate:}\ There exists $\alpha>0$,  such that if $[\D_tf]_{\alpha,\eta}<\infty$, with $\eta>0$,  
 then there exists a unique solution of the initial-boundary value problem equations~(\ref {eq:prob}) satisfying the bound
\beq 
|u(x,t)| \le  
\begin{cases} C\ \frac{1}{|x|} e^{-\min \{\eta, |\Im\omega_*|\} (t-\frac{|x|-1}{c}) }\  [\D_tf]_{\alpha,\eta} , \qquad  &1<|x|<1+ct\\
0,\ \ &|x|\ge1+ct
\end{cases}
\label{eq:decay-est1}\eeq
Here, $\omega_*\ (\Im\omega_*<0)$ denotes the {\it scattering resonance} of the exterior wave equation with Neumann boundary condition, which is closest to the real axis; see Theorem~\ref{thm:rigid-summary}.
 \item \emph{Resonance Expansion:}
Let $\alpha$ be as above.  Suppose $f:S^2\times\R_+\to\C$ has the  spherical harmonic expansion in the sense that 
\begin{align}
 & \lim_{L\to\infty} [ \D_t f - \D_t f^{(L)} ]_{\alpha,\eta}\ =\ 0,\ \  {\rm where}\ \ 
 f^{(L)}(\Omega,t) \equiv \sum_{l=0}^{L} \sum_{|m|\le l} f_{l}^{m}(t) Y_{l}^{m}(\Omega).  \label{eq:fsum}\end{align}
Denote by $u^{(L)}(r,\Omega,t)$ the resonance expansion partial sum:
 \beq    \label{eq:uresexp}
u^{(L)} = i\sum_{l=0}^{L} \sum_{|m|\le l}  \left[ \sum_{k=1}^{l+1}   \int\limits_{0}^{t-(r-1)/c} \!\!\!\!\!e^{-i \omega_{l}^{k}(t-s)} \partial_{s} f_{l}^{m}(s)\, ds  \ \frac{ h_{l}^{(1)}\left(\omega_{l,k} r/c\right)}{(\omega_{l,k}/c^2)\ \D^2 h_{l}^{(1)}\left(\omega_{l,k}/c\right)}  \right] Y_{l}^{m}(\Omega) .
\eeq 
Here $h_{l}^{(1)}(z)$ is the outgoing spherical Hankel function of order $l$ and 
\begin{equation}
\{\omega_{l}^{k}\}\ {\rm are\ the\ solutions\ of\ } \D h_{l}^{(1)}(\omega/c)=0;\nonumber
\end{equation}
see Theorem~\ref{thm:rigid-summary}. 
The limit 
$
u(x,t)=\lim_{L\to\infty} u^{(L)}(x,t)
$ exists and converges to the unique of solution of the initial-boundary value problem~\eqref{eq:prob}. Furthermore, we have the following 
error estimate for the resonance expansion. For  $0\le |a|+|b| \le 2$:
\begin{multline}   \label{eq:uerror}
\left|\ \D_t^a \D_x^b \bigl(u-u^{(L)} \bigr)(x,t)\ \right|  \\
\hspace{2em}\le 
\begin{cases} 
C\, [\D_tf -\D_t f^{(L)}]_{\alpha,\eta}   \,
\frac{1}{|x|} \, e^{-\min\{\eta,|\Im\omega_*|\}(t-\frac{|x|-1}{c})}, \ &1<|x|<1+ct\\
0,\ &|x|\ge1+ct.
\end{cases}
\end{multline}
%
\end{romannum}
\end{theorem}
\subsection{Proof of Theorem~\ref{theorem:NtD} }
We first note that the decay estimate of Part 1 is a consequence of the  resonance expansion of Part 2. Indeed, suppose Part 2 holds. Then, to prove the decay estimate~\eqref{eq:decay-est1},   note that $|u(x,t)| \le |u^{(L)}(x,t)| + |u(x,t)-u^{(L)}(x,t)|$. 
The first term is a finite sum, each of whose terms satisfies the decay estimate, while the second is controlled by~\eqref{eq:uerror}.

We now turn to the proof of the resonance expansion, Part 2 of Theorem~\ref{theorem:NtD}.  We first 
derive and, by Laplace transform methods,  solve the equations for  the spherical harmonic coefficients $u_{l}^{m}(r,t)$; see 
Equation~(\ref {eq:ulmrt}).  Then, we  form the partial sum for $u^{(L)}(r,\Omega,t)$ 
and prove  convergence.
In proving Theorem~\ref{theorem:NtD} we use the following two   estimates:
\begin{proposition}\label{proposition:hankrat-bound}
Fix $r>1$ and $l\ge0$. Then we have the bound
\begin{align}
Q_l(\omega,r)\ &\equiv\ \frac{h_l^{(1)}(\omega r/c)}{(\omega/c)  \ \D h_{l}^{(1)}(\omega/c)} e^{-i\omega(r-1)/c}\label{eq:hankrat}\\
&=\ \frac{1}{\omega r/c}
\left(\ 1+ {\cal O}_{l,r}\left(\frac{1}{|\omega|}\right)\ \right) \quad \text{as } \abs{\omega}\to\infty.
\label{Q-asymp}
\end{align}
\end{proposition}
\begin{proof}
We use facts about $h_l(z)$, documented in Appendix~\ref{app:polyhank}. Note that  
 $h_l^{(1)}(z) = z^{-l-1} p_l(z) e^{iz}$, 
where $p_l(z)$ is a polynomial of degree $l$. Furthermore, we have the following recursion relating $h_l$ and $\D h_l$ derivatives 
\beq
\D h_l^{(1)}(z) = z^{-1} l h_l^{(1)}(z) - h_{l+1}^{(1)}(z); 
\eeq
 see  Appendix~\ref {app:derhank}. Therefore,
 \begin{equation}
 Q_l(\omega,r)\ =\  \frac{p_{l} (\omega r/c)}{r^{l+1}\left[l p_{l}(\omega/c) - p_{l+1}(\omega/c)\right]}\ .
 \nonumber\end{equation}
 Now fix $r>1$ and $l\ge0$. The asymptotics~\eqref{Q-asymp}  follows using that $p_l(z)\sim a_l^l z^l$ for large $|z|$ and that
$|a_l^l/a_{l+1}^{l+1}|=1$.
\qquad\end{proof}

 \begin{proposition}\label{int-bound-lk} 
 Let $0<\ \eta\ <|\Im\omega_*|\ =_{\rm def} \ \min_{l,k} |\Im\omega_{l,k}|$. Then, 
 \begin{equation}
 \int_0^{t-\frac{r-1}{c}} e^{\Im \omega_{l,k} (t-s)} e^{-\eta s}\, ds\ 
e^{-\Im \omega_{l,k}\frac{r-1}{c}}\ \le\ \frac{C}{|\Im\omega_*|-\eta}\ 
e^{-\eta(t-\frac{r-1}{c})},\ \ \ r<1+ct.
\nonumber\end{equation}
 \end{proposition}

We now embark on the proof of Theorem~\ref{theorem:NtD}.
 Substitution of the expansion 
\begin{align}
u(r,\Omega,t) &= \sum_{l=0}^{\infty} \sum_{|m|\le l} u_{l}^{m}(r,t) Y_{l}^{m}(\Omega) \label{eq:usum}
\end{align} 
into  equation~(\ref {eq:uwave}), expressing the Laplacian as $ \Delta_r+r^{-2}\Delta_S$ (radial and spherical parts) and using that $-\Delta_{S} Y_{l}^{m}= l(l+1) Y_{l}^{m}$, we obtain the following equations for $u_{l}^{m}(r,t)$:
\begin{equation}
\left(\ c^{-2} \D_t^2  - \Delta_{r} + \frac{l(l+1)}{r^{2}}\ \right)\  u_{l}^{m}(r,t)  \ =\ 0.\label{eq:urtpde}
\end{equation}
From the boundary conditions equation~(\ref {eq:uBC}) we have
\begin{equation}
\D_r {u_{l}^{m}} (r,t) = {\D_t f_{l}^{m}}(t), \quad \text{at } r=1 \label{eq:urtBC}
\end{equation}
In terms of the time Laplace transform of $u$, $
\widetilde{u_{l}^{m}} (r,p) = \int_{0}^{\infty} e^{-pt} u_{l}^{m}(r,t) \, dt.
$, the initial-boundary value problem~\eqref{eq:urtpde},~\eqref{eq:urtBC}, with vanishing initial conditions for $r>1$ becomes
\begin{align}
&\left[\frac{1}{r^{2}}\dr \left(r^{2} \dr\right) + \left(\frac{ip}{c}\right)^{2}-\frac{l(l+1)}{r^{2}} \right] \widetilde{u_{l}^{m}} (r,p) = 0
\label{sbess}\\
&\D_r \widetilde{u_{l}^{m}} (r,p) = \widetilde{\dt f_{l}^{m}}(p), \quad \text{at } r=1,\qquad\ \ \
 \widetilde{u_{l}^{m}}(r,p)\ \to\ 0\ \ {\rm as}\ \ r\to\infty.\label{ulm-bc}
\end{align}
Equation~\eqref{sbess} is solved in terms of the {\it outgoing} spherical hankel function $h_l^{(1)}$; see Appendix~\ref{app:hankel}:
\begin{equation}
\widetilde{u_{l}^{m}} (r,p) = A_{l}^{m} (p) h_{l}^{(1)}(ipr/c), \label{eq:utilde}
\end{equation}
where $A_l^m(p)$ is to be determined. Imposing the boundary condition~\eqref{ulm-bc} yields
$
A_{l}^{m}(p) = \frac{\widetilde{\dt f_{l}^{m}}(p)}{(ip/c)\ \D h_{l}^{(1)}(ip/c)}.
$
Inversion of the Laplace transform yields:
\begin{equation}
u_{l}^{m} (r,t) = \frac{1}{2\pi i} \int_{\mu-i\infty}^{\mu+i\infty} e^{pt} \frac{\widetilde{\dt f_{l}^{m}}(p) h_{l}^{(1)}(ipr/c)}{(ip/c)\ \D h_{l}^{(1)}(ip/c)} \,dp\ ,
\end{equation}
where the contour is chosen in the right half plane, with $\Re p=\mu>0$ sufficiently large so that all poles lie to left.
\medskip

We find it convenient to work with a rotated (horizontal) contour and 
therefore make  the change of variable
\begin{equation} \omega = i p\ .
\label{omegaip}
\end{equation}
Then, 
\begin{equation}\label{eq:ulm}
\begin{split}
u_{l}^{m} (r,t) &= \frac{1}{2\pi} \int_{-\infty+i\mu}^{\infty+i \mu} e^{-i\omega t}\ \frac{\widetilde{\dt f_{l}^{m}}(-i\omega) h_{l}^{(1)}(\omega r/c)}{(\omega/c)\ \D h_{l}^{(1)}(\omega/c)} \,d\omega \\ 
&=\ \frac{1}{2\pi} \int_{-\infty+i\mu}^{\infty+i \mu} e^{-i\omega\left(t-(r-1)/c\right)} Q_{l}(\omega,r)\ d\omega,\quad \mu>0,
\end{split}
\end{equation}
where $Q_l(\omega,r)$ is defined in~\eqref{eq:hankrat}. 
\medskip

We first consider, $r=|x|$ fixed with $r>1+ct$ and prove the finite propagation speed result:
\begin{equation}
u_{l}^{m} (r,t)\ =\ 0,\ \ \ r>1+ct.
\label{ulm0-rlarge}
\end{equation}

\emph{Proof of Equation~\eqref{ulm0-rlarge}}.  By Theorem~\ref{thm:rigid-summary} the poles of $Q_{l}(\omega,r)$ 
are in the lower half $\omega-$ plane. Therefore, for any $\rho>0$
\begin{equation}
\frac{1}{2\pi}\left(\ \int_{-\rho+i\mu}^{\rho+i\mu}\ 
\ +\  
 \int_{{\cal C}_\rho}\ \right)\  e^{-i\omega\left(t-(r-1)/c\right)}\ Q_{l}(\omega,r) d\omega\ =\ 0, \ \ \mu>0, 
 \label{intC0}
 \end{equation}
 where $\mathcal{C}_\rho=\set{\ \omega : \abs{\omega-i\mu}=\rho,\,\Im \omega\ >\mu\ }$.
 It follows from~\eqref{intC0} that for $r>1+ct$
\begin{equation} \label{rho-limit}
\begin{split}
u_{l}^{m} (r,t) &=\ \lim_{\rho\to\infty} \frac{1}{2\pi} \int_{-\rho+i\mu}^{\rho+i\mu}  e^{i\omega\left((r-1)/c\ -\ t\right)} Q_{l}(\omega,r)\, d\omega,\\ 
&=\ -\lim_{\rho\to\infty} \frac{1}{2\pi}
 \int_{\mathcal{C}_\rho}  e^{i\omega ((r-1)/c\ -\ t)} Q_{l}(\omega,r)\, d\omega
\end{split}
\end{equation}
Note that the integrand is analytic in the upper half plane, continuous
up to the real line and, by Proposition~\ref{proposition:hankrat-bound}, 
 $\max_{\theta\in[0,\pi]}\ |Q_l(\rho e^{i\theta},r)|\to0$ as $\rho\to\infty$.  
By Jordan's Lemma the limit in~\eqref{rho-limit} vanishes, proving~\eqref{ulm0-rlarge}.
\qquad\endproof

\medskip

For $r\le 1+ct$, we decompose the Laplace transform of the initial data,
  $\widetilde{\dt f_{l}^{m}}$:
\begin{align}
\widetilde{\dt f_{l}^{m}}(-i\omega) &=\
\underbrace{
\int_{0}^{t-(r-1)/c} e^{i\omega s} \ds f_{l}^{m}(s) \,ds
}_{\cF_{l,m}^{(1)}(\omega;t)}
 +
\underbrace{
 \int_{t-(r-1)/c}^{\infty} e^{i\omega s} \ds f_{l}^{m}(s) \,ds
 }_{\cF_{l,m}^{(2)}(\omega;t)} \label{LTdata}
\end{align}
and the corresponding solution
\begin{equation}
 \label{eq:ulmrt}
 \begin{split}
u_{l}^{m} (r,t)\ &= 
 \underbrace{
 \frac{1}{2\pi} \int_{-\infty+i\mu}^{\infty+i \mu} e^{-i\omega t} \cF_{l,m}^{(1)}(\omega;t)\ \frac{h_{l}^{(1)}(\omega r/c)}{(\omega/c)\  \D h_{l}^{(1)}(\omega/c)} \,d\omega
 }_{ \equiv\ U_{l,m}^{(1)}(r,t) } \\
  &\quad+ 
 \underbrace{
  \frac{1}{2\pi} \int_{-\infty+i\mu}^{\infty+i \mu} e^{-i\omega t} \cF_{l,m}^{(2)}(\omega;t)\  \frac{h_{l}^{(1)}(\omega r/c)}{(\omega/c)\ \D h_{l}^{(1)}(\omega/c)} \,d\omega
 }_{\equiv\ U_{l,m}^{(2)}(r,t) }.
\end{split}
\end{equation}
%
In the following two propositions, we evaluate $U_{l,m}^{(1)}$ and $U_{l,m}^{(2)}$.
\begin{proposition}\label{proposition:Ulm1}
Let $\omega_{l,k}=\omega_{l,k}\left(c^{-1}\right),\  k=1,2,\dots, l+1$ denote the solutions of $\D h_l^{(1)}(\omega/c)=0$; see Theorem~\ref{thm:rigid-summary} with $\epsilon=c^{-1}$
 \footnote{The notation of Theorem~\ref{thm:rigid-summary} suggests we label the solutions of $\D h_l^{(1)}(\omega/c)=0$ as $\omega_{l,k}\left(c^{-1}\right)$.
For ease of presentation, in this self-contained section, we mildly abuse notation here and refer to these solutions as simply $\omega_{l,k}$.
 }.
For $r<1+ct$, $U_{l,m}$ can be expressed as a sum over residues:
\begin{align}
 U_{l,m}^{(1)}(r,t)\ &=\  \sum_{k=1}^{l+1}\ \int_{0}^{t-(r-1)/c} i e^{-i \omega_{l}^{k}(t-s)} \partial_{s} f_{l}^{m}(s)\, ds   \Res_{\omega=\omega_{l}^{k}}    \frac{h_{l}^{(1)}\left(\omega r/c\right)}{(\omega/c)\ \D h_{l}^{(1)}\left(\omega/c\right)},\nonumber\\
&=\  \sum_{k=1}^{l+1}\ \int_{0}^{t-(r-1)/c} i e^{-i \omega_{l}^{k}(t-s)} \partial_{s} f_{l}^{m}(s)\, ds\   
 \frac{h_{l}^{(1)}\left(\omega_{l,k} r/c\right)}{(\omega_{l,k}/c^2)\ \D^2 h_{l}^{(1)}\left(\omega_{l,k}/c\right)}.
 \label{U1m1}
 \end{align}
\end{proposition}
\begin{proposition}\label{Ulm2}
  $U_{l,m}^{(2)}(r,t)\ \equiv\ 0$.
\end{proposition}
%
%

\emph{Proof of Proposition~\ref{proposition:Ulm1}}.\ 
 We evaluate the contour integral representation of $U_{l,m}^{(1)}$ by the method of residues. 
  Consider the clockwise-traversed rectangular contour $\Gamma^{M,\mu,\gamma}=\Gamma_{1}+\Gamma_{2}+\Gamma_{3}+\Gamma_{4}$, where
\begin{align*}
\Gamma_{1} &= \set{x+i\mu : -M \le x \le M},\ \ 
\Gamma_{2} = \set{M-iy : -\mu \le y \le \gamma}  \\
\Gamma_{3} &= \set{-x-i\gamma : -M \le x \le M},\ \ 
\Gamma_{4} = \set{-M+iy : -\gamma \le y \le \mu}.
\end{align*}
By Theorem
~\ref{thm:rigid-summary},  $\D h_{l}^{(1)}(\omega/c)$ has $l+1$ zeros in the lower half plane and  we can therefore choose $M>0$ 
   and $\gamma>0$ such that $\Gamma^{M,\mu,\gamma}$ encircles these zeros.

\noindent{\bf Claim 1:} $\int_{\Gamma_2}$ and $\int_{\Gamma_4}\to0$ as $M\to\infty$\\
This claim is a direct application of the asymptotics
\eqref{Q-asymp} on the Hankel function ratio in the integrand. Namely, 
\begin{align*}
&\bigabs{\int_{\Gamma_{2, 4}} \int_{0}^{t-(r-1)/c} e^{-i \omega (t-s)}\ds f_{l}^{m}\ e^{i\omega\frac{r-1}{c}}\ ds\, \frac{h_{l}^{(1)}(\omega r/c)\ e^{-i\omega\frac{r-1}{c}}}{(\omega/c)\ \D h_{l}^{(1)}(\omega/c)} \,d\omega} \\
&\qquad\le \int_{\Gamma_{2,4}} \int_{0}^{t-(r-1)/c} e^{\mu (t-s)} \abs{\ds f_{l}^{m}}\ ds\, e^{-\Im\omega \frac{r-1}{c}}\ \frac{2}{\abs{\omega r/c}}\  d\omega \\ 
&\qquad\le \frac{C }{rM} \ \ e^{\gamma\frac{r-1}{c}}\ 
\int_{0}^{t-(r-1)/c} e^{\mu (t-s)} \abs{\ds f_{l}^{m}}\ ds 
 \to 0 \text{ as } M\to \infty.
\end{align*}
We therefore have
\begin{align}
U_{l,m}^{(1)}(r,t) &=  
\sum_{k=1}^{l+1}  
 \int_0^{t-(r-1)/c} i e^{-i \omega_{l,k} (t-s)} \D_s f_l^m(s)\, ds   
 \Res_{\omega=\omega_{l,k}}  \frac{h_l^{(1)}\left(\omega r/c\right)}{(\omega/c)\ \D h_{l}^{(1)}\left(\omega/c\right)}\nonumber  \\
&\ \ + \frac{1}{2\pi} \int_{\Gamma_3^{M,\gamma} }
\int_{0}^{t-(r-1)/c} e^{-i \omega(t-s)} \partial_{s} f_{l}^{m}(s)\, ds 
\frac{h_{l}^{(1)}\left(\omega r/c\right)}{(\omega/c)\ \D h_{l}^{(1)}\left(\omega/c\right)} \, d\omega, \label{Ulm1-A}
\end{align}
where $M>0$ sufficiently large. The size of $M$ depends on $l$, since the contour must enclose all $l+1$ poles of the integrand. This completes the proof of {\bf Claim 1}.\medskip

Our goal is to deform the contour 
$\Gamma_3=\Gamma_3^{M,\gamma}$ downward, by sending 
$\gamma\to\infty$, and to establish the following claim, from which Proposition~\ref{proposition:Ulm1} follows immediately. 

\noindent{\bf Claim 2:} The $d\omega$ integral in~\eqref{Ulm1-A} satisfies the estimate
\begin{equation}
\sup_{M>0}\left|\ \int_{\Gamma_3^{M,\gamma} }\ \dotsb\ d\omega\ \right|\ \to 0,\ \ {\rm as}\ \gamma\to\infty.
\end{equation}

To prove Claim 2, we integrate by parts in $s$, exploiting the oscillatory character of the integrand, in order to obtain absolutely convergent $d\omega$ integrals. 
\beq \label{eq:uI123}
\int_{\Gamma_3^{M,\gamma} } \dotsb\ d\omega 
= -\frac{1}{2\pi i} \left[I_{1}(r,t) + I_{2}(r,t) + I_{3}(r,t)\right]
\eeq
where 
\begin{subequations}
\begin{align}
I_{1}(r,t) &= \int_{-M-i\gamma}^{M-i\gamma} \int_{0}^{t-(r-1)/c} e^{-i \omega(t-s)} \partial_{s}^{2} f_{l}^{m}(s)\, ds \frac{h_{l}^{(1)}\left(\omega r/c\right)}{(\omega^{2}/c)\ \D h_{l}^{(1)}\left(\omega/c\right)} \, d\omega \label{eq:I1} \\
I_{2}(r,t) &= - \partial_{t} f_{l}^{m}\left(t-(r-1)/c\right) \int_{-M-i\gamma}^{M-i\gamma} e^{-i\omega(r-1)/c} \frac{h_{l}^{(1)}\left(\omega r/c\right)}{(\omega^{2}/c)\  \D h_{l}^{(1)}\left(\omega/c\right)} \, d\omega \label{eq:I2}  \\
I_{3}(r,t) &= \partial_{t} f_{l}^{m}(0) \int_{-M-i\gamma}^{M-i\gamma} e^{-i \omega t}\frac{h_{l}^{(1)}\left(\omega r/c\right)}{(\omega^{2}/c) \ \D h_{l}^{(1)}\left(\omega/c\right)} \, d\omega \label{eq:I3} 
\end{align}
\end{subequations}
We conclude the proof of {\bf Claim 2}, and therewith Proposition~\ref{proposition:Ulm1}, by now showing 
\begin{equation}
\text{for}\ j=1,2,3,\ \ \ \ \sup_{M>0} \left|\ I_j(r,t)\ \right|\ =\ {\cal O}(\gamma^{-1})\ \to\ 0,\ \ {\rm as}\ \gamma\to\infty.
 \label{I123to0}
 \end{equation}

\noindent{\bf Estimation of $I_1(r,t)$:} Choose $\gamma>2\eta$.
Using the bound~\eqref{eq:hankrat}, we have
\begin{align*}
\abs{I_{1}(r,t)} &\le \left|  \int_{-M-i\gamma}^{M-i\gamma} \int_{0}^{t-(r-1)/c} \bigabs{e^{-i \omega(t-s)} e^{i\omega\frac{r-1}{c}} 
 \D_s^2 f_l^m(s)} ds \frac{h_{l}^{(1)}(\omega r/c) e^{-i \omega\frac{r-1}{c}}}{(\omega/c) \D h_{l}^{(1)}\left(\omega/c\right)} \frac{1}{\omega}\ d\omega \right| \\
&\le \int_{0}^{t-(r-1)/c} e^{-\gamma (t-(r-1)/c-s)}  e^{-\eta s}\, ds\ \ [\D_tf]_{\wt,\eta}\ \frac{2c}{ r}\ \left|\ \int_{-M-i\gamma}^{M-i\gamma} 
\frac{1}{|\omega|^2}\ d\omega\ \right| \\
& \le \frac{8c}{\gamma r}\  e^{-\eta (t-(r-1)/c)}  \ \ [\D_tf]_{\wt,\eta}.
\end{align*} 

\noindent{\bf Estimation of $I_2(r,t)$:}
\begin{align*}
\abs{I_{2}(r,t)} &\le \abs{\D_t f_{l}^{m}(t-(r-1)/c)}\ \left|\  \int_{-M-i\gamma}^{M-i\gamma}  \bigabs{\frac{h_{l}^{(1)}\left(\omega r/c\right) e^{-i\omega((r-1)/c)}}{(\omega/ c) \D h_{l}^{(1)}\left(\omega/c\right)}} \frac{1}{|\omega|}\  d\omega\ \right| \\
&\le \frac{2c}{r} e^{-\eta (t-(r-1)/c)} \left| \int_{-M-i\gamma}^{M-i\gamma} \frac{1}{|\omega|^2}\ d\omega \right|\ [\D_tf]_{\wt,\eta}\ \le\ \frac{2c}{\gamma r}\ e^{-\eta (t-(r-1)/c)}\  [\D_tf]_{\wt,\eta}.
\end{align*}

\noindent{\bf Estimation of $I_3(r,t)$:}
Lastly,
\begin{align*}
\abs{I_{3}(r,t)} &\le \abs{\dt f_{l}^{m}(0)}\  \left|\ \int_{-M-i\gamma}^{M-i\gamma} \bigabs{e^{-i\omega t} e^{i\omega((r-1)/c)} } \bigabs{\frac{h_{l}^{(1)}\left(\omega r/c\right) e^{-i\omega((r-1)/c)}}{(\omega/c) \D h_{l}^{(1)}\left(\omega/c\right)}}\ \frac{1}{|\omega|}\ 
d\omega\ \right| \\
&\le \frac{2c}{r}\ \abs{\dt f_{l}^{m}(0)}\ e^{-\gamma (t-(r-1)/c)}\ \left|\ \int_{-M-i\gamma}^{M-i\gamma} \frac{1}{|\omega|^2}\ d\omega\ \right|\\ 
&\le\ \frac{2c}{\gamma\ r} \ e^{-\gamma (t-(r-1)/c)}\  [\D_t f_{l}^{m}]_{\wt,\eta}.\qquad\endproof
\end{align*}
%
\emph{Proof of Proposition~\ref{Ulm2}}. We must prove that $U_{l,m}^{(2)}(r,t)\equiv0$. Note the following bound,  for  $\Im\omega>0$ and $r< 1+ct$:
\begin{align}
|F^{(2)}_{l,m}(\omega,t)| \equiv \bigabs{\int_{t-(r-1)/c}^{\infty} \hspace{-1em}e^{i\omega s} \ds f_{l}^{m}(s) \, ds} & \le C \abs{\omega}^{-1} e^{-(\Im \omega + \eta)[t-(r-1)/c]} [\D_tf]_{\wt,\eta}.
\label{lap}
\end{align}
This follows from direct substitution and integration of the bound
$$
\left| e^{i\omega s}\D_sf_l^m(s) \right|\ \le\ e^{-\Im\omega s}\ e^{-\eta s}\ [\D_tf]_{\wt,\eta}.
$$
We now  use~\eqref{lap} to establish that $U_{l,m}^{(2)}(r,t)\equiv0$. Indeed,
\begin{align*}
\bigabs{U_{l,m}^{(2)}(r,t)} \,&\equiv\ \left|\
 \frac{1}{2\pi} \int_{-M+i\nu}^{M+i \nu} e^{-i\omega t} 
 F_{l,m}^{(2)}(\omega;t)\  
 \frac{h_l^{(1)}(\omega r/c)}{(\omega/c) \D h_l^{(1)}(\omega/c)} \,d\omega\ \right| \\
&\le \frac{2C c}{2\pi r}\left|  \int_{-M+i\nu}^{M+i \nu} 
e^{\nu( t-(r-1)/c )}\cdot   \frac{e^{-(\nu + \eta)[t-(r-1)/c]}}{|\omega|^2} \,d\omega\ \right|\ [\D_tf]_{\wt,\eta}\nn\\
& \le C \left|\ \int_{-M+i\nu}^{M+i\nu} \frac{d\omega}{\abs{\omega}^{2}}\ \right|\
  e^{-\eta(t-(r-1)/c)}
\le C \int_{-\infty}^{\infty} \frac{dz}{z^{2}+\nu^{2}} \, dz\ \to0
\end{align*}
as $\nu\to \infty$, and thus,  $U_{l,m}^{(2)}(r,t) = 0$. This completes the proof of {\bf Claim 2} and thus Proposition~\ref{proposition:Ulm1}. \qquad\endproof

\emph{Proof of the Resonance Expansion~\eqref{eq:fsum}}.
The resonance partial sum, defined as
$u^{(L)}(r,\Omega,t)\ =\ \sum_{l=0}^L\sum_{|m|\le l}\ u_{l}^{m}(r,t)\ Y_l^m(\Omega)
$
is 
$$
u^{(L)} = i\sum_{l=1}^{L} \sum_{|m|\le l}  \left[ \sum_{k=1}^{l+1}   \int_{0}^{t-(r-1)/c} \!\!\!\!\!e^{-i \omega_{l}^{k}(t-s)} \partial_{s} f_{l}^{m}(s)\, ds   \Res_{\omega=\omega_{l}^{k}}    \frac{ h_{l}^{(1)}\left(\omega r/c\right)}{(\omega/c)\ \D h_{l}^{(1)}\left(\omega/c\right)}  \right] Y_{l}^{m}(\Omega) 
$$
which gives Equation~(\ref {eq:uresexp}).
%

\noindent  To study the convergence of $u^{(L)}$ as $L$ tends to infinity, consider the difference \mbox{$u^{(N)}-u^{(L)}$,} where $N$ and $L$ are large with $N>L$ and prove:
 \begin{proposition}\label{proposition:int-est}
Let  $| \Im\omega_*|=\min_{k,l}\ |\Im\omega_{l,k}|$ and pick $\eta\in (0,| \Im\omega_* | )$.  Assume 
$[ \D_t f ]_{\wt,\eta}<\infty$. Then,   for $|a|+|b|\le 2$, 
\begin{equation}
\left|\  \D_t^{\alpha} \D_x^{\beta} \left(u^{(N)}(x,t)-u^{(L)}(x,t) \right)\
 \right| \ \le\ C\ 
  [\D_t f^{(N)}-\D_t f^{(L)}]_{\wt,\eta}\ 
   \frac{e^{-\eta\ (t-\frac{r-1}{c})}}{r}. \label{int-bound} 
\end{equation}
Therefore, for any compact subset $I_t\times\Omega_x\subset \R_t^+\times\R^d_x$, $\{u^{(L)}\}$ is a Cauchy sequence in $C^2(I_t\times\Omega_x)$ and converges uniformly on $I_t\times\Omega_x$ to a limit $u(x,t)$, which is a classical solution of the initial-boundary value problem~\eqref{eq:prob}.
\end{proposition}

\emph{Proof}.
We begin by estimating the difference $u^{(N)}(x,t)-u^{(L)}(x,t)$:
\begin{equation}
\begin{split}\label{uNuL-est}
&\left|\ u^{(N)}(x,t)-u^{(L)}(x,t)\ \right| \\
&\ \le\sum_{l=L+1}^{N+1} \sum_{|m|\le l} 
\left|Y_{l}^{m}(\Omega)\right|\   
\left| \sum_{k=1}^{l+1}  
\int_0^{t-\frac{r-1}{c}} e^{-i \omega_{l}^{k}(t-s)} \D_sf_{l}^{m}(s)\, ds  
\frac{h_{l}^{(1)}\left(\omega_{l,k} r/c\right)}{(\omega_{l,k}/c^2)\ 
\D^2 h_{l}^{(1)}\left(\omega_{l,k}/c\right)}\ \right|\\ 
&\ \le\sum_{l=L+1}^{N+1} \sum_{|m|\le l} 
\left|Y_{l}^{m}(\Omega)\right|\ \sum_{k=1}^{l+1}  \left|
\int_0^{t-\frac{r-1}{c}} e^{-i \omega_{l}^{k}(t-s)}\D_s f_l^m(s)\ ds \right|\
\!\left|\frac{ h_{l}^{(1)}\left(\omega_{l,k} r/c\right)}{(\omega_{l,k}/c^2)\D^2 h_{l}^{(1)}\left(\omega_{l,k}/c\right)}\right|\\  
&\ \le \sum_{l=L+1}^{N+1} \sum_{|m|\le l} 
\left|Y_{l}^{m}(\Omega)\right| \left| \D f_l^m\right|_{\infty,\eta} \sum_{k=1}^{l+1}  \!
\int\limits_0^{t-\frac{r-1}{c}} \!\!\!\!e^{\Im \omega_{l,k} (t-s)} e^{-\eta s}\, ds
\left|\frac{ h_{l}^{(1)}\left(\omega_{l,k} r/c\right)}{(\omega_{l,k}/c^2)\ \D^2 h_{l}^{(1)}(\omega_{l,k}/c)}\right|.
\end{split}
\end{equation}
%
To complete the proof of Proposition~\ref{proposition:int-est} we require Proposition~\ref{int-bound-lk} and the following  two propositions:
\begin{proposition}\label{proposition:tokitish}
There is a constant $C>0$, independent of $r$, such that for  $0\le a\le2$, and any $r>1$,
\begin{equation}
\left|\  \ \frac{ \D_r^a\left(\ h_{l}^{(1)}\left(\omega_{l,k} r/c\right)\ \right)}{(\omega_{l,k}/c^2)\ \D^2 h_{l}^{(1)}\left(\omega_{l,k}/c\right)}\ \right|
\le\ C\ \frac{(1+l)^{q_\star} \cdot\  l^a}{r}.
\label{eq:tokitish}
\end{equation}
\end{proposition}

\begin{proposition}  \label{proposition:Ylmbounds} 
$
\abs{\dtheta^{a}\dphi^{b}\,Y_l^m(\theta,\phi)} = O(l^{1/2+n}), \quad \text{for }\abs{a}+\abs{b}=n,\ \  n=0,1,2
$.
\end{proposition}

Proposition~\ref{proposition:tokitish} is proved using the uniform asymptotic expansions for  Hankel of large argument and large order and the asymptotic locations of the resonances $\{\omega_{l,k}\}$. The arguments are of the type used in~\cite{Tokita:1972lr} and very close to our proof of Proposition~\ref{proposition:tokitish2}, presented in Appendix~\ref{appendix:resbound}.  Proposition~\ref{proposition:Ylmbounds} is a classical result of Calderon and Zygmund~\cite{Calderon:1957zh}.
We now complete the proof using Propositions~\ref{int-bound-lk}, 
\ref{proposition:tokitish} and~\ref{proposition:Ylmbounds}.  First, we bound the $(l,m)$ term of~\eqref{uNuL-est} as follows:
\begin{multline}
\left|Y_l^m(\Omega)\right| \left| \D f_l^m\right|_{\infty,\eta} \ \ \sum_{k=1}^{l+1}  \
\int_0^{t-\frac{r-1}{c}} e^{\Im \omega_{l,k} (t-s)} e^{-\eta s}\, ds\ 
\ 
\left|\ \frac{ h_l^{(1)}\left(\omega_{l,k} r/c\right)}{(\omega_{l,k}/c^2)\ \D^2 h_l^{(1)}\left(\omega_{l,k}/c\right)}\ \right| \nn\\
\le\ C\ l^{1/2}\times \left| \D f_l^m\right|_{\infty,\eta}\times
\frac{(l+1)}{|\Im\omega_*|-\eta}\ 
e^{-\eta(t-\frac{r-1}{c})} \times \frac{(1+l)^{q_\star}}{r}.
\nn
\end{multline}
Summing this estimate over $l$ and $m$ yields:
$$
\left|\  u^{(N)}(x,t)-u^{(L)}(x,t)\  \right|\ \le\
 \frac{C}{|\Im\omega_*|-\eta}\ [\D_t f^{(N)} - \D_t f^{(L)}]_{q_\star+\frac{3}{2},\eta}\  r^{-1}e^{-\eta(t-\frac{r-1}{c})}.
$$
 Similarly, we have for any $a,b$ with $|a|+|b|\le2$, we obtain~\eqref{int-bound}.
 This completes the proof of Proposition~\ref{proposition:int-est}.
Theorem~\ref{theorem:NtD} now follows by passing to the limit $N\to\infty$, which proves the resonance expansion bound on $u-u^{(L)}$,~\eqref{eq:uerror}. \qquad\endproof
%
\section{Linearized bubble dynamics: solution of the initial-boundary value problem, decay estimates \& resonance expansion} \label{sec:linsol}

We next state our main results on decay estimates and resonance expansions for solutions of the initial-boundary value problem~\eqref{eq:n3linear},~\eqref{ib-data}.
 A key tool is the general result,  Theorem~\ref{theorem:NtD}, for the Neumann to Dirichlet map.

\begin{theorem} \label{thm:linsol}
Fix $\epsilon>0$ and arbitrary.  Suppose $\| \beta(t=0) \| = \sum_{l\ge0}\sum_{|m|\le l} (1+ l)^{2+{1\over6}}\ |\beta_l^m(0)|<\infty$. 
There exists a unique solution $ \Psi(r,\Omega,t),\ \beta(\Omega,t)$, defined for $r>1,\ \Omega\in S^2$, which solves the initial value problem~\eqref{eq:n3linear}. 
\begin{romannum}
\item {\bf Decay Estimates:}\ 
 The  solution satisfies the following decay estimates:
\begin{align}
|\beta(\Omega,t)|\ &\le\ C\ \| \beta(t=0) \|\ e^{-| \Im\lambda_\star(\epsilon) | t},\ \ \ \Omega\in S^2, \label{beta-decay-est}\\
|\Psi(x,t)|\ &\le\ 
 \begin{cases} C\ \frac{1}{|x|} e^{-\min \{\Im\lambda_\star(\epsilon), |\Im\omega_*(\epsilon)|\} (t-\epsilon(|x|-1) }\ \| \beta(t=0) \|, \, &1<|x|<1+\epsilon^{-1}t,\\
0, &|x|>1+\epsilon^{-1}t.
\end{cases}
\label{eq:Psi-decay-est}\end{align}
%
%
Here, $\Im \lambda_\star(\epsilon)<0$ and $\Im \omega_*(\epsilon)<0$ are, respectively, imaginary parts of the  deformation and rigid scattering resonances, which are closest to the real axis; see  Theorems~\ref{thm:rigid-summary} and~\ref{thm:deformation-summary}.
\item {\bf Resonance Expansion of $\beta$:}
%
Define the resonance partial sum for the bubble surface perturbation:
\begin{align}
\beta^{(L)}(\Omega,t)\ &=\  \sum_{l=0}^L \sum_{|m|\le l} \beta_l^m(t) Y_l^m(\Omega)\nn \\&=\ 
\sum_{l=0}^{L} \sum_{|m|\le l} \beta_l^m(0) \sum_{j=1}^{l+2}\ \lambda_{l,j}  e^{-i\lambda_{l,j} t} \Res_{\lambda=\lambda_{l,j}} \left[\ \lambda^2 +
 \rlhat G_l\left(\epsilon\lambda\right)\  \right]^{-1}\ Y_l^m(\Omega) ,
\label{beta-partial}
\end{align}
where $\{\lambda_{l,j}\}$, the solutions of $\lambda^2 +
 \rlhat G_l\left(\epsilon\lambda\right)=0$,   are the deformation scattering resonances; see~\eqref{srp-en-eqn} and
 Theorem~\ref{thm:deformation-summary}. 
  The limit
$
\beta(\Omega,t)\ =\ \lim_{L\to\infty} \beta^{(L)}(\Omega,t)
$
exists and converges uniformly in $C^2(S^2)$.

\item {\bf Resonance Expansion of $\Psi$:}
Given $\beta$, defined in \emph{(ii)}, define
\begin{align}
\Psi^{(L)}(r,\Omega,t) & = \sum_{l=0}^{L} \sum_{|m|\le l} i Y_l^m(\Omega) 
\sum_{k=1}^{l+1}   \!\int\limits_{0}^{t-\epsilon(r-1)} \!\!\!\! e^{-i \omega_{l,k}(t-s)} \partial_{s} \beta_{l}^{m}(s) \, ds    \frac{h_l^{(1)}(\epsilon \omega_{l,k} r)}{(\epsilon^2 \omega_{l,k}) \D^2 h_l^{(1)} (\epsilon \omega_{l,k})}
,
\label{Psi-partial}
\end{align}
where $\{\omega_{l,k}\}$ denotes the set of rigid Neumann scattering resonances; see Theorem~\ref{thm:rigid-summary}.

The limit
$
\Psi(r,\Omega,t)\ =\ \lim_{L\to\infty} \Psi^{(L)}(r,\Omega,t)
$
exists and converges in $C^2$ on any compact subset $K\subset\{|\bx|>1\}$. 
%
\end{romannum}
\end{theorem}

\emph{Proof of Theorem~\ref{thm:linsol}}.
 We apply the results of Section~\ref {sec:extneumann} with 
$u(r,\Omega,t)=\Psi(r,\Omega,t)$, $f(\Omega,t)=\beta(\Omega,t)$
 and $c=\epsilon^{-1}$, with  $\D_r\Psi(1,\Omega,t)=\D_t\beta(\Omega,t)$ as Neumann data,  coming from the linearized kinematic boundary condition and obtain:
\begin{align}
\Psi(r,\Omega,t) &= \sum_{l=0,\ l\ne1}^{\infty} \sum_{|m|\le l} \Psi_l^m(r,t) Y_l^m(\Omega), \label{eq:usum1}\\
\Psi_l^m(\Omega,r,t) &= \frac{1}{2\pi i} \int_{\mu-i\infty}^{\mu+i\infty} e^{pt} \widetilde{\D_t\beta_l^m}(p) \frac{h_l^{(1)}(i\epsilon p r)}{(i\epsilon p)\ \D h_l^{(1)}(i\epsilon)} \,dp,\ \ l=0,2,3,\dots,\ \ |m|\le l\label{Psi-lm}
\end{align} 
Setting $\omega=ip$ in~\eqref{Psi-lm} and applying Proposition~\ref{Ulm2} we obtain
\begin{align}
\Psi_l^m(\Omega,r,t) &= \frac{1}{2\pi} \int_{-\infty+i\mu}^{+\infty+i\mu} e^{-i\omega t}\ \int_0^{t-\epsilon(r-1)}\ e^{i\omega s}\ \D_s\beta_l^m(s)\ ds\
 \frac{1}{G_l(\epsilon\omega)}\ d\omega\ .\label{Psi-lm1}
 \end{align}

Our goal is to obtain a closed equation for $\beta$. This we obtain by substitution of~\eqref{eq:usum1},~\eqref{Psi-lm1} into the linearized dynamic boundary condition~\eqref{eq:lin3bernoulli}. 
Evaluating~\eqref{Psi-lm1} at $r=1$ and expressing it as a time convolution, we get, for each $l=0,2,3,\dots$ and $|m|\le l$,
\begin{equation}
\D_t\Psi_l^m(1,\Omega,t) =
\D_t \int_0^t\ \left({\cal L}^{-1} \frac{1}{G_l(\epsilon\star)}\right)(t-s)\ \D_s\beta_l^m(s)\ ds,
\label{LHS-bnlli}\end{equation}
where ${\cal L}^{-1}$ denotes the inverse Laplace transform.
Now taking the Laplace transform of equation~\eqref{LHS-bnlli} we obtain for the Laplace transform of the linearized dynamic boundary condition~\eqref{eq:lin3bernoulli}:
\begin{equation}
\frac{1}{G_l(i\epsilon p)}\left[ p^2\widetilde{\beta_l^m}(p)-p\beta_l^m(0) \right] 
  = \left[ 3\gamma \left( \tfrac{\Ca}{2} + \tfrac{2}{\We} \right)  \delta_{l0}
   + \tfrac{1}{\We}(l+2)(l-1) \right]\ \widetilde{\beta_l^m}(p),
\label{eq:lm-mode-kinematic}
\end{equation}
for $l=0,2,3,\dots$, and $|m|\le l$.
In obtaining~\eqref{eq:lm-mode-kinematic}, we have used  well-known relations for the Laplace transform of a time-derivative and a time-convolution, and that $-\Delta_SY_l^m=l(l+1)Y_l^m$.\ 
Next, solving~\eqref{eq:lm-mode-kinematic}  for $\widetilde{\beta_l^m}(p)$ yields 
\begin{align}
\widetilde{\beta_l^m}(p) \ &=\ \beta_l^m(0) \ p\  \left(p^2-\rlhat G_l(i\epsilon p)\right)^{-1}, 
 \label{betalmp}
\end{align}
where  $\rlhat = \tfrac{1}{\We}(l+2)(l-1)+ 3\gamma \left( \tfrac{\Ca}{2} + \tfrac{2}{\We} \right) \delta_{l0}$.
Inverting the Laplace transform and changing variables to $\lambda=ip$, we have
\begin{equation}
\beta_l^m(t)\ =\ 
\beta_l^m(0)\ \frac{1}{2\pi}
\int_{-\infty+i\mu}^{+\infty+i\mu} e^{-i\lambda t} \frac{ i\lambda  }{ \lambda^2+\rlhat G_l(\epsilon\lambda)}\ d\lambda,\ \ \ \mu>0.
\label{beta-lmt}
\end{equation}
Thus,
\beq  \label{eq:betauhp}
\beta(\Omega,t) = \sum_{l=0}^{\infty} \sum_{|m|\le l} \beta_l^m(0)  \frac{1 }{2\pi} \int_{-\infty+i\mu}^{+\infty+i\mu} e^{-i\lambda t} \frac{i\lambda}{\lambda^{2} + \rlhat G_{l}(\epsilon \lambda)} \, d\lambda\  Y_{l}^{m}(\Omega),\ \ \mu>0.
\eeq

\noindent We will show  
 (1) each term in the series~\eqref{eq:betauhp}
 can be expressed as a finite sum of residues at poles in the lower half $\lambda$-plane and 
  (2) under suitable regularity hypotheses on $\beta(\Omega,0)$, the series converges uniformly in $C^2(S_\Omega^2\times \R_t^+)$.

\emph {\it Evaluation of $\beta_l^m(t)$, in terms of residues:}
  The poles of the integrand of~\eqref{beta-lmt} are precisely the {\it deformation resonances}, described in Theorem~\ref{thm:deformation-summary}. In particular, there are $l+2$ poles, zeros of $\lambda^2+\rlhat G_l(\epsilon\lambda)$.
   
Choose $\Gamma^{M,\mu,\gamma}=\Gamma_1+\Gamma_2+\Gamma_3+\Gamma_4$   
  a rectangular contour of the  type used in the proof of Proposition~\ref{proposition:Ulm1}, with $M, \gamma, \mu$  so that $\Gamma^{M,\mu,\gamma}$ encircles  these $l+2$ poles in the complex $\lambda$-plane. By the residue theorem
 \begin{equation}
 \sum_{j=1}^{l+2} i\lambda_{l,j}e^{-i\lambda_{l,j} t} 
  \Res_{\lambda=\lambda_{l,j}} \left[ \lambda^2 +
 \rlhat G_l\left(\epsilon\lambda\right)  \right]^{-1} = -\frac{1}{2\pi i} \int_{\Gamma^{M,\mu,\gamma}}
 \frac{ i\lambda  \; e^{-i\lambda t}  }{ \lambda^2+\rlhat G_l(\epsilon\lambda)}\ d\lambda.
 \label{res-thm-app}
 \end{equation}
 Note that 
 \begin{equation}
\beta_l^m(t) = \beta_l^m(0)\ \frac{1}{2\pi}\ \lim_{M\to\infty} \int_{\Gamma^M_1}\ \dotsi\ d\lambda,
\end{equation}
where $\ \dotsi$ denotes the integrand in~\eqref{res-thm-app}.
 Thus, we would have
 \begin{equation}
 \beta_l^m(t) = \sum_{j=1}^{l+2} \lambda_{l,j} e^{-i\lambda_{l,j} t} 
  \Res_{\lambda=\lambda_{l,j}} \left[\ \lambda^2 +
 \rlhat G_l\left(\epsilon\lambda\right)\  \right]^{-1}\ ,
 \label{beta-lm-res-sum}
 \end{equation}
 as well as the  expansion~\eqref{beta-partial}, if we can prove
 $
 \int_{\Gamma_{2,4}}\dotsi d\lambda\to0$ as $M\to\infty$, and $\int_{\Gamma_3}\dotsi d\lambda\to0$ as $\gamma\to\infty.
$
 
To prove $\int_{\Gamma_{2,4}}\dotsi d\lambda\to0$ as  $M\to\infty$ we begin by noting
 (Appendix~\ref {app:polyhank})  the identity
\beq
G_l(\epsilon\lambda) =  -(l+1)+i\epsilon\lambda + \frac{\epsilon \lambda p_{l}'(\epsilon\lambda)}{p_{l}(\epsilon\lambda)} .
\eeq
Thus $G_l(\epsilon\lambda)$ to be asymptotically linear in $\epsilon\lambda$. Thus  $M$ can be chosen large enough so that
\beq
\frac{1}{\abs{\lambda^{2}+\rlhat G_l(\epsilon\lambda)}} \le \frac{C_{l}}{\abs{\lambda}^{2}}, \quad \text{for } \Re \lambda>M .
\eeq
Therefore, we have
\begin{equation}
\bigabs {\int_{\Gamma_{2}} e^{-i\lambda t} \frac{i \lambda}{\lambda^{2} + \rlhat G_l(\epsilon \lambda)} \, d\lambda } \le \int_{\Gamma_{2}} \abs{e^{-i\lambda t}} \frac{C_{l}\abs{\lambda}}{\abs{\lambda}^{2}} \, d\lambda 
\le C_{l} (\mu+\gamma) e^{\mu t} \frac{1}{M} \to 0 \,\text{ as }\, M\to\infty.\nn \end{equation}
A similar estimate holds  for the integral along $\Gamma_{4}$.\

We next bound the integral: $\lim_{M\to\infty}\int_{\Gamma_3}\dots\ d\lambda=\int_{\Im\lambda=-\gamma}\dots \ d\lambda$ as $\gamma\to\infty$.  
For $\gamma=\Im\lambda$ sufficiently large,
\begin{align}
\left( \lambda + \frac{\rlhat G_l(\epsilon\lambda)}{\lambda } \right) \ &=\ \left( \lambda + \frac{\rlhat}{\lambda} \left[ -(l+1)+i\epsilon\lambda + \frac{\epsilon \lambda p_{l}'(\epsilon\lambda)}{p_{l}(\epsilon\lambda)} \right] \right)^{-1}\\ 
\ &=\ \frac{ 1}{\lambda } \frac{ 1}{1+ \epsilon \rlhat \cO\bigl(\lambda^{-1}\bigr) } \ =\ \frac{ 1}{\lambda } \left[ 1 + \epsilon \rlhat \cO\bigl(\lambda^{-1}\bigr) \right] . 
\nn\end{align}
Therefore,
$
\D_\lambda \left( \lambda + \frac{\rlhat G_{l}(\epsilon\lambda)}{\lambda } \right)^{-1}  =\ -\lambda^{-2} + \cO\bigl(\lambda^{-3}\bigr). 
$
Using these asymptotics and integration by parts we have
\begin{equation}
\begin{split}
&\bigabs{\int_{-M-i\gamma}^{M-i\gamma} e^{-i\lambda t} \frac{i \lambda}{\lambda^{2} + \rlhat G_{l}(\epsilon \lambda)} \, d\lambda }\\ 
 &\quad\qquad\le \frac{ 1}{t } \bigabs {\int_{-M-i\gamma}^{M-i\gamma} e^{-i\lambda t}  \left( -\frac{ 1}{\lambda^{2} } + \cO\bigl(\lambda^{-3}\bigr) \right) \, d\lambda } + e^{-\gamma t} \bigabs{\frac{ 1}{\lambda } \left.\left( 1 + \epsilon \rlhat \cO\bigl(\lambda^{-1}\bigr) \right) \right|_{-M-i\gamma}^{M-i\gamma}  }\\ 
&\quad\qquad\le \frac{ n_l e^{-\gamma t}}{t } \int_{-M-i\gamma}^{M-i\gamma} \frac{ 1}{\abs{\lambda}^{2} } \, d\lambda + \cO\bigl(M^{-1}\bigr)\ \le \frac{ \kappa_l e^{-\gamma t}}{t }, \quad t> 0.\label{Mgamma}\raisetag{-88pt}
\end{split}
\end{equation}
Recall that 
\begin{equation}
\int_{\Im\lambda=-\gamma}\dotsi\ d\lambda = 
\lim_{M\to\infty} \int_{-M-i\alpha}^{M-i\alpha} e^{-i\lambda t} \frac{i \lambda}{\lambda^{2} + \rlhat G_{l}(\epsilon \lambda)} \, d\lambda\ .
\nn\end{equation}
Thus, passing to the limit in~\eqref{Mgamma}, first  $M\to\infty$
 and then  $\gamma\to\infty$, we obtain that the contribution from  the contour $\Gamma_3$ can be made arbitrarily small. This proves~\eqref{beta-lm-res-sum}, the expression for  $\beta_l^m(t)$ as a sum over residues. Furthermore, multiplication of expression~\eqref{beta-lm-res-sum} by $Y_l^m(\Omega)$ and summing over $l$ from $0$ up to $L$ yields the expression for $\beta^{(L)}(\Omega,t)$, the resonance expansion partial sum for the bubble surface perturbation, $\beta(\Omega,t)$.

We now discuss the convergence of $\beta^{(L)}(\Omega,t)$ as $L\to\infty$. Let $ N>L$ and consider the difference:
\begin{equation}
\beta^{(N)}(\Omega,t)-\beta^{(L)}(\Omega,t)\ =\ 
 \sum_{l=L}^N\sum_{|m|\le l} \beta_l^m(0) \sum_{j=1}^{l+2} e^{-i\lambda_{l,j} t} \Res_{\lambda=\lambda_{l,j}} \frac{\lambda}{  \lambda^2 +
 \rlhat G_l\left(\epsilon\lambda\right) }\ Y_l^m(\Omega).
 \label{beta-diff}\end{equation}
 We require the following bound, proved in Appendix~\ref{appendix:resbound}:

\begin{proposition}\label{proposition:tokitish2}
For $l$ sufficiently large,
\beq
\Res_{\lambda=\lambda_{l}^{j}} \frac{\lambda}{\lambda^{2}+\rlhat G_l(\epsilon\lambda)}  = \bigO\Bigl(l^{-4/3}\Bigr).
\eeq
\end{proposition}

To bound the difference in~\eqref{beta-diff}  we use Proposition~\ref{proposition:tokitish2} and Theorem~\ref{thm:deformation-summary} characterizing and estimating the deformation resonances, $\lambda_{l,j}$ and the pointwise bounds on $Y_l^m(\Omega)$ of Proposition~\ref{proposition:Ylmbounds}. Let $\D_\Omega^\alpha$ denote any differential operator of order $|\alpha|$, acting tangentially to $S^2$, {\it i.e.} $\D_\Omega^\alpha=\D_\theta^{\alpha_1}\D_\phi^{\alpha_2},\ |\alpha|=\alpha_1+\alpha_2$. Then, 
\begin{align}
\left| \D_\Omega^\alpha\left(\beta^{(N)}(\Omega,t)-\beta^{(L)}(\Omega,t)\right) \right|
 &\le C\ e^{-t\ |\Im\lambda_\star(\epsilon)|}\ \sum_{l=L}^N\sum_{|m|\le l} |\beta_l^m(0)|\ \sum_{j=1}^{l+2} (1+l)^{-{4\over3}}\ l^{|\alpha|+{1\over2}}\nn\\
 & \le\ C\ e^{-t\ |\Im\lambda_\star(\epsilon)|}\  \sum_{l=L}^N\sum_{|m|\le l} (1+ l)^{{1\over6}+|\alpha|}\ |\beta_l^m(0)|~.
 \end{align}
 
 Thus, if \ 
$
 \sum_{l=L}^N\sum_{|m|\le l} (1+ l)^{2+{1\over6}}\ |\beta_l^m(0)| < \infty
 $, 
 we have that $\beta^{(L)}(\Omega,t)$ and derivatives in $\Omega$ and $t$ up to order two, converge uniformly in $ S_\Omega^2\times\R_t^+$.  Moreover, the $\lim_{L\to\infty}\beta(\Omega,t)$ is in $C^2(S_\Omega^2\times\R_t^+)$ and satisfies the decay estimate
  \begin{equation}
 \left|\ \beta(\Omega,t)\ \right|\ \le\ 
  \sum_{l=L}^N\sum_{|m|\le l} (1+ l)^{2+{1\over6}}\ |\beta_l^m(0)|
 \ \cdot\  e^{-\left| \Im\lambda_\star \right| t }\ .
 \label{beta-decay}
 \end{equation}

 Now apply our result on the Neumann to Dirichlet map, Theorem~\ref{theorem:NtD} , to conclude 
  for $|x|>1,\ \ t>0$ that 
\begin{equation}
\begin{split}
\abs{\Psi(r,\Omega,t)} &\le   C\ e^{-\min \{\abs{\Im\lambda_*(\epsilon)},\, \abs{\Im\omega_*(\epsilon)}\} t}\  [\D_t\beta]_{\alpha,\abs{\Im\lambda_*(\epsilon)}}\\ 
&\le C  e^{-\min \{|\Im\lambda_*(\epsilon)|\} t}\ \sum_{l=L}^N\sum_{|m|\le l} (1+ l)^{2+{1\over6}}\ |\beta_l^m(0)|~.
\label{eq:decay-est}
\end{split}
\end{equation}
The second inequality in~\eqref{eq:decay-est}, follows since
 $|\Im\lambda_*(\epsilon)| <  |\Im\omega_*(\epsilon)|=\cO(\epsilon^{-1})$ for $\epsilon$ sufficiently small. 
Moreover, $(\Psi(r,\Omega,t),\beta(\Omega,t))$ is a classical solution of the initial value problem.
%
\section{Scattering Resonance Energies\ ---\ Statements of Detailed Results} \label{sec:resonances}

As discussed in Section~\ref{sec:linsol}, two families of scattering resonance energies arise in the analysis of the linearized bubble dynamics near the spherical equilibrium state:
\begin{remunerate}
\item Rigid resonance energies, $\bigset{\omega_{l,k}}$, solutions of 
$\D h_l^{(1)}(\epsilon\omega)\ =\ 0 $.
\item Deformation resonance energies, $\big\{\lambda_{l,k}\big\}$, solutions of 
\begin{equation}
\lambda^{2} + \rlhat  G_l(\epsilon\lambda) = 0,\qquad\ \  G_l(z)\ \equiv\ \frac{z\ \D h_l^{(1)}(z) }{h_l^{(1)}\left(z \right)},\ \ \rlhat\ \  {\rm is\  given\ by\
~\eqref{eq:rlhat-notation}.}
\label{eq:deformation-eqn}
\end{equation}
  \end{remunerate}

 Both classes of resonances depend on $\epsilon$. The incompressible limit, $\epsilon\to 0$, is a singular limit. Indeed our results show that for each $l=0,2,3,\dots$,
 as $\epsilon$ tends to zero, the imaginary parts of all  $l+1$ rigid 
 resonances and of $l$ of the $l+2$  deformation resonances lie near circular arcs in the lower half plane and have imaginary parts which tend to minus infinity. 
 The remaining two deformation resonances are just below and converge to the real axis as $\epsilon\to0$; see Figure~\ref{fig:resonances} of the Introduction.  The  imaginary parts are $\cO( \epsilon^{-3} \exp(\kappa\epsilon^{-2}) )$ as $\epsilon$ tends to zero.
%
%

\subsection[Deformation resonances]{Deformation resonances;\ $\bigset{\lambda_{l,j}(\epsilon)}$, such that $\lambda^2 + \rlhat\ G_l(\epsilon\lambda) = 0$} \label{sec:soft}
${}$ 
\begin{theorem} \label{thm:def-res-general}
Fix  $\epsilon>0$ and arbitrary.
%
%
\begin{romannum}
\item
There are $l+2$ solutions of 
\begin{equation}
\lambda^2 + \rlhat\ G_l(\epsilon\lambda) = 0, 
\label{lambda-eqn}\end{equation}
 denoted $\{\lambda_{l,j}(\epsilon)\}$, for $l=0,2,3,4,\dotsc$ and $j=1,\dotsc,l+2$. These are the deformation resonance energies.
\item
$\{\lambda_{l,j}(\epsilon)\}$ are symmetric about the imaginary axis and satisfy $\Im \lambda_{l,j} <0$.
\end{romannum}
\end{theorem}
\bigskip

\begin{theorem}[Asymptotics of Rayleigh deformation resonances, $\lambda_l^\pm(\epsilon)$  for small Mach number, $\epsilon$]\label{thm:2smallep}

\begin{romannum}
\item
For $l=0$, the two deformation resonances are located  in the lower half plane, a distance $\cO(\epsilon)$ from the real axis, are given by:
\beq \label{eq:eps0-res-l0}
\lambda_{0}^{\pm}(\epsilon) = \pm\sqrt{r_{0} - \frac{r_{0}^{2}}{4} \epsilon^{2}} - i \frac{r_{0}}{2} \epsilon,\qquad\  r_0=\frac{3\gamma}{2}\Ca + 2(3\gamma-1)\frac{1}{\We}\ .
\eeq
\item
Fix $L_*>2$.  There exists $\epsilon^{*}(L^{*})$ such that for $\epsilon \le \epsilon^{*}$  and $2\le l \le L^{*}$, two of the  $l+2$ deformation resonances (see Theorem~\ref{thm:def-res-general}), are located  \emph{very slightly} below the real axis and given by: $\lambda^\pm_l=\Re\lambda^\pm_l(\epsilon)+i\Im\lambda^\pm_l(\epsilon)$, where
\begin{align} 
\Re\lambda_{l}^{\pm}(\epsilon) &= \pm \sqrt{\frac{1}{\We}(l+2)(l-1)(l+1)} \left[ 1 - \frac{(l+2)(l-1)}{2(2l-1)} \left( \frac{\epsilon} {\sqrt{\We}} \right)^{2} + \cO_{l,\text{real}}\left(\epsilon^{4}\right) \right], \nonumber \\
%
\Im\lambda_{l}^{\pm}(\epsilon) &= - \, \frac{1}{\epsilon}\! \left[  \frac{1}{2} \left[(l+2)(l-1) \right]^{l+1} (l+1)^{l} \left[\frac{2^{l}l!}{(2l)!}\right]^{2} \!\left( \!\frac{\epsilon} {\sqrt{\We}}\! \right)^{2l+2} \!\!\!\!\!\! + \bigO_{l} \Bigl(\! \bigl[ \epsilon \We^{-1/2} \bigr]^{2l+4} \Bigr) \!
\right]\!. \label{eq:smallepsres}
\end{align}
\item In the incompressible limit,\; $\epsilon\to0^+$,\; the deformation resonance equation~\eqref{lambda-eqn}, which  has  $l+2$ roots in the lower half plane for $\epsilon>0$, reduces to the quadratic equation: 
\begin{align*}
\lambda^2 + \rlhat\ G_l(0) = \lambda^2 - 3\gamma \left( \tfrac{\Ca}{2} + \tfrac{2}{\We} \right)\, \delta_{l0}- \tfrac{1}{\We} (l+2)(l+1)(l-1)&= 0, \quad l=0,2,3,\dotsc
\end{align*}
and has two real roots:
\begin{align}
\lambda^\pm_{l}(0) = \pm\sqrt{3\gamma \left( \frac{\Ca}{2} + \frac{2}{\We} \right)\, \delta_{l0} +\frac{1}{\We} (l+2)(l+1)(l-1)}, \qquad l=0,2,3,\dotsc \label{eq:eps0-res}
\end{align}
which are \emph{real} frequencies. \ The corresponding solutions of the linearized time-dependent perturbation equations:
\begin{equation}
\Psi_{l}^{\pm} = A^\pm_{l}\ e^{-i\lambda^{\pm}_{l}(0)\, t}\ Y_l^m(\Omega)\ r^{-l-1}, \  \ \ 
\beta_{l}^{\pm} = B^\pm_{l}\ e^{-i\lambda^{\pm}_{l}(0)\, t}\ Y_l^m(\Omega),\ \ \ l=0,2,\dots \nn 
\end{equation}
are {\it undamped} and oscillatory solutions.
\end{romannum}
\end{theorem}

\begin{remark}
In Section~\ref {sec:longlivedres-thm}, we present an asymptotic analysis and numerical computations offering strong evidence for the scattering resonance energy nearest to the real axis being a distance of order
\begin{equation}
\cO\!\left(\epsilon^{-1}\ \We\,\epsilon^{-2}e^{-\kappa\,\We\, \epsilon^{-2}}\right),\ \ \kappa>0, \quad{\rm occurring\ at\ order}
 \ \ l_\star(\epsilon)=\cO\!\left(\We\,\epsilon^{-2}\right).
 \end{equation}
\end{remark}
%
 %
%
%
%
%

Finally, the following theorem establishes that  for any fixed $\epsilon$ and  sufficiently large $l$, the deformation resonances (excluding one) are distributed near a circular arc. In particular, this implies that for any fixed $\epsilon$ scattering resonances are in the lower half plane and uniformly bounded away from the real axis.
 
\begin{theorem} \label{thm:big-l}
Fix $\epsilon>0$ and arbitrary. The $l+2$ deformation resonances are approximated as follows:
%
%
\begin{romannum}
\item
There are $l+1$ resonances near zeros of $\D H_{\nu}^{(1)}(\epsilon\lambda)$; see Equation~(\ref {eq:dH-zero}).  Precisely, there exists $K_1>0$ such that for 
 $
l \ge K_1 \left(\epsilon^{-2}\We\right)^{3/(1-3q)},\ 0<q<1/3,
$ 
those in the fourth quadrant can be expressed as:
%
\begin{multline}
\lambda_{l,s}(\epsilon) = \frac{l+1/2}{\epsilon} \biggr[1 + 2^{-1/3} e^{-2\pi i/3} (l+1/2)^{-2/3}  \bigabs{\eta_{s}'}\\ + \tfrac{3}{10} 2^{-1/3} e^{-4\pi i /3} (l+1/2)^{-4/3}\bigabs{\eta_{s}'}^{2} + \bigO\Bigl( s^{-1/3} (l+1/2)^{-2/3-q} \Bigr) \biggl],  \label{eq:l+1-res}
\end{multline}
%
for $s=1,2,\dotsc, \left\lfloor (l+1)/2 \right\rfloor + 1$.
Here, $\eta_{s}'$ is the $s$th zero of $\D_{z} Ai(z)$;\ $0>\eta_{1}'>\eta_{2}'>\dotsb$~\cite{Olver:1954lr}. Note the error is uniform in $\epsilon$.
These resonances come in pairs, symmetric about the imaginary axis,  $\{\lambda_{l,s} , -\overline{\lambda_{l,s}} \}$. For even~$l$, $\lambda_{l,(l+2)/2}$ is a simple resonance  on the negative imaginary axis.
\item
The remaining deformation resonance is located on the negative imaginary axis and can be expressed, for some $K_{2}$ sufficiently large and  $l\ge K_{2} \,\epsilon^{-2} \We $, as
\beq 
\quad\lambda_{l,l+2}(\epsilon) = -i\, \frac{l+1/2}{\epsilon} \left[ \frac{\epsilon^{2}(l+1/2)}{\We} + \frac{1}{2} \frac{\We}{\epsilon^{2}(l+1/2)} + \bigO\!\left(\left[\epsilon^{2}\We^{-1}(l+1/2)\right]^{-3}\right) 
\right].
\eeq
\item The set of deformation resonance energies is  uniformly bounded away from $\R$. That is, for some $l_{\star}(\epsilon), j_{\star}(\epsilon)$, 
\begin{equation}
\Im\lambda_{l,j}(\epsilon) \le\ \Im\lambda_{{l_\star},{j_\star}}(\epsilon)\ \equiv\  \Im\lambda_\star(\epsilon) < 0,\ \ \ {\rm all}\ \  l\ge0,\ |j|\le l+2~.
\end{equation}
\item Moreover, there exists $C>0$ such that as $l\to\infty$, $\lambda_{l,j}(\epsilon) = \epsilon^{-1} {\cal O}( l)$, $C\ \epsilon^{-1} l^{1/3} \le \abs{\Im\lambda_{l,j}(\epsilon)}.$
\end{romannum}
\end{theorem}

We conclude this section with a proof of Theorem~\ref{thm:def-res-general}. Theorem~\ref{thm:2smallep} is proved in Section~\ref {sec:rayleigh-res-analysis}. Theorem~\ref{thm:big-l} is proved in Section~\ref {sec:def-res-pfs}. 
To prove Theorem~\ref{thm:def-res-general} we use the following rational function representation of $G_l(z)\equiv\D h_l^{(1)}(z)/h_l^{(1)}(z)$, which follows from~\eqref{eq:poly-rep-hank}:
\begin{proposition}
$G_l(z) = -(l+1) + iz + \frac{z p_{l}'(z)}{p_{l}(z)}$. 
\label{thm:Gl-poly}
\end{proposition}

\emph{Proof of Theorem~\ref{thm:def-res-general}}.
\begin{romannum}
\item  We first observe that there $l+2$ scattering resonance energies. To see this we observe, by Proposition~\ref{thm:Gl-poly},  that $\lambda^2+\rlhat G_l(\epsilon\lambda)=0$ is equivalent to 
\beq
\lambda^{2} + \rlhat \left[-(l+1)+i\epsilon\lambda + \frac{\epsilon \lambda p_{l}'(\epsilon\lambda)}{p_{l}(\epsilon\lambda)} \right] = 0,
\label{eq:def-res-p}\eeq
which clearly has $l+2$ roots since $p_{l}(z)$ is a polynomial of degree $l$.   For $\epsilon=0$, the equation reduces to $\lambda^{2}=\rlhat (l+1)$, yielding two real frequencies, corresponding to the non-decaying Rayleigh modes of the incompressible limit problem~\cite{Lamb:1993mu}.%
\medskip

\item We now prove that if $\lambda$ is a solution of~\eqref{eq:def-res-p}, then so is $-\overline{\lambda}$. 
By~\eqref{eq:hlHnu} 
\beq
\lambda^2\ +\ \rlhat\ G_l(\epsilon\lambda)\ =\ 0\ \ \iff\ \ 
 \lambda^{2} + \rlhat \left[-\frac{1}{2} + \frac{\epsilon\lambda\, \D H_{\nu}^{(1)}(\epsilon\lambda)}{H_{\nu}^{(1)}(\epsilon\lambda)}\right] = 0,\ \ \ \nu=l+1/2.
\label{eq:def-res-p1}\eeq
Using~\eqref{H-recur} we have 
$
\lambda^{2} + \rlhat \left[-\frac{1}{2} + \frac{\epsilon\lambda}{H_\nu^{(1)}(\epsilon\lambda)}\
 \left(\ \nu\ \frac{H_\nu^{(1)}(\epsilon\lambda)}{\epsilon\lambda} - H_{\nu+1}^{(1)}(\epsilon\lambda)\ \right)\ \right]\ =\ 0.
$

Taking the complex conjugate  and using the analytic continuation formulae
~\eqref{eq:analy1} and~\eqref{eq:analy2} we obtain
\begin{align}
0 &= \overline{\lambda}^{2} + \rlhat\! \left[-\frac{1}{2} + \epsilon\nu
 - \epsilon\overline{\lambda}
  \frac{\overline{H_{\nu+1}^{(1)}(\epsilon\lambda)} }{\overline{H_{\nu}^{(1)}(\epsilon\lambda)   }     } \right] 
  \!= \bigl(-\overline{\lambda}\bigr)^{2}\! + \rlhat \!\left[-\frac{1}{2} + \epsilon\nu
 + \epsilon (-\overline{\lambda})
  \frac{{H_{\nu+1}^{(1)}(\epsilon(-{\overline\lambda}) )} }{{H_{\nu}^{(1)}(\epsilon(-{\overline\lambda}) )   }     } \right]\!.
 \end{align}
Thus, $-\overline{\lambda}$ satisfies~\eqref{eq:def-res-p1}
and is therefore a scattering resonance.
\medskip

We now prove that all resonances lie in the open lower half plane, $\Im\lambda<0$. We begin, following~\cite{Tokita:1972lr}, by showing that there are no resonances along the real axis.  Let $\nu=l+1/2$.
Using the definition of $H_\nu^{(1)}$ in terms of Bessel funcitons, equation~\eqref{eq:def-res-p1} for the resonances can be written
\begin{align}
0 &= \left(\frac{\lambda^{2}}{\rlhat}-\frac{1}{2} \right) \left[J_{\nu}\left(\epsilon \lambda\right) +i Y_{\nu}\left(\epsilon \lambda\right)\right] + \left(\epsilon \lambda\right) \left[\D J_{\nu}\left(\epsilon \lambda\right) +i \D Y_{\nu}\left(\epsilon \lambda\right)\right] .
\end{align}
Now assume $\lambda=x/\epsilon$, where $x$ is real and nonzero, is a resonance and we show that this leads to a contradiction.  Indeed, 
\begin{align}
0 &= \left(\frac{1}{\epsilon^{2}}\frac{s^{2}}{\rlhat}-\frac{1}{2} \right) \left[J_{\nu}(x) + i Y_{\nu}(x) \right] + x \left[J_{\nu}'(x) + i Y_{\nu}'(x)\right] \\
&= \left[\left(\frac{1}{\epsilon^{2}}\frac{x^{2}}{\rlhat}-\frac{1}{2} \right) J_{\nu}(x) + x J_{\nu}'(x)\right] + i \left[\left(\frac{1}{\epsilon^{2}}\frac{x^{2}}{\rlhat}-\frac{1}{2} \right) Y_{\nu}(x) + x Y_{\nu}'(x)\right].
\end{align}
Since $x\in\R$ both quantities in braces equal zero.  Therefore,
 we have %
\beq
\left(\frac{1}{\epsilon^{2}}\frac{x^{2}}{\rlhat}-\frac{1}{2} \right) = -\frac{x J_{\nu}'(x)}{J_{\nu}(x)} = -\frac{x Y_{\nu}'(x)}{Y_{\nu}(x)} 
\nn\eeq
implying
$
-x\ \left[\ J_{\nu}'(x) Y_{\nu}(x) -  Y_{\nu}'(x) J_{\nu}(x) \ \right] = 0. 
$
The expression in square bracket is the Wronskian of two  linearly independent solutions of Bessel's equation and satisfies the following identity~\cite{Abramowitz:1965zr}:
\beq
J_{\nu}'(x) Y_{\nu}(x) -  Y_{\nu}'(x) J_{\nu}(x) = \frac{2}{\pi x} \ne 0. \eeq
Hence, for $x\ne0$,  we have a contradiction and thus $\Im\lambda\ne0$.
\medskip

Next, we claim that there are no resonances in the upper half plane, and therefore $\Im\lambda<0$. For if $\lambda$ is a scattering resonance with 
$\Im\lambda>0$, this would give rise to 
a solution of the initial-boundary value problem for the linearized bubble perturbation system~\eqref{eq:n3linear} of the form: 
$\Psi(r,\Omega)=e^{-i\lambda t} h_l^{(1)}(\epsilon\lambda_l r),\ \ \beta(\Omega) = e^{-i\lambda t} Y_l^m(\Omega)$. 
Since $\Im\lambda>0$, this solution would  decay to zero exponentially as $r=|x|\to\infty$  ( $h^{(1)}_l(z)$ is $e^{iz}$ times a rational function of $z$), be smooth and be exponentially {\it growing} in $t$. This contradicts conservation of energy, Proposition~\ref{thm:energy}. Thus, $\Im\lambda\le0$ and by the above discussion $\Im\lambda<0$.\qquad\endproof
\end{romannum}
%
%
%
\section{Longest-lived mode; formal asymptotics and precision numerics}\label{sec:longlivedres-thm}
From Theorems~ \ref {thm:def-res-general} to~ \ref {thm:big-l}  we know that for any finite (not necessarily small) $\epsilon$, there is a  strip:\ 
\begin{equation}
\Im \lambda^{\pm}_\star(\epsilon)< \Im\lambda <0,
\nn\end{equation}
 in which there are no scattering resonances. Moreover, the incompressible limit spectral problem, $\epsilon=0$, has eigenvalues only at the real Rayleigh frequency pairs, given in Equations~\eqref{eq:eps0-res-l0} and~\eqref{eq:eps0-res}, and none in the lower half plane.\\

\noindent {\it Question:\ \  How thin is the strip? That is, can we estimate $\Im \lambda^{\pm}_\star(\epsilon)$, for $\epsilon$ small?}\\

Figure~\ref {fig:resonances} displays the distribution of numerically computed Rayleigh resonances for a fixed $\epsilon>0$ and several values of $l$. The figure shows the resonance in the lower half plane (a pair, $\lambda^{\pm}_{\star}(\epsilon)$ by $\lambda\mapsto-\overline{\lambda}$ symmetry), which is closest to the real axis.  These resonances correspond to the slowest time-decaying term in the resonance expansion and the resulting decay rate in Theorem~\ref{theorem:NtD}.\bigskip

Let $l_{\star}(\epsilon)$ denote the mode index of the resonances nearest the real axis and $\lambda_{\star}^{\pm}(\epsilon)=\lambda_{l_{\star}(\epsilon)}^{\pm}$ the resonance itself.  
 In Equation~(\ref {eq:smallepsres}) we display the leading order behavior of the $\Im\lambda_l^\pm(\epsilon)$, for $l$ not too large, and $\epsilon$ small:
\begin{equation} \Im\lambda_l^\pm(\epsilon)\ \sim\ 
-\,i\, \mathcal{M}(l,\epsilon)\label{lam-iM}\end{equation}
where
\beq \label{eq:impart}
\mathcal{M}(l,\epsilon) =  \frac{1}{2\epsilon} \left[(l+2)(l-1) \right]^{l+1} (l+1)^{l} \left[\frac{2^{l}l!}{(2l)!}\right]^{2}  \left( \frac{\epsilon}{\sqrt{\We}} \right)^{2(l+1)} 
\eeq

\noindent We attempt to compute the $l_\star(\epsilon)$ and $\Im\lambda^{\pm}_\star(\epsilon)\sim -\mathcal{M}\left(l_\star(\epsilon),\epsilon\right)$, by minimization of $\mathcal{M}(l,\epsilon)$ over $l$, treated as a continuous variable, while keeping $\epsilon\,\We^{-1/2}$ fixed. The calculation presented below in Remark~\ref{remark:minimize} suggests the following conjecture: 

%

\noindent {\it Longest lived mode index:}\ There exists $A,\epsilon_{*}>0$ such that for $\epsilon\le \epsilon_{*}$
\beq \label{eq:lopt}
l_{\star}(\epsilon) \approx A\, \epsilon^{-2}\, \We.
\eeq
{\it Real part and imaginary part (``reciprocal lifetime'')
  of longest lived resonance mode:}\\ 
There exists $r_{1},i_{1},i_{2}>0$ such that $\lambda_{\star}^{\pm}(\epsilon) = \Re\lambda_{\star}^{\pm}(\epsilon)\ +\ i\, \Im \lambda_{\star}^{\pm}(\epsilon)$, where
\begin{align}
\Re \left\{\lambda_{\star}^{\pm}(\epsilon)\right\} &\approx \pm \frac{r_{1}}{\epsilon} \left(\epsilon^{-2}\We\right),\qquad\ 
%
\Im \left\{\lambda_{\star}^{\pm}(\epsilon)\right\} \approx - \frac{i_{1}}{\epsilon} \left(\epsilon^{-2}\We\right) \, \exp\left( -i_{2}\, \epsilon^{-2}\, \We \right) \label{eq:im-lam-opt}
\end{align}
%
%
\begin{remark}\label{remark:contrast-rates}
 $\left|\Im\lambda^{\pm}_\star(\epsilon)\right|$, the distance from $\lambda^{\pm}_\star(\epsilon)$  to the real axis,  tends to zero  very rapidly (beyond all orders in $\epsilon$) as  $\epsilon\to0$. In contrast,  the radial  resonance, $\lambda_{0}(\epsilon)$, tends to the real axis only \emph{linearly} in $\epsilon$; see Theorem~\ref{thm:2smallep}.
\end{remark}

Numerical evidence supporting~\eqref{eq:lopt} and~\ref{eq:im-lam-opt} is obtained by numerically computing the resonances for fixed $\epsilon$ and various $l$, and determining the optimal mode index, $l_{\star}(\epsilon)$, and resonance $\lambda_{\star}^{numerical}(\epsilon)$.  The Mathematica function {\verb+FindRoot+} with options {\verb+AccuracyGoal -> 20+}, {\verb+PrecisionGoal -> 20+}, {\verb+WorkingPrecision -> 2000+}, and \verb+MaxIterations -> 1500+ were used to compute the resonances.
   Figure~\ref{fig:lopt-fit-0_03-0_20} shows the numerically computed optimal mode along with a numerical fit of the data. Note the expressions in Equations~(\ref {eq:lopt}) and~(\ref {eq:im-lam-opt}) are the leading terms of the numerical fit functions.
\begin{figure}[!htb] 
   \centering
   	\includegraphics[width=\textwidth]{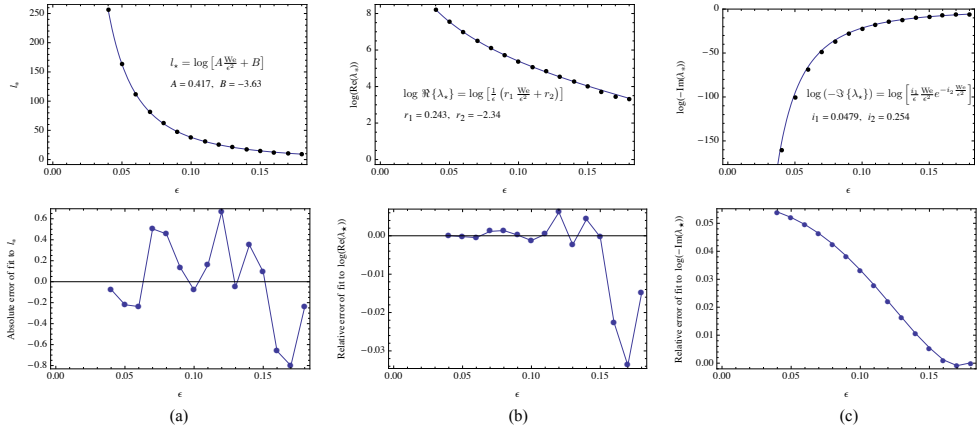}
   \caption{Numerical computations of \textup{(a)} $l_\star(\epsilon)$, angular momentum index of slowest decaying mode (left column), \textup{(b)} real part, $\Re\lambda_\star(\epsilon)$ (center) and \textup{(c)} imaginary  part, $\Im\lambda_\star(\epsilon)$ (right)  of the corresponding Rayleigh - deformation resonances, \smash[b]{$\lambda_{l_\star(\epsilon)}=\lambda_\star(\epsilon)$}.  Numerically computed values are indicated by dots. A  numerical fit (solid curve) to the functional form displayed in the inset of each plot is also given. }
   \label{fig:lopt-fit-0_03-0_20}
\end{figure}

\begin{remark}\label{remark:minimize}
Formally  minimizing $\mathcal{M}(l,\epsilon)$ over $l\ge2$ for fixed $\epsilon\We^{-{1\over2}}$, will approximately determine the longest lived mode  and its lifetime. The optimization can be  performed by treating $l$ as a continuous index and applying Stirling's formula, $n!\approx \sqrt{2\pi n}\ (n/e)^n,\ n\gg1$. This leads to a formal approximation of $l_\star(\epsilon)$, the angular momentum index of the most slowly decaying mode, and $\lambda^\pm_{l_\star(\epsilon)}=\lambda^\pm_\star(\epsilon)=\Re\lambda^\pm_\star(\epsilon) + i \Im\lambda^{\pm}_\star(\epsilon),\  \Im\lambda^{\pm}\star<0 $, the corresponding scattering resonance energies.
In particular, we obtain:
$
l_\star(\epsilon) \sim \frac{4}{e^3}\ \epsilon^{-2}\,\We$, $\ 
\Re\lambda^\pm_\star(\epsilon) \sim   \pm\frac{1}{\epsilon} \left( \frac{4}{e^3} \right)^{3/2}\left(\epsilon^{-2}\,\We\right)$, $\ 
\Im \lambda^\pm_\star(\epsilon) \sim - \frac{1}{\epsilon}\, \frac{4}{e^4} \left(\epsilon^{-2}\We\right) \, \exp\!\left( -\frac{4}{e^3}\, \epsilon^{-2}\, \We \right).$
\end{remark}
%
\section{Rayleigh deformation resonance analysis; Proof of Theorem~\ref{thm:2smallep}}\label{sec:rayleigh-res-analysis}

In this section we prove the assertions of 
Theorem~\ref{thm:2smallep} concerning the {\it Rayleigh resonances}, $\lambda^\pm_l(\epsilon),\ l=0,2,3,\dotsc$ appearing very close to the real axis for $\epsilon$ small.

The main ingredients of the proof of Theorem~\ref{thm:2smallep} are the following several propositions, which establish the detailed structure of the Taylor expansion of $\lambda^{2} + \rlhat G_{l}(\epsilon\lambda)=0$ for~$\lambda$ bounded and~$\epsilon$ sufficiently small. This will enable the explicit determination of the Taylor expansion of the Rayleigh resonances, $\lambda_l^\pm(\epsilon)$, and in particular their imaginary parts.

%
\begin{proposition} \label{thm:def-res-series}
 For $\epsilon$ small, the equation $\lambda^{2} + \rlhat G_{l}(\epsilon\lambda)=0$ can be written as
\beq
\lambda^{2} + \rlhat \Bigg[\underbrace{-(l+1) + \sum_{j=1}^{\infty} \frac{(\epsilon \lambda)^{2j}}{(2j)!}G_{l}^{(2j)}(0)}_{{\cal A}(\lambda)} +  \underbrace{\sum_{j=1}^{\infty} \frac{(\epsilon \lambda)^{2j+1}}{(2j+1)!}G_{l}^{(2j+1)}(0)}_{{\cal B}(\lambda)} \Bigg]  = 0,  \label{eq:def-res-series}
\eeq
where ${\cal A}(\lambda)$ is even and real for real  $\lambda$ and 
 ${\cal B}(\lambda)$ is odd and purely imaginary for real $\lambda$. 
\end{proposition}
\begin{proof}
This proposition follows by Taylor expansion of $G_l(z)$ about $z=0$ and the following:
\begin{claim} \label{thm:even-odd} 
For all $j=0,1,2,\dotsc$
\beq
\Im \bigset{G_{l}^{(2j)}(0)} = 0, \qquad \Re \bigset{G_{l}^{(2j+1)}(0)} = 0. \label{eq:even-odd}
\eeq
\end{claim}
This claim follows directly from the observation that 
\begin{equation}
G_{l}(z)\ {\rm  is\ real\ on\ the\ imaginary\ axis.}
 \label{it:Glziaxis}
 \end{equation}
To prove~\eqref{it:Glziaxis}, we consider $G_{l}(iy),\ y\in\R$. By Proposition~\ref{thm:Gl-poly}, it suffices to show that the ratio $p_{l}'(iy)/p_{l}(iy)$ is purely imaginary. By~\eqref{eq:poly-sum} we have 
\begin{align}
p_{l}(iy) &= \sum_{n=0}^{l} a_{n}^{l}(iy)^{n} 
= -i\ \sum_{n=0}^{l} \frac{(2l-n)!}{2^{l-n}(l-n)!n!} y^{n},\ \ {\rm which\ is\  purely\ imaginary}\nn\\
p_{l}'(iy) &=  -\ \sum_{n=0}^{l} \frac{(2l-n)!n}{2^{l-n}(l-n)!n!} y^{n-1},\ \ {\rm which\ is\  real}\ .
\end{align}
 The ratio is therefore purely imaginary, which implies that $G_l(iy)$ is real for $y$ real. \\

We now prove Claim~\ref{thm:even-odd} by induction. Consider the rational representation of $G_l$ in  Proposition~\ref{thm:Gl-poly}. As noted in~\eqref{hankrat-lim}, $G_l(0)=-(l+1)$. 
Moreover, for $y\in\R$
\begin{equation}
\D G_l(0) = \lim_{y \to 0} \frac{G_{l}(iy)-G_{l}(0)}{iy} =  -i \lim_{y \to 0} \frac{G_{l}(iy)-G_{l}(0)}{y}\ \in\ i\R,\nn
\end{equation}
 since $G_l(iy)\in\R$ for $y\in\R$.

 Next, we make the induction hypothesis:
 \begin{equation}
 \Im \left(\frac{d^{2j-2}G_{l}}{dz^{2j-2}}(0)\right) = 0,\ \ {\rm and}\ \ \Re \left(\frac{d^{2j-1}G_{l}}{dz^{2j-1}}(0)\right) = 0\ 
 \label{eq:ihype}\end{equation}
 and seek to prove the assertion~\eqref{eq:ihype} with $j$ replaced by $j+1$.
 
From the Taylor series of $G_l(z)$ we have:
\begin{equation}
\frac{d^{2j}G_{l}}{dz^{2j}}(0) = \lim_{y \to 0} \frac{(2j)!}{(iy)^{2j}}\!\! \left[G_{l}(iy) - \!\!\sum_{k=1}^{j} \frac{(iy)^{2k-2}}{(2k-2)!} \frac{d^{2k-2}G_{l}}{dz^{2k-2}}(0) -  
\!\!\sum_{k=1}^j \frac{(iy)^{2k-1}}{(2k-1)!}\frac{d^{2k-1}G_l}{dz^{2k-1}}(0)\!\right]\!.
\end{equation}
Thus, $\Im\ \D^{2j}G_l(0)=0$ is real by  induction hypothesis~\eqref{eq:ihype}. 

Similarly,  we have 
\begin{align}
\frac{d^{2j+1}G_{l}}{dz^{2j+1}}(0) = \lim_{y \to 0} \frac{(2j+1)!}{(iy)^{2j+1}} \!\left[G_{l}(iy) - \!\!\sum_{k=1}^{j} \frac{(iy)^{2k}}{(2k)!} \frac{d^{2k}G_{l}}{dz^{2k}}(0) -  
\!\!\sum_{k=1}^j \frac{(iy)^{2k-1}}{(2k-1)!}\frac{d^{2k-1}G_l}{dz^{2k-1}}(0)\right]
\end{align}
which is purely imaginary by induction hypothesis and since $\Im\ \D^{2j}G_l(0)=0$. 
This completes the proof of Claim~\ref{thm:even-odd} and thus Proposition~\ref{thm:def-res-series}.
\qquad\end{proof}
\begin{proposition}\label{proposition:qstar}
There is a positive integer
$2\le q_{*} < \infty$, for which Equation~(\ref {eq:def-res-series}) is of the form:
\beq
\lambda^{2} + \rlhat \Bigg[\underbrace{-(l+1) + \sum_{j=1}^{\infty} \frac{(\epsilon \lambda)^{2j}}{(2j)!}G_{l}^{(2j)}(0)}_{{\cal A}(\lambda)}\, + \, \underbrace{\phantom{\sum_{j=1}^{l}}\!\!\!\!\!\!\!\!\! \frac{(\epsilon \lambda)^{2q_*+1}}{(2q_*+1)!}G_{l}^{(2q_*+1)}(0) + \bigO\left( \epsilon^{2q_*+3} \right)}_{ {\cal B}(\lambda)}\Bigg] = 0,
\eeq
where ${\cal A}(\lambda)$ is even and real-valued for real $\lambda$ and ${\cal B}(\lambda)$ is odd and imaginary-valued for 
real $\lambda$.
\end{proposition}

\emph{Proof of Proposition~\ref{proposition:qstar}}.
We begin by observing that  $G_l(z)$ is {\it not} an even function of $z$.  This is a direct consequence of the asymptotic behavior 
 of $G_l(x)$ as $x\to\pm\infty$: 
\begin{align}
G_{l}(x) &\approx -(l+1)+ix + l = -1 + ix,\ \ x\gg1\nn\\
G_{l}(-x) &\approx -(l+1)-ix - l = -(2l+1) -ix,\ \ x\gg1.
\end{align}
Therefore, there is a first non-zero odd derivative, Taylor coefficient. We call the order of this coefficient, $q_*$.
The next result  asserts that $q_*=l$. 
\qquad\endproof

\begin{proposition} \label{thm:imagterm}
\beq
G_{l}^{(2k+1)}(0) = 
\begin{cases}
	i\ (2l+1)! \left(\frac{2^{l}l!}{(2l)!}\right)^{2}, & k=l, \\
	0, & 0\le k\le l-1.
\end{cases}
\label{imag-Gl2kp1}\eeq
\end{proposition}

\emph{Proof of Proposition~\ref{thm:imagterm}}. 
Since we are interested in the odd derivatives of $G_l$ at the origin we calculate:
\begin{align}
G_{l}(z) - G_{l}(-z) &= -(l+1) + iz +  \frac{z p_{l}'(z)}{p_{l}(z)} +(l+1) + iz - \frac{-z p_{l}'(-z)}{p_{l}(-z)} \nn  \\
&= 2iz + z \left[ \frac{ p_{l}'(z)}{p_l(z) } + \frac{ p_{l}'(-z)}{p_l(-z) }\right] \nn  \\
&= z \left[\frac{  2i p_{l}(z)p_{l}(-z) + p_{l}'(z)p_{l}(-z) + p_{l}'(-z)p_{l}(z) }{p_{l}(z)p_{l}(-z) } \right].   
\end{align}
By~\eqref{eq:poly-sum}, $a_l^l= i^{-l-1}$ and therefore
$
2i p_{l}(z)p_{l}(-z) = 2i a_l^l z\ a_l^l (-z)^l  + \cO(z^{2l-1})= -2iz^{2l}
 +  \cO(z^{2l-1}).
$
This suggests separating out this term and writing
\begin{equation}
G_{l}(z) - G_{l}(-z) = z \left[ \frac{ -2iz^{2l} + R_{l}(z)}{p_{l}(z)p_{l}(-z) } \right],\label{Gl-odd}
\end{equation}
where 
\begin{equation}
R_{l}(z) \equiv 2i p_{l}(z) p_{l}(-z) \ +\  p_{l}'(z) p_{l}(-z) + p_{l}'(-z) p_{l}(z)\ -\ \left( -2i z^{2l} \right)
\label{Rldef}
\end{equation}
The expression~\eqref{Gl-odd} simplifies due to a subtle cancellation which follows from
\begin{lemma} \label{thm:wronskian}
\beq  \label{eq:zero}
R_{l}(z) = 0
\eeq
and therefore
\begin{equation}
G_{l}(z) - G_{l}(-z) = -2iz^{2l+1} \left[ \frac{ 1 }{p_{l}(z)p_{l}(-z) } \right]
\nn\end{equation}
\end{lemma}
We prove~\eqref{eq:zero} below, but first note its consequences.
 For $z$ small we have
\begin{align}
G_{l}(z) - G_{l}(-z) &= \frac{ -2iz^{2l+1}}{\big(a_{0}^{l}\big)^{2} + \cO(z)}  = 2i \left[ \frac{2^{l}l!}{(2l)!} \right]^{2}  z^{2l+1}  +  \cO\bigl( z^{2l+3} \bigr)  \label{eq:oddGl}
\end{align}
where we have  used that
\beq
a_{0}^{l}= -i \frac{ (2l)!}{2^{l} l! } \nn .
\eeq
Since
\beq
G_{l}(z) - G_{l}(-z) = 2 \sum_{k=0}^{\infty} \frac{ G_{l}^{(2k+1)}(0)}{(2k+1)! } z^{2k+1}, 
\eeq
matching terms with Equation~(\ref {eq:oddGl}) implies $q_*=l$
 and Equation~(\ref {imag-Gl2kp1}). 

To complete the proof of Proposition~\ref{thm:imagterm} we now give the proof of Lemma~\ref{thm:wronskian}. The idea is to re-express the polynomial expression for $R_l(z)$ in terms of spherical Hankel functions and to achieve simplification by using 
 the Wronskian identity governing $h_l^{(1)}$ and $h_l^{(2)}$.  
Recall
\begin{align}
p_{l}(z)  &=  z^{l+1} h_{l}^{(1)}(z)  e^{-iz}  \\
p_{l}'(z)  &=  z^{l}\left[ z h_{l}^{(1)}{'}(z) + (l+1-iz) h_{l}^{(1)}(z) \right] e^{-iz}.
\end{align}
Substituting into~\eqref{Rldef} yields
\begin{align}
R_{l}(z) &= 2i \left[ z^{l+1} h_{l}^{(1)}(z) e^{-iz} (-1)^{l+1} z^{l+1} h_{l}^{(1)}(-z) e^{iz}  + z^{2l} \right]  \nn \\
&\quad{}+  z^{l} \left[ z \D h_{l}^{(1)}(z)  + (l+1-iz) h_{l}^{(1)}(z)  \right]  e^{-iz} (-1)^{l+1} z^{l+1} h_{l}^{(1)}(-z)  e^{iz} \nn \\
&\quad{}+  (-1)^{l} z^{l}  \left[ -z \D h_{l}^{(1)}(-z) + (l+1+iz) h_{l}^{(1)}(-z) \right] e^{iz} z^{l+1} h_{l}^{(1)}(z) e^{-iz}  \nn \\
&= 2i z^{2l} \left[ (-1)^{l+1}z^{2} h_{l}^{(1)}(z)  h_{l}^{(1)}(-z)  + 1 \right]   \nn \\
&\quad{}+ (-1)^{l+1} z^{2l}  \left[ z^{2} \D h_{l}^{(1)}(z) h_{l}^{(1)}(-z)  + (\cancel{l+1}-iz) z h_{l}^{(1)}(z) h_{l}^{(1)}(-z) \right]  \nn \\
&\quad{}+ (-1)^{l+1} z^{2l}  \left[ z^{2} \D h_{l}^{(1)}(-z) h_{l}^{(1)}(z)  - (\cancel{l+1}+iz) z h_{l}^{(1)}(z) h_{l}^{(1)}(-z) \right]   \nn \\
&= 2i z^{2l}  \left[ \cancel{(-1)^{l+1} z^{2} h_{l}^{(1)}(z)  h_{l}^{(1)}(-z)}  + 1\right]  \nn \\
&\quad{}+\! (-1)^{l+1} z^{2l} \! \left[ z^{2} \!\left( \D h_{l}^{(1)}(z) h_{l}^{(1)}(-z) \!+ \!\D h_{l}^{(1)}(-z) h_{l}^{(1)}(z)  \right) \!- \cancel{2i z^{2} h_{l}^{(1)}(z) h_{l}^{(1)}(-z)} \right]   \nn \\
&= z^{2l}  \left[ 2i + (-1)^{l+1} z^{2}  \left( \D h_{l}^{(1)}(z) h_{l}^{(1)}(-z) + \D h_{l}^{(1)}(-z) h_{l}^{(1)}(z)  \right) \right]  =\ z^{2l} \left[\ 0\ \right] = 0\ .\label{longcalc}
\end{align}
In the last line we used a Wronskian relation 
$h_{l}^{(1)}(z) \D h_{l}^{(2)}(z)  -  h_{l}^{(2)}(z) \D h_{l}^{(1)}(z)  = -2iz^{-2}
$
with the analytic continuation formulae~\cite{Abramowitz:1965zr} 
$
h_{l}^{(2)}(z)= (-1)^{l} h_{l}^{(1)}(-z)$ and $\D h_{l}^{(2)}(z)=(-1)^{l+1} \D h_{l}^{(1)}(-z)$ . 
  This completes the proof of the key Lemma~\ref{thm:wronskian} and thus Proposition~\ref{thm:imagterm}.\qquad\endproof

We now complete of the proof of Theorem~\ref{thm:2smallep}.
 Recall that deformation scattering resonance energies satisfy the 
equation $\lambda^{2} + \rlhat G_{l}(\epsilon \lambda) = 0$. In the 
(incompressible) limit, $\epsilon\to0$, we have the equation
 $\lambda^2+\rlhat G_l(0)=0$, which has two real solutions,
  the Rayleigh frequencies, to which there are associated undamped oscillatory modes of the incompressible linearized equations about the spherical bubble equilibrium. 
  
  We are interested in where these Rayleigh frequencies perturb for positive and small $\epsilon$. By Theorem~\ref{thm:def-res-general}, they necessarily perturb to the open lower half plane. Propositions
~\ref{thm:def-res-series} and~\ref{thm:imagterm}  enable us to compute, in detail, the location of the {\it Rayleigh scattering resonances} to which these real frequencies perturb.
\\ \\
We first consider the case: $l=0$. 
Using properties of  Hankel functions (see Appendix~\ref {app:polyhank})
  we have $G_{0}(\epsilon\lambda)=i\epsilon\lambda-1$
  and the equation for the resonance energies is:
$
\lambda^{2}+ i\epsilon \rhat_{0} \lambda-\rhat_{0}=0$
yielding the solution in Equation~(\ref {lambda-eqn}). This completes the proof of Part 1 of Equation~(\ref {lambda-eqn}).
\\ \\
 For general $l\ge2$, we can, by Propositions~\ref{proposition:qstar} and~\ref{thm:imagterm}, write the resonance equation in the form:
\beq
\lambda^{2} + \rlhat \Biggl[\underbrace{-(l+1) + \sum_{j=1}^{l} \frac{(\epsilon \lambda)^{2j}}{(2j)!}G_{l}^{(2j)}(0)}_{{\cal A}(\lambda)}\, +\,  \underbrace{\phantom{\sum_{j=1}^{l}}\!\!\!\!\!\!\!\!\! \frac{(\epsilon \lambda)^{2l+1}}{(2l+1)!}G_{l}^{(2l+1)}(0)}_{{\cal B}(\lambda)} + \bigO\Bigl(\epsilon^{2l+2}\Bigr)\Biggr]  = 0, \label{eq:lambdaseries}
\eeq
where ${\cal A}(\lambda)$ is even and real-valued for real $\lambda$ and ${\cal B}(\lambda)$ is odd and purely imaginary for real $\lambda$ and $\rlhat>0$ is given in~\eqref{eq:rlhat-notation}.
It is clear from equation~(\ref {eq:lambdaseries}) that the lowest order term contributing an imaginary part is at~$\cO(\epsilon^{2l+1})$.
 Taking $2l+1$ derivatives in $\epsilon$ and setting $\epsilon=0$ yields
$
2 \lambda(0) \lambda^{(2l+1)}(0) + \rlhat  G_{l}^{(2l+1)}(0) \lambda(0)^{2l+1} =0.
$
Again, using Proposition~\ref{thm:imagterm}, we find 
$
\lambda^{(2l+1)}(0) = - \frac{\rlhat}{2} G^{(2l+1)}(0) \lambda(0)^{2l} \ \ {\rm and}
$
\begin{align}
\Im \lambda^{(2l+1)}(0) &= - \frac{\rlhat}{2} \left[ \rlhat(l+1) \right]^{l} (2l+1)! \left[\frac{2^{l} l!}{(2l)!}\right]^{2}
\! =  - 2^{2l-1} \rlhat^{l+1} (l+1)^{l} (2l+1)! \left[\frac{l!}{(2l)!}\right]^{2}\!.\label{eq:lambda-2l+1}
\end{align}
This proves Part 2 of Theorem~\ref{thm:2smallep}.
%
\section{Estimation of deformation resonances for \texorpdfstring{$l\ge K_{1}\epsilon^{-6-\delta},\ \delta>0$}{large $l$};\\ Proof of \texorpdfstring{Theorem~\ref{thm:big-l}}{\ref{thm:big-l}} }\label{sec:def-res-pfs}

In Theorem~\ref{thm:def-res-general}, we established that there are $l+2$ deformation resonances in the lower half plane. These are solutions of $\lambda^2+\rlhat G_l(\epsilon\lambda)=0$; see
Equations~(\ref {eq:rlhat-notation}) and~(\ref {eq:deformation-eqn})  or
\begin{align}
\frac{\left( \epsilon\lambda \right)^{2}}{\epsilon^{2} {\We}^{-1}(l+2)(l-1)} \;h_{l}^{(1)}\left( \epsilon\lambda \right) + \epsilon\lambda \;\D h_{l}^{(1)}\left( \epsilon\lambda \right) &=0.
\intertext{Equivalently, in terms of standard Hankel functions: }
\left(\frac{(\epsilon\lambda)^{2}}{\epsilon^{2}\We^{-1}\left(\nu^{2}-(3/2)^{2}\right)} - \frac{1}{2}\right)  H_{\nu}^{(1)}(\epsilon\lambda) + \epsilon\lambda \;\D H_{\nu}^{(1)}(\epsilon\lambda) &= 0,\ \qquad \nu=l+1/2.  
\label{Hnu-form}\end{align}

In this section, we estimate the location of these resonances in the parameter regime where the \emph{second} term dominates.
In Section~\ref{sec:l+1res}, we will show that $l+1$ deformation resonances are near the~$l+1$ zeros of $\D H_{\nu}^{(1)}(\epsilon\lambda)$ by applying Rouch\'e's Theorem. See Appendix~\ref {app:hankel-zeros} for the locations of the zeros of Hankel functions.  In Section~\ref{sec:1res}, the remaining resonance is found by the method of dominant balance.

\medskip

The proof of Theorem~\ref{thm:big-l} is based on an application of  Rouch\'e's Theorem~\cite{Ahlfors}: 
\ {\it Suppose $f$ and $g$ are analytic inside and on a simple, closed curve $\mathcal{C}$ and satisfy
$ \abs{g} < \abs{f}$ on $ \mathcal{C}$.
Then $f$ and $f+g$ have the same number of zeros (counting multiplicities) inside $\mathcal{C}$.}

\subsection{Proof of  Theorem~\ref{thm:big-l}, Part 1} \label{sec:l+1res}
Let
\beq \label{eq:zdef}
\nu z = \epsilon\lambda   
\eeq
and rewrite~\eqref{Hnu-form} as
\beq \label{eq:fplusg}
\underbrace{\phantom{\bigg(}\!\!\!\!  \D H_{\nu}^{(1)}(\nu z)  \!\!\!\!\phantom{\bigg)}}_{f(z)}  +  \underbrace{\frac{1}{\nu z}   \left(\frac{(\nu z)^{2}}{\epsilon^{2}\We^{-1}\left(\nu^{2}-9/4\right)} - \frac{1}{2}\right) H_{\nu}^{(1)}(\nu z)}_{g(z)} = 0,\ \qquad \nu=l+1/2.
\eeq
We control the zeros of~\eqref{eq:fplusg} in terms of the zeros of $\D H_\nu^{(1)}$ for $l$ sufficiently large.
Here, $f$ and $g$ are defined for the application of Rouch\'e's Theorem to the locations of all $l+2$ resonances.  

We prove that the $l+1$ roots of Equation~(\ref {eq:fplusg}), denoted $z_{\nu,1},z_{\nu,2},\dotsc,z_{\nu,l+1}$,  are asymptotically ($l$ large) close  to zeros of $\D H_{\nu}^{(1)}(\nu z)$. This implies, via~\eqref{eq:zdef} that there are resonances at the locations:
\begin{equation}
\lambda_{l,j}(\epsilon)=\nu z_{\nu,j}/\epsilon
\label{lam2z}
\end{equation}

Let $y_{\nu,j}'$ denote a solution of $\D H_{\nu}^{(1)}(y)=0$ in the fourth quadrant; see Equation~(\ref {eq:dH-zero}).  We claim that $(f+g)(z)$ has a unique simple zero, $z_{\nu,j}$, in the interior of $\gamma_{z}$, a small curve about the point $z=\nu^{-1}y_{\nu,j}'$.  The curve $\gamma_{z}$ will be the image of a small circle defined in~\eqref{eq:big-l-rouche-p3}.
 To prove this, we show that 
for all $z\in\gamma_{z}$,
\beq
\bigabs{\frac{g(z)}{f(z)}} = \bigabs{\frac{1}{\nu z}   \left(\frac{(\nu z)^{2}}{\epsilon^{2} \We^{-1}\left(\nu^{2}-9/4\right)} - \frac{1}{2}\right) \frac{H_{\nu}^{(1)}(\nu z)}{\D H_{\nu}^{(1)}(\nu z)}} < 1,
\label{eq:big-l-rouche}
\eeq
for $\nu\ge K_1 \!\left(\epsilon^{-2}\We\right)^{3/(1-3q)},$ where $0<q<1/3$.
 
The key estimate,~\eqref{eq:big-l-rouche}, is a consequence of the following two claims. Below, we use $C$ to denote a generic positive constant, independent of $\nu, \epsilon, \We$.
\begin{claim} \label{claim:factors}
\begin{romannum}
\item Choose  $\epsilon>0$ and $0< c_{1} < c_{2}$. Assume  $c_{1} \le \abs{z} \le c_{2}$. Then, there exists $C>0$, depending only on $c_1$ and $c_2$, and  $\nu_0=\cO(1)>0$ such that  
\beq \label{eq:big-l-rouche-p1}
\bigabs{\frac{1}{\nu z}  \left(\frac{(\nu z)^{2}}{\epsilon^{2}\We^{-1}\left(\nu^{2}-9/4\right)} - \frac{1}{2}\right)} \le \frac{C}{\epsilon^2 \We^{-1}\;\nu},\ \ \nu\ge\nu_0.
\eeq
\item There exists $\nu_{0}>0$ and $c_{3}>0$ such that for all $\nu\ge \nu_{0}$ and $\abs{\arg z} < \pi$
with $\abs{z}\ge c_{3}$,
\beq \label{eq:big-l-rouche-p2}
\bigabs{\frac{H_{\nu}^{(1)}(\nu z)}{\D H_{\nu}^{(1)}(\nu z)} } \le C \nu^{1/3} \bigabs{\frac{Ai(\eta)}{\D Ai(\eta)}},\ \ \ where
\eeq
\begin{align}
\eta &= e^{2\pi i/3}\nu^{2/3}\zeta,\qquad  
\frac{2}{3} \zeta^{3/2} = \int_{z}^{1}\,\frac{\sqrt{1-t^{2}}}{t}\,dt = \log \frac{1+\sqrt{1-z^{2}}}{z} - \sqrt{1-z^{2}}\ ;\  \label{eq:zetadef}
\end{align}
see also Sections~\ref{app:unifhankel} and~\ref{app:hankel-zeros}.
\end{romannum}
\end{claim}
\begin{claim} \label{thm:big-l-rouche}
Let $q\in(0,1/3)$ and denote by $\eta_{s}'$ the $s$th solution of $\D Ai(\eta)=0$; see Section~\ref{sec:airy0}.\ 
There exists $\nu_{1}>0,\ \kappa>0$ such that if $\nu\ge\nu_{1}$ and  $s\in\set{ 1,2,\dotsc, \floor{(l+1)/2}+1 }$,  then 
\beq \label{eq:big-l-rouche-p3}
\bigabs{\frac{Ai(\eta)}{\D Ai(\eta)}} \le  C \nu^{1/3+q}, \qquad \text{on the circle } C_{\eta}=\bigset{\eta:\, \bigabs{\eta-\eta_{s}'} = \frac{\kappa}{2} s^{-1/3} \nu^{-q}},
\eeq
chosen to enclose exactly one zero of $\D Ai$; see Equation~\eqref{eq:big-etap-spacing}. 
\end{claim}

Inequality~\eqref{eq:big-l-rouche}  is a direct consequence of~\eqref{eq:big-l-rouche-p1},~\eqref{eq:big-l-rouche-p2}, and~\eqref{eq:big-l-rouche-p3} of Claims~\ref{claim:factors} and~\ref{thm:big-l-rouche}.


Assuming \expandafter \MakeUppercase claims~\ref {claim:factors} and~\ref {thm:big-l-rouche}, we now prove the assertion of Part 1 of Theorem~\ref{thm:big-l}, concerning $l+1$ of the $l+2$ resonances. (The $(l+2)$nd is covered by a separate argument below; see Section~\ref{sec:1res}.)

First, fix the parameters $q$, $\nu_{1}$, and $\kappa$ as in Claim~\ref{thm:big-l-rouche} and choose $\check{\eta}\in C_{\eta}$, so that the upper bound in~\eqref{eq:big-l-rouche-p3} holds.    For such $\check{\eta}$, by the change of variables~\eqref{eq:zetadef}, $z(\check{\eta})$ varies in an annulus $0<c_1\le |z|\le c_2$, where $0<c_1< c_2<\infty$ depend on $\kappa$ and $\nu_{1}$.

By~\eqref{eq:big-l-rouche-p3}, we write $\check{\eta}\in C_{\eta}$ as
\beq
\check{\eta} = \eta_{s}' + \frac{\kappa}{2} s^{-1/3} \nu^{-q} e^{i\vartheta},\ \qquad
0\le\vartheta<2\pi.
\eeq
Under the mapping $\eta\mapsto \zeta$,~\eqref{eq:zetadef}, the circle $C_{\eta}$ is transformed into the circle $C_{\zeta}$: 
\beq \label{eq:C-zeta}
\zeta(\check{\eta}) = \nu^{-2/3}e^{-2\pi i /3}\left(\   \eta_{s}' +  \frac{\kappa}{2} s^{-1/3} \nu^{-q} e^{i\vartheta} \right) \in C_{\zeta}.
\eeq
 Finally, we map $C_{\zeta}$ to the curve $\gamma_{z}$ in the $z$-plane using a formula from~\cite{Olver:1954lr}
\beq 
z(\zeta) = 1 - 2^{-1/3} \,\zeta + \tfrac{3}{10} 2^{-2/3}\, \zeta^{2} + \tfrac{1}{700}\, \zeta^{3} + \zeta^{4} \bigO(1),\ \ \ \zeta\ {\rm small}. \tag{\ref{eq:zzeta}}
\eeq
 Thus, to leading order, $\gamma_{z}$ is a still a circle:
\beq \label{eq:gamma-z}
\gamma_{z} = z(C_{\zeta}) = 1 - 2^{-1/3}\, C_{\zeta} + \bigO\!\left( \left[ \eta_{s}' \nu^{-2/3} \right]^{2} \right)
\eeq
with radius the same as that of $2^{-1/3}\,C_{\zeta}$, namely $2^{-4/3}\kappa \, s^{-1/3} \nu^{-2/3-q}$. Since this radius is decreasing in both $s$ and $\nu$, it can be bounded uniformly.  Using this uniform bound, we can find fixed $c_{1}$ and $c_{2}$ defining an annulus which encloses $\gamma_{z}$ for all values of $s$ and $\nu$. 

Additionally, if we take $y_{\nu,s}'$ to be a zero of $\D H_{\nu}$ corresponding to $\eta_{s}'$, and approximated by Equation~(\ref {eq:dH-zero}) with $k=s$, then $\nu^{-1} y_{\nu,s}'$ is inside the curve $\gamma_{z}$ defined by Equation~(\ref {eq:gamma-z}).  It is equivalent to show that when mapped into the $\zeta$  variable, the image of $\nu^{-1} y_{\nu,s}'$ lies in $C_{\zeta}$. From Equation~(\ref {eq:dH-zero}), we see the image of $\nu^{-1} y_{\nu,s}'$ lies in a disk of radius $\sim\nu^{-4/3}$ which is smaller than the radius of $C_{\zeta}$.
With this choice of $c_1$ and $c_2$, we have, using the definitions in~\eqref{eq:fplusg} and \expandafter \MakeUppercase claims~\ref {claim:factors} and~\ref {thm:big-l-rouche}, that 
\begin{equation}
\left|\ \frac{g(z)}{f(z)}\ \right|\ \le\ C\ \frac{\We}{\epsilon^2 l}\times l^{1\over3}\  l^{\frac{1}{3}+q}, \nn
\end{equation}
Therefore, there exists $K_1>0$ such that for any $q\in(0,1/3)$, if 
$
l\ge K_1 \!\left(\epsilon^{-2}\We\right)^{3/(1-3q)}  
$
then $ \abs{f(z)/g(z)} < 1$ on the required curve $\gamma_{z}$. The largest range of $l$-values is obtained by taking $q=\delta>0$ and small and thus $l\ge K_1 \!\left(\epsilon^{-2}\We\right)^{3+\cO(\delta)}$.

Since we have shown $\D H_\nu^{(1)}(\nu z)=0$ has one zero inside $\gamma_{z}$, by Rouch\'e's Theorem there is also a single solution, $z_{\nu,s}$, of~\eqref{eq:fplusg}, inside $\gamma_{z}$.
 To finish the proof of Theorem~\ref{thm:big-l} Part 1, we compute $z_{\nu,s}$ and map it to the $\lambda$-plane.  Since $z_{\nu,s}$ is inside $\gamma_{z}$, we use Equation~(\ref {eq:zzeta}) evaluated at~\eqref{eq:C-zeta} for the approximation.
We find:
\beq
z_{\nu,s} = 1 + 2^{-1/3} e^{-2\pi i/3} (l+1/2)^{-2/3}  \bigabs{\eta_{s}'} + \tfrac{3}{10} 2^{-1/3} e^{-4\pi i /3} \nu^{-4/3}\bigabs{\eta_{s}'}^{2} + \bigO\Bigl(\! s^{-1/3} \nu^{-2/3-q}\! \Bigr).
\eeq

Mapping back to $\lambda$, using~\eqref{lam2z} completes the proof.

\emph{Proof of Claim~\ref{claim:factors}}.
Estimate~\eqref{eq:big-l-rouche-p1} is straightforward. 
To prove  estimate~\eqref{eq:big-l-rouche-p2} we use the following leading order asymptotics of  Hankel functions, uniform in $z$ and valid for large $\nu$  (see Appendix~\ref {app:unifhankel})
\begin{alignat}{2}
H_{\nu}^{(1)}(\nu z) &\sim 2 e^{-\pi i/3} \left(\frac{4\zeta}{1-z^{2}}\right)^{1/4} \frac{Ai(\eta)}{\nu^{1/3}}, &\quad \abs{\arg z} < \pi, \tag{\ref{eq:Hnu-nuz}} \\
\D H_{\nu}^{(1)}(\nu z) &\sim -4 e^{\pi i/3} \frac{1}{z} \left(\frac{1-z^{2}}{4\zeta}\right)^{1/4} \frac{\D Ai(\eta)}{\nu^{2/3}}, &\abs{\arg z} < \pi. \tag{\ref{eq:dHnu-nuz}}
\end{alignat}
In particular, for any fixed $0<\delta<1$ and sufficiently large  $\nu$ 
\beq
\bigabs{\frac{H_{\nu}^{(1)}(\nu z)}{\D H_{\nu}^{(1)}(\nu z)} }\ \le\ \frac{1+\delta}{1-\delta}\  \nu^{1/3}\  \bigabs{ z \left(\frac{\zeta}{1-z^{2}}\right)^{1/2}\frac{Ai(\eta)}{\D Ai(\eta)}}.
\eeq
To show this is bounded near $z=1$, we use the Taylor series of $\zeta(z)$ about $z=1$, as given in Appendix~\ref {app:unifhankel}:
\begin{align}
\zeta &= - e^{\pi i/3} 2^{1/3} (z-1) \left[ 1 + \bigO \left( z-1 \right) \right],
\end{align}
so, due to the boundedness of $z$,
$
\bigabs{ z \left(\frac{\zeta}{1-z^{2}}\right)^{1/2}} \le C.
$
This completes the proof of~\eqref{eq:big-l-rouche-p2}, and therewith Claim~\ref{claim:factors}.
\qquad\endproof

\emph{Proof of Claim~\ref{thm:big-l-rouche}}.
Since $\D Ai(\eta_s')=0$, Taylor expansion of $Ai$ and $\D Ai$ about $\eta=\eta_s'$ implies the existence of analytic functions  $A_{1}$ and $A_{2}$, such that 
\begin{align}
Ai(\eta) &= Ai(\eta_{s}') + A_{1}(\eta) (\eta-\eta_{s}')^{2} 	\label{eq:A1def}  \\
\D Ai(\eta) &= \eta_{s}' Ai(\eta_{s}') (\eta-\eta_{s}') + A_{2}(\eta) (\eta-\eta_{s}')^{2}.		\label{eq:A2def}
\end{align}
To obtain~\eqref{eq:A2def} we also used that $\D^2 Ai(\eta)=\eta Ai(\eta)$.  It follows that
\begin{align}
\frac{Ai(\eta)}{\D Ai(\eta)} &= \frac{Ai(\eta_{s}')+ A_{1}(\eta)(\eta-\eta_{s}')^{2}}{\eta_{s}' Ai(\eta_{s}')(\eta-\eta_{s}') + A_{2}(\eta)(\eta-\eta_{s}')^{2}} 
%
 \equiv \frac{1}{\eta_{s}'(\eta-\eta_{s}')}\, a_{\nu}'(\eta)  \label{eq:anup}
\end{align}
Thus, we now estimate $a_{\nu}'(\eta)$. Using the domain specified in Equation~(\ref {eq:big-l-rouche-p3}), we have the bound
\beq
\bigabs{\frac{1}{\eta_{s}'(\eta-\eta_{s}')} } \le C s^{1/3} \nu^{q}.
\eeq
We complete the proof of Claim~\ref{thm:big-l-rouche} 
  by showing $\abs{a_{\nu}'(\eta)}\le C$ on a circle around $\eta_{s}'$ that is small enough such that it encloses no other $\D Ai(\eta)$ zero. The specific circle used is defined by
\begin{equation}
\bigabs{\eta-\eta_{s}'} = C s^{-1/3} \nu^{-q} \nn
\end{equation}
and is determined by the spacing between consecutive zeros of $\D Ai(\eta)$; see Equations~(\ref {eq:big-etap-spacing}) and~(\ref {eq:small-etap-spacing}). There are two cases to study:
%
\begin{remunerate}
\item[$\bullet$]
First, consider the case $s\ge s_{1}' \ge s_{0}'\ge 1$ ($s_{0}'$ and $s_{1}'$ independent of $l$). Using~\eqref{eq:DAiryzeros}, giving the locations of the larger $\D Ai(\eta)$ zeros, $\set{\eta_{s}'|\, s\ge s_{0}'}$, we estimate the spacing between the zeros as
\begin{align}     	
\bigabs{\eta_{s+1}'-\eta_{s}'} &= \tfrac{2}{3} \left( \tfrac{3\pi}{2} \right)^{2/3} s^{-1/3} + \bigO\!\left( s^{-4/3} \right) \\
\intertext{so that the spacing satisfies the lower bound}
\bigabs{\eta_{s+1}'-\eta_{s}'} &\ge \kappa \, s^{-1/3}, \quad s\ge s_{0}'  \label{eq:big-etap-spacing}
\end{align}
where $\kappa = (3\pi/2)^{2/3}/3\approx0.93693$. So, first, we consider the open disk $\abs{\eta-\eta_{s}'} < \frac{\kappa}{2} s^{-1/3}$, which clearly contains only one zero of $\D Ai$.

Now we estimate the $A_1$ and $A_2$, appearing in~\eqref{eq:anup}. Using the Cauchy integral formula and $A_{1}(\eta)$ obtained from~\eqref{eq:A1def} we have
\begin{align}
A_{1}(\eta) &= \frac{1}{2\pi i} \oint_{\abs{t-\eta_{s}'}=\kappa s^{-1/3}} \frac{A_{1}(t)}{t-\eta} \, dt   \ =\ \frac{1}{2\pi i} \oint_{\abs{t-\eta_{s}'}= \kappa s^{-1/3}} \frac{Ai(t)-Ai(\eta_{s}')}{(t-\eta_{s}')^{2}(t-\eta)} \, dt.
\end{align}

Using the asymptotic expansion of $Ai(\eta)$ on the negative real axis from Appendix~\ref{app:airy} and the location of the $\eta_{s}'$ from~\eqref{eq:Airyzeros} we have
\beq
C_{0} s^{-1/6} \le \sup_{\abs{\eta-\eta_{s}'} \le \kappa s^{-1/3}} \abs{Ai(\eta)} \le C_{1} s^{-1/6}, \quad s\ge s_{1}'.   \label{eq:Aibound}
\eeq 
This allows the estimate for $\abs{\eta-\eta_{s}'} < \kappa s^{-1/3}$:
\begin{align*}
\abs{A_{1}(\eta)} &\le \frac{1}{2\pi} \sup_{\abs{\bar{\eta}-\eta_{s}'}= \kappa s^{-1/3}} \frac{2\abs{Ai(\bar{\eta})}}{(\kappa s^{-1/3})^{2}} \oint_{\abs{t-\eta_{s}'} = \kappa s^{-1/3}} \frac{1}{\abs{t-\eta}} \, \abs{dt}, \\
&\le C \frac{s^{-1/6}}{(\kappa s^{-1/3})^{2}} \oint_{\abs{t-\eta_{s}'}=\kappa s^{-1/3}} \frac{1}{\abs{t-\eta}} \, \abs{dt}, \;\,\qquad\qquad\text{by equation~(\ref {eq:Aibound})}, \\
&\le C \frac{s^{-1/6}}{(\kappa s^{-1/3})^{2}} \; \frac{1}{\kappa s^{-1/3}-\abs{\eta-\eta_{s}'}} \oint_{\abs{t-\eta_{s}'}=\kappa s^{-1/3}}\hspace{-3em} 1 \, \abs{dt}, \qquad \text{by the triangle inequality,} \\
&\le C \frac{s^{-1/6}}{\kappa s^{-1/3}} \;  \frac{1}{\kappa s^{-1/3}-\abs{\eta-\eta_{s}'}}.
\end{align*}

Next, we apply the last inequality to the \emph{smaller} circle $\abs{\eta-\eta_{s}'}=C s^{-1/3} \nu^{-q}$ we can estimate 
\begin{align}
\abs{A_{1}(\eta)} &\le C \frac{s^{-1/6}}{(\kappa s^{-1/3})^{2}} \; \frac{1}{1-\abs{\eta-\eta_{s}'}/(\kappa s^{-1/3})}
\ \le\  C s^{1/2} \; \frac{1}{1-C\nu^{-q}}\ \le\ C s^{1/2}.\label{eq:A1bnd}
\end{align}
Similarly,
\beq
\abs{A_{2}(\eta)} \le C s^{5/6}.\label{eq:A2bnd}
\eeq

The full numerator of~\eqref{eq:anup} can be estimated, using~\eqref{eq:Aibound} and~\eqref{eq:A1bnd},
\begin{align*}
\bigabs{1 +  Ai(\eta_{s}')^{-1} A_{1}(\eta)(\eta-\eta_{s}')^{2}} &\le 1 + \bigabs{Ai(\eta_{s}')}^{-1} \abs{A_{1}(\eta)} \bigabs{\eta-\eta_{s}'}^{2}\\ 
&\le 1 + C s^{1/6} \, Cs^{1/2} \left( Cs^{-1/3}\nu^{-q} \right)^{2} \le 1 + C \nu^{-2q}.
\end{align*}
For the denominator, using~\eqref{eq:DAiryzeros},~\eqref{eq:Aibound} and~\eqref{eq:A2bnd},
\begin{align*}
\bigabs{\left[ \eta_{s}' Ai(\eta_{s}')\right]^{-1} A_{2}(\eta)(\eta-\eta_{s}')} &\le \bigabs{\eta_{s}'}^{-1} \bigabs{Ai(\eta_{s}')}^{-1} \bigabs{A_{2}(\eta)} \bigabs{\eta-\eta_{s}'}\\
& \le C s^{-2/3}\; Cs^{1/6} \; C s^{5/6} \; Cs^{-1/3} \nu^{-q} \ \le\ C \nu^{-q}.
\end{align*}
Therefore, 
$
\bigabs{1 + \left[ \eta_{s}' Ai(\eta_{s}')\right]^{-1} A_{2}(\eta)(\eta-\eta_{s}')}^{-1} \le 1 + C\nu^{-q}$. Combining these estimates, we find $\abs{a_{\nu}'}\le C$.
%
\item[$\bullet$]
The case $1\le s<s_{1}'$ is handled similarly. Since we do not have an explicit  approximate formula for the zeros of $\D Ai$ zeros, we introduce the among the first  $s_1'$ zeros
\beq		\label{eq:small-etap-spacing}
2 \delta' = \min_{1\le s\le s_{1}'} \bigabs{\eta_{s+1}'-\eta_{s}'}.
\eeq
For $1\le s<s_{1}'$, define $A_{1}^{s}$ and $A_{2}^{s}$ via~\eqref{eq:A1def} and~\eqref{eq:A2def}. Then,
\beq
a_{\nu}' = \frac{1 + Ai(\eta_{s}')^{-1} A_{1}^{s}(\eta)(\eta-\eta_{s}')^{2}}{1 + \eta_{s}' Ai(\eta_{s}')^{-1} A_{2}^{s}(\eta)(\eta-\eta_{s}')}.
\eeq
In the disk $\abs{\eta-\eta_{s}'} \le \delta' \nu^{-q}$ (in the theorem $Cs^{-1/3}=\delta'$) we use the Cauchy integral formula to make the estimates
$
\abs{A_{1}^{s}(\eta)} \le C $ and $ \abs{A_{2}^{s}(\eta)} \le C
$,
which give
$
\abs{a_{\nu}'} \le C.\nn
$
\end{remunerate}
%
This completes the proof of Claim~\ref{thm:big-l-rouche} and therewith, Part 1 of Theorem~\ref{thm:big-l}.
\qquad\endproof

\subsection{Proof of  Theorem~\ref{thm:big-l} Part 2}\label{sec:1res}
We now discuss the single resonance, away from the arc. 
 As before, we begin by seeking  solutions of Equation~(\ref {eq:deformation-eqn}) for large $l$ using the asymptotics in Equations~(\ref {eq:Hnu-nuz}) and~(\ref {eq:dHnu-nuz}). However, we now seek zeros in a domain not associated with zeros of $\D Ai(\eta)$. In particular, we focus where $\abs{\eta}$ is large and $\abs{\arg \eta} < \pi$.  Using the leading term of the asymptotic approximation to Airy functions in Equations~(\ref {eq:Airy-large-arg}) and~(\ref {eq:Airyp-large-arg}) and Equation~(\ref {eq:zetadef}) we  obtain the approximate equation
\beq
\nu^{2}z^{2} - \frac{ \epsilon^{2}}{2 } \frac{1}{\We } \left( \nu^{2}-\frac{ 9}{4 } \right) = - \epsilon^{2} \nu \frac{1}{\We }  \left( \nu^{2}-\frac{ 9}{4 } \right) \left( 1-z^{2} \right)^{1/2}
\eeq
where $\lambda=\nu z/\epsilon$.
For $\nu\gg 3/2$, we can drop the terms involving $9/4$, square both sides, and write an approximate polynomial equation for the resonances in $z$:
\beq
z^{4} + \nu^{2} \epsilon^{4} \left(\tfrac{1}{\We}\right)^{2} z^{2} - \epsilon^{2} \left(\tfrac{1}{\We}\right) z^{2} - \nu^{2} \epsilon^{4} \left(\tfrac{1}{\We}\right)^{2} = 0.
\eeq
All neglected terms are of higher order in $\epsilon$ and will be absorbed into the final error.

We identify the large parameter $N=\frac{\epsilon^{2}\nu}{\We}$, set $\alpha=\frac{\epsilon^{2}}{\We}$ and obtain:
\beq
z^{4}+N^{2}z^{2}-\alpha^{2}z^{2}-N^{2} = 0.
\eeq
In the regime
, $l\epsilon^{2}\We^{-1}\ge K_{2}$, we have $N\ge K_{2}$, and the dominant balance is between the first two terms. Choosing a solution in the lower half plane leads to $z\approx -i N$. Accordingly, we look for a solution in the form
\beq
z(N) = -i N + z_{0} + z_{-1} N^{-1} + z_{-2} N^{-2} + \bigO\!\left( N^{-3} \right).
\eeq
Iteratively, we compute $z_{0}=0$, $z_{-1}=i(\alpha-1)/2$, and $z_{-2}=0$. Changing variables back to $\lambda$ leads to the desired result.
 This completes the proof of Theorem~\ref{thm:big-l}.

\appendix
\section{Spherical coordinates} \label{app:spherical}

\noindent {\bf Spherical coordinates in $\R^3$:}
\begin{equation}
x = r\sin\theta\cos\phi,\ \ y=r\sin\theta\sin\phi,\ \ z=r\cos\theta
\label{spherical-coord}
\end{equation}

\noindent {\bf Standard orthonormal frame}
\begin{equation}
\unitr= \begin{pmatrix} 
\sin\theta \cos\phi\\
 \sin\theta \sin\phi\\
 \cos\theta
 \end{pmatrix}
 ,\ \ \unittheta=
 \begin{pmatrix} 
\cos\theta \cos\phi\\
 \cos\theta \sin\phi\\
 -\sin\theta
 \end{pmatrix} = -\D_\theta\unitr
 ,\ \ \unitphi=
 \begin{pmatrix} 
-\sin\theta \sin\phi\\
 \sin\theta \cos\phi\\
0
 \end{pmatrix} = \D_\phi\unitr
\label{standardS2}\end{equation}

\noindent {\bf $\vect{grad},\ \vect{div}, \Delta$ in spherical coordinates :}
\begin{align}
{\vect{grad}}\ {f}\ &=\ \grad{f} = 
\frac{\partial f}{\partial {r}}\unitr + \frac{1}{r}\frac{\partial f}{\partial {\theta}}\unittheta + \frac{1}{r\sin \theta}\frac{\partial f}{\partial {\phi}}\unitphi \label{grad-polar}\\
{\rm div}\ \vect{F} &= \div\vect{F} = \frac{1}{r^{2}}\frac{\partial}{\partial{r}}(r^{2}F_{r}) + \frac{1}{r\sin\theta}\frac{\partial}{\partial{\theta}}(\sin\theta F_{\theta}) + \frac{1}{r\sin\theta}\frac{\partial}{\partial\phi}F_{\phi} \label{div-polar} \\
 {\rm div}\ {\vect{grad}} f &= \Delta f 
= \Delta_{r}f + \frac{1}{r^{2}}\Delta_{S}f\nn\\
\Delta_r\ f &= \frac{1}{r^{2}}\ \frac{\partial }{\partial{r}}\left(r^{2}\ \frac{\partial f}{\partial{r}}\right),\ \ {\rm radial\ Laplacian}\label{Delta_r}\\
\Delta_S\ f\ &= \left[\frac{1}{\sin\theta}\ \frac{\partial}{\partial\theta}\left(\sin\theta\ \frac{\partial f}{\partial\theta}\right) + \frac{1}{\sin^{2}\theta}\ \frac{\partial^{2} f }{\partial\phi^{2}}\right],\ 
 {\rm Laplacian\ on}\ S^2\label{Delta_S}
\end{align}

\noindent {\bf Surface and Volume elements:}
$
d\Omega = \sin\theta\, d\theta\, d\phi,\ \  
d\bx = r^2\sin\theta\ dr\, d\theta\, d\phi =  r^2\, dr\, d\Omega
$
%
\section{Spherical Harmonics} \label{app:Ylm}
$Y_{l}^{m}(\theta,\phi) = Y_{l}^{m}(\Omega)$, $l\ge0,\ |m|\le l$ denote the spherical harmonics. These are eigenfunctions of the Laplacian on $S^{2}$:
\beq
-\Delta_{S} Y_{l}^{m}(\theta,\phi)\ =\ l(l+1) Y_{l}^{m}(\theta,\phi).
\eeq
The functions $\{Y_l^m\}$ form a complete orthonormal system in $L^2(S^2)$:\\
$
\langle Y_l^m,Y_{l'}^{m'} \rangle_{L^2(S^2)} = \delta_{l l' ,m m'}
$
with  $L^2(S^2)$  inner product:
\beq
\langle f, g \rangle_{L^2(S^2)} =  \int_0^{2\pi}\ \int_0^\pi\ f(\theta,\phi)\ \overline{g(\theta,\phi)} \sin \theta\ d\theta\,d\phi =
 \int_{S^2} f(\Omega)\ \overline{g(\Omega)}\ d\Omega
\eeq

We shall make use of the following explicit formulae:
\begin{align}
Y_0^0(\theta,\phi) &= \frac{1}{2} \sqrt{\frac{1}{\pi}}\label{Y00}\\
Y_1^{-1}(\theta,\phi) &=  \frac{1}{2}\sqrt{\frac{3}{2\pi}} \sin\theta\ e^{-i\phi},\ \, 
Y_1^{0}(\theta,\phi) =  \frac{1}{2}\sqrt{\frac{3}{\pi}} \cos\theta,\ \, 
Y_1^{1}(\theta,\phi) =  -\frac{1}{2}\sqrt{\frac{3}{2\pi}} \sin\theta\ e^{i\phi}\label{Ypm10}
\end{align}
%
%
%

\section{PDEs for a bubble in a compressible fluid and linearization} \label{app:derivations}
The inviscid, compressible fluid phase is governed by the Euler and conservation of mass equations with normal velocity and stress boundary conditions:
%
\begin{subequations}  \label{eq:euler}
\begin{align}
\vect{0} &=\D_{t} \fixed{\vect{u}} + \left( \fixed{\vect{u}} \cdot \grad \right) \fixed{\vect{u}} + \frac{1}{\fixed{\rho}} \grad \fixed{p}(\rho), && \fixed{\vect{x}}\in \R^{3}\setminus \fixed{B}(t) \label{mom} \\
0 &= \D_{\fixed{t}} \fixed{\rho} + \div (\fixed{\rho}\fixed{\vect{u}}), && \fixed{\vect{x}}\in \R^{3}\setminus \fixed{B}(t) \label{mass1}\\
\fixed{\vect{u}} \cdot \unitn &= \D_{t}  \fixed{\vect{R}} \cdot \unitn, &&   \D\fixed{B}(\fixed{t})\label{kin} \\
\left.p_{\rm bubble}\right|_{\D B(t)} - \left.p_{\rm fluid}\right|_{\D B(t)} &= \sigma\, \div\unitn, &&   \D\fixed{B}(\fixed{t}).\label{eq:dynamic}
\end{align}
\end{subequations}
The pressure within the fluid is assumed to obey an equation of state: 
$p=p_{\rm fluid}=\fixed{p}(\fixed{\rho})$.
The pressure of the gas within the bubble is assumed to be 
constant and to vary inversely with bubble-volume:
\beq
p_{\rm bubble}\ =\ \fixed{P}_{\fixed{B}}(\fixed{\vect{R}}) = \frac{k}{\abs{\fixed{B}(\fixed{t})}^{\gamma}}.
\eeq
For air (composed of mostly diatomic gases), the ratio of specific heats $\gamma = 1.4$. Equation~(\ref {eq:dynamic}) is called the Young-Laplace equation~\cite{Leal:2007fk} and states that the discontinuity in pressure at the interface is proportional to the mean curvature $H$ due to the relation $\div\unitn=2 H$; see Lamb~\cite{Lamb:1993mu} (Article 275, Equation 5), or do Carmo~\cite{Carmo:1976ve} (Section 3-3, Example 5).

Equations~(\ref {eq:euler}) have a spherically symmetric equilibrium solution:
\begin{align}
{\vect{u}} &= \vect{0},\qquad
{\vect{R}}(\theta,\phi) = a \,\unitr(\theta,\phi) \\
\fixed{p} &= p_{\infty},\qquad
\fixed{\rho} = \rho_{\infty},\ \  \ \ 
\fixed{P}_{\fixed{B}} = P_{\text{eq}}
\end{align}
where the equilibrium bubble radius, $a$ and pressure, $P_{\text{eq}}$ are given by
~\eqref{eq:Peq} .
 \medskip
 
To make the fluid compressibility more explicit, we define a variable sound speed 
$
\fixed{c}^{2} = \frac{d\fixed{p}}{d\fixed{\rho}}
$
and rewrite the continuity equation~\eqref{mass1} in terms of the pressure, $p=p(\rho)$:%
\beq
\D_tp + \grad p \cdot {\vect{u}} + {c}^{2} {\rho} \div{\vect{u}}=0.
\eeq

%
%
%
\subsection{Nondimensionalization of the compressible Euler equations}
We shall use the equilibrium bubble radius, $a$, as a length scale, and a typical bubble wall velocity, $U$, as a velocity scale to nondimensionalize the compressible Euler equations~\eqref{eq:euler}. In this section, we denote the nondimensional variables with subscript asterisks. Thus, the non-dimensional spatial coordinate is denoted $\ndim{\fixed{\vect{x}}}$.
Now, introduce non-dimensional variables:
\begin{align}
\fixed{B} &= a\, \ndim{\fixed{B}}	& \fixed{\vect{x}} &= a \,\ndim{\fixed{\vect{x}}}	& \fixed{t} &= \frac{a}{U} \, \ndim{\fixed{t}}	& \fixed{\vect{u}}	 &= U \, \ndim{\fixed{\vect{u}}} 	& \fixed{p} &= \rho_{\infty}U^{2} \, \ndim{\fixed{p}}\\
\fixed{\D B} &= a\, \ndim{\fixed{\D B}}	& \grad_{\fixed{\vect{x}}} &= \frac{1}{a} \, \grad_{\ndim{\fixed{\vect{x}}}}	& \D_{\fixed{t}} &= \frac{U}{a} \, \D_{\ndim{\fixed{t}}}	& \fixed{\rho} &= \rho_{\infty} \, \ndim{\fixed{\rho}}	& \fixed{P_{\fixed{B}}} &= \rho_{\infty}U^{2} \, \ndim{\fixed{P}}{}_{\ndim{\fixed{B}}}
\end{align}
In terms of these non-dimensional variables, the first two equations of the Euler system become:
\begin{align}
\D_{\ndim{\fixed{t}}} \ndim{\fixed{\vect{u}}} + \left( \ndim{\fixed{\vect{u}}} \cdot \grad_{\ndim{\fixed{\vect{x}}}} \right) \ndim{\fixed{\vect{u}}} + \frac{1}{\ndim{\fixed{\rho}}} \grad_{\ndim{\fixed{\vect{x}}}} \ndim{\fixed{p}}(\ndim{\fixed{\rho}}) &= \vect{0}, && \ndim{\fixed{\vect{x}}} \in \R^{3}\setminus \ndim{\fixed{B}}(\ndim{\fixed{t}}) \label{eq:ndim-euler}\\
\D_{\ndim{\fixed{t}}} \ndim{\fixed{\rho}} + \grad_{\ndim{\fixed{\vect{x}}}}\!\cdot (\ndim{\fixed{\rho}} \ndim{\fixed{\vect{u}}}) &= 0, && \ndim{\fixed{\vect{x}}} \in \R^{3}\setminus \ndim{\fixed{B}}(\ndim{\fixed{t}}). \label{eq:ndim-mass}
\end{align}

\begin{remark}\label{remark:irrotational}
If $\grad_{\ndim{\fixed{\vect{x}}}} \wedge \ndim{\fixed{\vect{u}}}(t=0,\ndim{\fixed{\vect{x}}})=\vect{0}$, \ie{}~the initial velocity field is irrotational, it follows that $\grad_{\ndim{\fixed{\vect{x}}}} \wedge \ndim{\fixed{\vect{u}}}(t,\ndim{\fixed{\vect{x}}})=\vect{0}$ for all $t$; see~\cite{Lamb:1993mu}. In this case $\ndim{\fixed{\vect{u}}}(t,\ndim{\fixed{\vect{x}}})=\grad_{\ndim{\fixed{\vect{x}}}} \ndim{\fixed{\varphi}}(t,\ndim{\fixed{\vect{x}}})$.
\end{remark}
Alternatively, we may write the continuity equation~\eqref{eq:ndim-mass} in terms of the non-dimensional pressure, $\ndim{\fixed{p}}$, as:
\beq
\D_{\ndim{\fixed{t}}} \ndim{\fixed{p}} + \grad_{\ndim{\fixed{\vect{x}}}} \ndim{\fixed{p}} \cdot \ndim{\fixed{\vect{u}}} + {\ndim{\fixed{c}}}^{2} \ndim{\fixed{\rho}} \grad_{\ndim{\fixed{\vect{x}}}}\!\cdot \ndim{\fixed{\vect{u}}}=0.
\eeq
Here, we have introduced the non-dimensional speed:
$
\ndim{\fixed{c}}^2 = \left(\frac{c}{U}\right)^2 = \frac{d \ndim{\fixed{p}}}{d\ndim{\fixed{\rho}}}(\ndim{\fixed{\rho}}).
$
 Note that $\ndim{\fixed{p}}$, $\,\ndim{\fixed{\vect{u}}}$, and $\, \ndim{\fixed{\rho}}$ are all functions of $(\ndim{\fixed{\vect{x}}},\ndim{\fixed{t}})$.

Next, we nondimensionalize the boundary conditions. Using the substitutions above we have
\begin{align}
\ndim{\fixed{\vect{u}}} \cdot \unitn &= \D_{\ndim{\fixed{t}}}  \ndim{\fixed{\vect{R}}} \cdot \unitn, && \ndim{\fixed{\vect{x}}} \in \D\ndim{\fixed{B}}(\ndim{\fixed{t}}),	 \label{eq:ndim-kinematic}	\\
\ndim{\fixed{P}}{}_{\ndim{\fixed{B}}} - \ndim{\fixed{p}} &= \frac{1}{\We}\, \grad_{\ndim{\fixed{\vect{x}}}} \!\cdot \unitn, && \ndim{\fixed{\vect{x}}} \in \D\ndim{\fixed{B}}(\ndim{\fixed{t}}) \label{eq:ndim-dynamic}
\end{align}
where $\We$ denotes the Weber number, denoted by
\beq
\We = \frac{\rho_{\infty} a U^{2}}{\sigma} = \frac{\text{INERTIA}}{\text{CURVATURE}}.
\eeq
We now express the dynamic boundary condition more explicitly.
 First note
$
\abs{\fixed{B}(t)} = \abs{a\ndim{\fixed{B}}(\ndim{\fixed{t}})} = a^{3} \abs{\ndim{\fixed{B}}(\ndim{\fixed{t}})}.
$
Furthermore, the bubble pressure at equilibrium is given by 
$
P_{eq}\ =\ \frac{k}{\left(\frac{4\pi}{3}a^3\right)^\gamma}\ 
\implies\ k=P_{eq}\ \left(\frac{4\pi}{3}a^3\right)^\gamma
\ \implies\ P_{B(t)} = P_{eq}\ \left( \frac{\frac{4\pi}{3}a^3}{ |B(t)|} \right)^{\gamma}$.
Thus, using the expression for $P_{eq}$ in~\eqref{eq:Peq}, we have 
\begin{equation}
\begin{split}
{\ndim{\fixed{P}}}_{\ndim{\fixed{B}}} &= \frac{\fixed{P_{\fixed{B}}}}{\rho_{\infty}U^{2}}  
\ =\ \frac{P_{\text{eq}}}{\rho_{\infty}U^{2}} \left( \frac{\frac{4\pi}{3}}{\abs{\ndim{\fixed{B}}(\ndim{\fixed{t}})}} \right)^{\gamma} \\
&=\ \left( \frac{p_{\infty}+2\sigma/a}{\rho_{\infty}U^{2}} \right) \left( \frac{\frac{4\pi}{3}}{\abs{\ndim{\fixed{B}}(\ndim{\fixed{t}})}} \right)^{\gamma}\ =\  \left( \frac{\Ca}{2} + \frac{2}{\We} \right) \left( \frac{\frac{4\pi}{3}}{\abs{\ndim{\fixed{B}}(\ndim{\fixed{t}})}} \right)^{\gamma}
\label{eq:ndim-PB}
\end{split}
\end{equation}
Here, $\Ca$ denotes the  {\it Cavitation number}
\beq
\Ca = \frac{p_{\infty}}{\tfrac{1}{2}\rho_{\infty}U^{2}} = \frac{\text{EXTERNAL PRESSURE}}{\text{KINETIC ENERGY PER VOLUME}}.
\eeq
The nondimensional system is given by Equations~(\ref {eq:ndim-euler}),  (\ref {eq:ndim-mass}),  (\ref {eq:ndim-kinematic}) and~(\ref {eq:ndim-dynamic}).

Finally, we conclude this section by displaying the spherical bubble equilibrium  solution in non-dimensional variables:
\begin{align}
\ndim{\fixed{\vect{u}}} &=\vect{0},\ \ \ 
\ndim{\fixed{\vect{R}}}=\unitr(\theta,\phi),\ \ \ 
\ndim{\fixed{p}} =\frac{1}{2} \Ca,\ \ \ 
\ndim{\fixed{\rho}} =1,\ \ \ 
\ndim{\fixed{P}}{}_{\ndim{\fixed{B}}} = \frac{\Ca}{2}+\frac{2}{\We}
\label{eq:nondimbub}
\end{align} 
%
%
%
%
\subsection{Nondimensional Euler equations in Center of Mass Coordinates}
The bubble's dimensional center of mass is defined
\beq
\com(\fixed{t}) := \int_{\fixed{B}(\fixed{t})} \fixed{\vect{x}} \,d\fixed{\vect{x}}
\eeq
and is nondimensionalized by $\com(t) = a\, \ndim{\com{}}(\ndim{t})$.
In this section, variables and operators defined with respect to the moving coordinate system are denoted with a prime ($'$). Let the moving coordinates be defined
\begin{align}
\ndim{\moving{\vect{x}}}(\ndim{\fixed{t}}) := \ndim{\fixed{\vect{x}}} - \ndim{\com{}}(\ndim{\fixed{t}})
\end{align}
such that 
\beq
\int_{\ndim{\moving{B}}(\ndim{\fixed{t}})} \ndim{\moving{\vect{x}}} \;d\moving{\ndim{\vect{x}}} = \vect{0}.
\eeq
We will use the following substitutions.
\begin{align}
\grad_{\ndim{\fixed{\vect{x}}}} &= \grad_{\ndim{\moving{\vect{x}}}},\ \ \ \unitn = \moving{\unitn},\ \ \ 
\ndim{\fixed{\vect{R}}} = \ndim{\moving{\vect{R}}} + \ndim{\com{}}(\ndim{\fixed{t}}) \nn \\
\ndim{\fixed{\vect{u}}} (\ndim{\fixed{\vect{x}}},\ndim{\fixed{t}}) &= \ndim{\moving{\vect{u}}} (\ndim{\fixed{\vect{x}}}-\ndim{\com{}}(\ndim{\fixed{t}}),\ndim{\fixed{t}}) \ = \ndim{\moving{\vect{u}}} (\ndim{\moving{\vect{x}}},\ndim{\fixed{t}})\nn \\
\ndim{\fixed{\rho}}(\ndim{\fixed{\vect{x}}},\ndim{\fixed{t}})  &= \ndim{\moving{\rho}}(\ndim{\moving{\vect{x}}},\ndim{\fixed{t}}),\ \ \ \\ 
\ndim{\fixed{p}}(\ndim{\fixed{\vect{x}}},\ndim{\fixed{t}})  &= \ndim{\moving{p}}(\ndim{\moving{\vect{x}}},\ndim{\fixed{t}})
\ndim{\fixed{P}}{}_{\ndim{\fixed{B}}} (\ndim{\fixed{\vect{R}}},\ndim{\fixed{t}})  = \ndim{\moving{P}}{}_{\ndim{\moving{B}}} (\ndim{\moving{\vect{R}}},\ndim{\fixed{t}}),  
\ndim{\fixed{c}}(\ndim{\fixed{\vect{x}}},\ndim{\fixed{t}})  &= \ndim{\moving{c}}(\ndim{\moving{\vect{x}}},\ndim{\fixed{t}}).\nn
\end{align}
The resulting system of nonlinear equations is:
\begin{subequations} 	\label{eq:euler-com}
\begin{alignat}{2}
&\vect{0} = ( \D_{\ndim{\fixed{t}}}\!-\D_{\ndim{\fixed{t}}} \ndim{\com{}} \cdot \grad_{\ndim{\moving{\vect{x}}}}) \ndim{\moving{\vect{u}}}  + \left(  \ndim{\moving{\vect{u}}} \!\cdot \grad_{\ndim{\moving{\vect{x}}}} \right)  \ndim{\moving{\vect{u}}}  +  \frac{1}{\ndim{\moving{\rho}}} \grad_{\ndim{\moving{\vect{x}}}} \ndim{\moving{p}}, &\;\;\;& \ndim{\moving{\vect{x}}} \!\in\! \R^{3}\!\setminus\! \ndim{\moving{B}}(\ndim{\fixed{t}})	\\
&0 = ( \D_{\ndim{\fixed{t}}}\!-\D_{\ndim{\fixed{t}}} \ndim{\com{}} \cdot \grad_{\ndim{\moving{\vect{x}}}}) \ndim{\moving{p}}  +  \grad_{\ndim{\moving{\vect{x}}}} \ndim{\moving{p}} \cdot \ndim{\moving{\vect{u}}}  \!+\!  \ndim{\moving{c}}{}^{2} \ndim{\moving{\rho}} \;\grad_{\ndim{\moving{\vect{x}}}} \!\cdot \ndim{\moving{\vect{u}}}, && \ndim{\moving{\vect{x}}} \!\in\! \R^{3}\!\setminus\!\ndim{\moving{B}}(\ndim{\fixed{t}})			\\
&\ndim{\moving{\vect{u}}} \cdot \moving{\unitn} = \D_{\ndim{\fixed{t}}}  \!\left( \ndim{\moving{\vect{R}}} + \ndim{\com{}} \right)\cdot \moving{\unitn}, && \ndim{\moving{\vect{x}}} \in \D\ndim{\moving{B}}(\ndim{\fixed{t}})   \\
&\ndim{\moving{P}}{}_{\ndim{\moving{B}}} \!-\! \ndim{\moving{p}} = \frac{1}{\We}\, \grad_{\ndim{\moving{\vect{x}}}} \!\cdot \moving{\unitn}, && \ndim{\moving{\vect{x}}} \in \D\ndim{\moving{B}}(\ndim{\fixed{t}}) 	\\
&\vect{0} = \int_{\ndim{\moving{B}}} \ndim{\moving{\vect{x}}} \;d\ndim{\moving{\vect{x}}}.
\end{alignat}
\end{subequations}
\subsection{Linearization of Euler equations about the spherical equilibrium} \label{sec:euler-linear}
 To simplify the notation, we drop asterisks and primes in Equations~(\ref {eq:euler-com}). and linearize about the nondimensional equilibrium~\eqref{eq:nondimbub} by taking:
 \begin{align}
 \vect{u} &= \vect{0} + \delta\, \vect{u}_{1} + \bigO\!\left(\delta^2\right),\ \ 
\vect{R}(\theta,\phi,t) = R(\theta,\phi,t)\, \unitr =
\left[ 1 + \delta\, R_{1}(\theta,\phi,t) + \bigO\!\left(\delta^2\right)\right] \unitr \label{eq:Rthetaphit} \\
p &= \frac{\Ca}{2} + \delta\, p_{1} + \bigO\!\left(\delta^2\right);\ \ p=p(\rho),\ \ 
\rho = 1 + \delta\, \rho_{1} + \bigO\!\left(\delta^2\right) \\
\com(t) &= \vect{0} + \delta\, \comone(t) + \bigO\!\left(\delta^2\right),\ \ \ 
\comone(t) =  \xi_{1,x}(t)\unitx+\xi_{1,y}(t)\unity+\xi_{1,z}(t)\unitz 
\label{delta-expand} \end{align} 

Consider initial conditions that are perturbations from equilibrium of size $\delta,\,\delta\le1$. We expand quantities in Equations~(\ref {eq:euler-com}) in powers of $\delta$:
$
c^{2} = \frac{dp}{d\rho}(1+\delta\rho_1+\dots)\ =\  \frac{1}{\Ma^{2}} + \delta\, {c_{1}}^{2} + \bigO\!\left(\delta^2\right)$,  
where $\Ma$ denotes the {\it Mach number}
\beq
\Ma= \left( \frac{dp}{d\rho}(1)\right)^{-1}=\frac{U}{c_{\infty}} = \left[ \frac{\text{INERTIA}}{\text{COMPRESSIBILITY}} \right]^{1/2}.
\eeq

In addition to the above expansions in powers of $\delta$, we require the implied expansions for $\unitn$ and $\vect{div}\ \unitn$.
From the relation
\begin{equation}
F(r,\theta,\phi, t) \,\equiv\,  r\ -\ R(\theta,\phi,t)\ =\ 0\nn 
\end{equation}
which defines the bubble surface, $\D B(t)$, we have using the polar coordinate representations,~\eqref{grad-polar} and~\eqref{div-polar},  for $\vect{grad}$ and ${\rm div}$ 
\begin{align}
\begin{split}
\left.\unitn\,\right|_{r=R}   \ &=\ \left.\frac{\vect{grad}\ F}{|\vect{grad}\ F|}\,\right|_{r=R} 
\ =\ \frac{\unitr-\frac{1}{R}\D_\theta R\ \unittheta-\frac{1}{R\sin\theta}\D_\phi R\ \unitphi }{ \left(1+ \frac{1}{R^2}(\D_\theta R)^2+\frac{1}{R^2\sin^2\theta}(\D_\phi R)^2 \right)^{1\over2}}
\\ &= 1\cdot \unitr - \delta\, \left( \dtheta R_{1} \unittheta + \frac{1}{\sin \theta} \dphi R_{1} \unitphi\right) + \bigO\!\left(\delta^2\right), \label{unitn-expand}
\end{split}\\
\begin{split}
\left.\left( {\rm div}\ \unitn \right) \right|_{r=R} &= \left.\left[ {\rm div}\left( 1\cdot\unitr\ -\ \delta\ \D_\theta R_1\ \unittheta\ -\ \delta\ \tfrac{1}{\sin\theta}\ \D_\phi R_1\ \unitphi\ +\ \cO(\delta^2) \right) \right]\right|_{r=R}
  \\ &= 2 - \delta\ (2 + \Delta_{S}) R_{1} + \cO(\delta^{2})
\end{split}
\label{divunitn-expand}
\end{align}
  
It follows that 
\begin{align}
\vect{u} \cdot \unitn &= 0\ + \delta\ \dr \vect{u}_{1} + \cO(\delta^{2})  \nn \\
\dt\com \cdot \unitn &= 0 + \delta\ 
\left(\ \dt\xi_{1,x}(t)\unitx+\dt\xi_{1,y}(t)\unity+\dt\xi_{1,z}(t)\unitz\right)\cdot \unitr + \cO(\delta^{2})\nn\\
&=\ \ \delta\left( \dt\xi_{1,x}(t)\sin\theta\cos\phi\ +\ 
 \dt\xi_{1,y}(t)\sin\theta\sin\phi\ +\ \dt\xi_{1,z}(t)\cos\theta\ \right) + 
 \cO(\delta^2)\nn\\
\dt \vect{R} \cdot \unitn &= 0\ + \delta\ \dt R_{1} + \cO(\delta^{2})\nn 
\end{align}

For the bubble pressure, we use the following expansion 
of the volume, $|B(t)|$:
%
%
\begin{align*}
|B(t)| = \int_0^\pi \!\!d\phi  \int_0^{2\pi}\!d\theta\, \sin\theta \!\!\int\limits_0^{R(\theta,\phi,t)}\! r^2  dr= \frac{4}{3}\pi+\delta \!\int_0^\pi \!d\phi \!\int_0^{2\pi}\!R_1(\theta,\phi,t) \sin\theta\ d\theta + \cO(\delta^2) 
\end{align*}
Therefore,
\begin{equation}
\left(\frac{\frac{4\pi}{3}}{\bigabs{{B}(t)}}\right)^{\gamma} =\  1 -\ 3\gamma\ \delta\  \langle Y_0^0 , R_{1} \rangle_{L^2(S^2)}\ Y_0^0 + \cO(\delta^{2}),\ \ \ {\rm where}\ \ Y_0^0(\theta,\phi) = \frac{1}{2} \sqrt{\frac{1}{\pi}} .
\nn\end{equation}

The center of mass constraint can be expanded is follows:
\begin{align}
\vect{0} &= \int_{B(t)}\bx d\bx =\int_0^{2\pi}\ d\theta\  \int_0^\pi\ d\phi\ 
 \int_0^{R(\theta,\phi,t)}\ r\ \unitr(\theta,\phi)\ r^2 \sin\theta\ dr\\  
             &= 
              \frac{1}{4}\ \int_0^{2\pi} \int_0^\pi 
             R^4(\theta,\phi,t)
             \begin{pmatrix}
             \sin\theta \cos\phi\\
             \sin\theta \sin\phi\\
             \cos\theta
             \end{pmatrix}  
              \sin\theta\ d\theta\ d\phi \nn 
\end{align}
Since $R=1 + \delta\ R_1 + \cO(\delta^2)$, we have at order $\delta$:
\begin{equation}
\int_{0}^{2\pi} \int_{0}^{\pi} 
             R_1(\theta,\phi,t)
             \begin{pmatrix}
             \sin\theta \cos\phi\\
             \sin\theta \sin\phi\\
             \cos\theta
             \end{pmatrix}  
              \sin\theta\ d\theta\ d\phi = \vect{0}
              \label{R1orthog-spherical}
\end{equation}
Equivalently, we can express these three orthogonality conditions
in terms of spherical harmonics of degree $l=1$: $Y_l^1(\theta,\phi)$, $Y_1^0(\theta,\phi)$ and $Y_1^{-1}(\theta,\phi)$.  Indeed, by Equation~(\ref {Ypm10})
%
\begin{align}
\sqrt{\frac{3}{2\pi}} \sin\theta \cos\phi &= Y_1^{-1} - Y_1^1,\ \ 
\sqrt{\frac{3}{2\pi}} \sin\theta \sin\phi = i\left( Y_1^{-1} + Y_1^1 \right),\ \ 
 \frac{1}{2}\sqrt{\frac{3}{\pi}}\ \cos\theta =  Y_1^{0}\label{Yc}.
\end{align}
Therefore, the three orthogonality conditions of Equation~\eqref{R1orthog-spherical} are equivalent to
$\langle R_1(\cdot,t) , Y_1^m(\cdot)\rangle_{L^2(S^2)} = 0$, $ m=-1,0,1$.
We can also re-express $\D_t\com\cdot\unitn$, given above, in terms of $Y_1^m$ as follows
%
 \begin{align}
& \left(\D_t\com\cdot\unitn\right)(\theta,\phi,t)\ =\ 
\delta\ \left( \D_t\com\cdot\unitr \right)(\theta,\phi,t)+ \cO(\delta^2)\nn\\
 &= \sqrt{\frac{\pi}{3}} \delta \left(\!\sqrt{2}\dt\xi_{1,x}(t)\left(Y_{-1}^1-Y_1^1\right) \!+\! \sqrt{2}\, i \,
 \dt\xi_{1,y}(t)  \left(Y_1^{-1}+Y_1^1\right)\! +  \!2\dt\xi_{1,z}(t) Y_1^0 \!\right) \!+\!  \cO(\delta^2). \label{dtcomdotn}
 \end{align}
%
Substitution of these expansions into Equations~(\ref {eq:euler-com}) yields the following equations for the linearized perturbations to the fluid phase:
\begin{align}
0 &= \dt \vect{u}_1 +  \grad p_{1},\ \ \ \ 
0 = \dt p_{1} + \frac{1}{\Ma^{2}}  \div \vect{u}_{1}
\end{align}

Assuming the flow is irrotational, we have
$
 \vect{u}\ =\ \delta\ \vect{u}_1\ +\ \delta^2\ \vect{u}_2\ +\ \cO(\delta^3)\ =\  \delta\ \grad \varphi_{1}\ + \ \delta^2\ \grad \varphi_{2}\ +\ \cO(\delta^3).
$
Since $\vect{u}_1=\grad{\varphi}_1$, we have
 $\vect{0} = \grad \left(\dt \varphi_{1} + p_{1} \right)$ and therefore
\begin{align} 
0 = \dt \varphi_{1} + p_{1}, \ \  \
0 = \dt p_{1} + \frac{1}{\Ma^{2}}  \Delta\varphi_{1}\nn
\end{align}
implying the wave equation for $\varphi_1$:
\beq
\Ma^{2} \dtt \varphi_{1} - \Delta \varphi_{1} = 0, \quad r>1,\ t>0
\label{eq:wave}\eeq
The linearized boundary conditions and center of mass constraint are 
\begin{align}
&\dr \varphi_{1} = \dt R_{1} +  \left(\D_t\com\cdot\unitr\right)(\theta,\phi,t), &r=1,\ t\ge0 \label{kinematicR1}\\
&\dt \varphi_{1} =\ 3 \gamma\ \left( \frac{\Ca}{2} + \frac{2}{\We} \right)\ \langle Y_0^0 , R_{1} \rangle Y_0^0  - \frac{1}{\We} (2 + \Delta_{S}) R_{1} , & r=1,\ t\ge0 \label{dynamic-phi1}\\
&\langle R_{1}(\cdot,t), Y_{1}^{m} \rangle = 0,  & |m|\le 1 \label{R1orthog}
\end{align}
As initial conditions we take:
\beq  \label{initial-data}
\begin{split} 
&\left. R_1(\theta,\phi,t)\right|_{t=0}\ \  {\rm sufficiently\ smooth}\\
& \left. \varphi_1(r,\theta,\phi,t)\right|_{t=0},\ \left.\D_t\varphi_1(r,\theta,\phi,t)\right|_{t=0}\ \equiv0,\ \ r>1,\\
&\left. \comone(t)\right|_{t=0}=\vect{0}.
\end{split}
\eeq
\begin{proposition}\label{proposition:xi10}
Let $\varphi_1(r,\theta,\phi,t),\ R_1(\theta,\phi,t),\ \comone(t)$ denote a sufficiently regular solution of the initial-boundary value problem for the wave equation~\eqref{eq:wave} on the region $r>1$ with 
  linearized kinematic and dynamic boundary conditions~\eqref{kinematicR1} and~\eqref{dynamic-phi1},\ linearized coordinate normalization  (center of mass) condition,\ and initial conditions~\eqref{initial-data}.
Then, 
\begin{align}
\comone(t) &= \vect{0},\ \ t\ge0.
\end{align}
It follows that the linearized perturbation: $\Psi=\varphi_{1}$, $\beta=R_{1}$ satisfies the initial-boundary value problem~\eqref{eq:n3linear},~\eqref{ib-data} of the Introduction.
\end{proposition}
%

To prove Proposition~\ref{proposition:xi10}, we first 
introduce the  projection operators onto the $|m|=1$ spherical harmonics
\begin{equation}
\mathcal{P}_{1}^{m}=\langle\,\cdot\,,Y_{1}^{m}\rangle_{L^2(S^2)} \ Y_{1}^{m}.\label{P1mdef}
\end{equation}
Since ${\cal P}_1^m$ commutes with $\D_t$ and $\Delta$, for any $|m|\le 1$, the function $U= \mathcal{P}_{1}^{m} \dt\varphi_{1}$ satisfies the wave equation
\begin{equation}
\left( \Ma^2\D_t^2 - \Delta \right) U = 0, \ r>1 
\label{waveU}\end{equation}
Moreover,  applying $\mathcal{P}_{1}^{m}$ to the dynamic boundary condition~\eqref{dynamic-phi1} and using~\eqref{R1orthog} yields the Dirichlet boundary condition
\begin{equation}
U = 0, \ \ r=1.
\label{UDirichlet}
\end{equation}
It follows that $U=\mathcal{P}_{1}^{m} \dt\varphi_{1}$ is identically zero on $r>1$. Therefore, 
\begin{equation}
\mathcal{P}_{1}^{m} \varphi_{1}(r,\theta,\phi,t) = F(r)\ Y_1^m(\theta,\phi).
\nn\end{equation}
Since $\mathcal{P}_1^m \varphi_1$ is a time-independent solution of the wave equation for $r>1$, clearly $\Delta \mathcal{P}_1^m \varphi_1=0$, and therefore $(\Delta_{r}-2r^{-2})F(r)=0$. 
%
%
Thus, $F(r) = c_1\, j_1(r) + c_2 \,y_1(r)$ for some constants $c_1, c_2$ and therefore
\beq \label{c1c2}
\begin{aligned}
\mathcal{P}_{1}^{m} \varphi_{1}(r,\theta,\phi,t) = \left( c_1\, j_1(r) + c_2\, y_1(r)\right)\ Y_1^m(\theta,\phi),\ \ {\rm and}\\
\D_r \mathcal{P}_{1}^{m} \varphi_{1}(r,\theta,\phi,t) = \left( c_1 \,\D_r j_1(r) + c_2\, \D_r y_1(r)\right)\ Y_1^m(\theta,\phi).
\end{aligned}
\eeq
Evaluating~\eqref{c1c2} at $r=1$ and $t=0$ we have, using the initial condition 
 $\varphi_1(r,\theta,\phi,t=0)=0$ for $ r>1$, 
 \begin{equation}
 c_1\,j_1(1)+c_2\,y_1(1)=0,\qquad c_1\, \D_r j_1(1)+c_2\, \D_ry_1(1)=0.
 \nn\end{equation}
 implying $c_1=c_2=0$, by linear independence of $\{j_1(r),y_1(r)\}$.
Now applying  $ \mathcal{P}_1^m$ to the kinematic boundary condition~\eqref{kinematicR1} and using~\eqref{R1orthog}, we obtain
 \begin{equation}
 \left\langle\ Y_1^m , \left(\D_t\comone(t)\cdot\unitr\right)\  \right\rangle_{L^2(S^2)}\ =\ 0,\ \ |m|\le1
 \label{xi1r-orthog}
 \end{equation}
 
 We now claim that $\D_t\comone(t)=\vect{0},\ \ t\ge0$ and therefore $\comone(t)=\vect{0},\ t\ge0$. To see this, observe that 
\begin{align}
\D_t\comone(t)\cdot\unitr\ &=\ \D_t\xi_{1x}(t)\sin\theta\cos\theta
 \ +\ \D_t\xi_{1y}(t)\sin\theta\sin\phi\ +\ \D_t\xi_{1z}(t)\cos\theta\nn\\
 &= \sqrt{\frac{2\pi}{3}} \left[\D_t\xi_{1x}(t) (Y_1^{-1}\!-\!Y_1^1) +
  i \, \D_t\xi_{1x}(t) (Y_1^{-1}\!+\!Y_1^1) + 
   \tfrac{1}{\sqrt{2}}\D_t\xi_{1z}(t) Y_1^0 \right].\nn
 \end{align}
 Taking  inner product with $Y_1^m$, $|m|\le1$, yields 
$
 \D_t\xi_{1x}(t)=\D_t\xi_{1y}(t)= \D_t\xi_{1z}(t)= 0$, $t\ge0
 $. This completes the proof.
%
\section{Hankel Functions} \label{app:hankel}
%
Separation of variables, applied to the to the wave equation in $\R^{3}$, in spherical coordinates, leads to the spherical Bessel's equation 
\begin{equation}
r^{2} R''(r) + 2 r R'(r)+\left(r^{2}-l(l+1)\right)R(r) = 0.
\end{equation}
Two linearly independent solution are the spherical Bessel functions
\begin{align}
&j_{l}(r) = \sqrt{\frac{\pi}{2r}} J_{l+1/2}(r),\ \ \qquad
y_{l}(r) = \sqrt{\frac{\pi}{2r}} Y_{l+1/2}(r)
\end{align}
where $J_{l+1/2}$ and $Y_{l+1/2}$ are Bessel functions of order $l+1/2$.
The outgoing spherical Hankel function $h_{l}^{(1)}(z)$,  $z\in\C$, is obtained by taking  a linear combination of $j_l$ and $y_l$:
\beq
h_{l}^{(1)}(z) = j_{l}(z) +i y_{l}(z).
\eeq
In terms of the Hankel functions, $H_{l+1/2}^{(1)}(z) = J_{l+1/2}(z) + i Y_{l+1/2}(z)$ and 
\beq
h_{l}^{(1)}(z) = \sqrt{\frac{\pi}{2z}} H_{l+1/2}^{(1)}(z). \label{eq:hlHnu}
\eeq
\subsection{Polynomial Representation} \label{app:polyhank}
We can represent the spherical Hankel functions as in~\cite{Taylor:1998hz} by
\beq
h_{l}^{(1)}(z) = z^{-l-1}p_{l}(z)e^{iz}\label{eq:poly-rep-hank}
\eeq
where $p_{l}(z)$ is the polynomial of order $l$ given by
\begin{align}
p_{l}(z) = i^{-l-1} \sum_{k=0}^{l}\left(\frac{i}{2}\right)^{k} \frac{(l+k)!}{k!(l-k)!} z^{l-k} 
= \sum_{n=0}^{l} \frac{i^{-n-1}}{2^{l-n}} \frac{(2l-n)!}{(l-n)!n!} z^{n} 
\equiv \sum_{n=0}^{l} a_{n}^{l} z^{n} \label{eq:poly-sum}
\end{align}
\subsection{Analytic continuation}
The Hankel functions have the following analytic continuations and symmetries
\begin{align}
H_{\nu}^{(1)}(z e^{\pi i}) &= -e^{-\nu \pi i} H_{\nu}^{(2)}(z) & H_{\nu}^{(1)}(\bar{z}) &= \overline{H_{\nu}^{(2)}(z)} \label{eq:analy1}\\
H_{\nu}^{(2)}(z e^{-\pi i}) &= -e^{\nu \pi i} H_{\nu}^{(1)}(z) & H_{\nu}^{(2)}(\bar{z}) &= \overline{H_{\nu}^{(1)}(z)} = -e^{\nu \pi i}H_{\nu}^{(1)}(\bar{z}e^{\pi i}) 
\label{eq:analy2}
\end{align}
\subsection[Derivatives]{Derivatives} \label{app:derhank}
The derivative of a spherical Hankel function can be expressed in terms of spherical Hankel functions in a variety of ways:
\begin{align}
\frac{d}{dz}h_{l}^{(1)}(z) &= \frac{l h_{l}^{(1)}(z)}{z} - h_{l+1}^{(1)}(z) = h_{l-1}^{(1)}(z) - \frac{(l+1) h_{l}^{(1)}(z)}{z} 
\label{eq:diff-hank}
\end{align}
These relations are also valid for $j_{l}$, $y_{l}$, $h_{l}^{(2)}$. Similarly,
\begin{align}
\frac{d}{dz} H_{\nu}^{(1)}(z) &= \frac{\nu H_{\nu}^{(1)}(z)}{z} - H_{v+1}^{(1)}(z) = H_{v-1}^{(1)}(z) - \frac{\nu H_{\nu}^{(1)}(z)}{z} 
\label{H-recur}\end{align}
These relations are also valid for the Bessel functions $J_{l}$, $Y_{l}$, and the incoming Hankel function $H_{l}^{(2)}$.
\medskip

Finally, a limit we require is a consequence of the above recursions and the polynomial representation of $h^{(1)}$: 
\begin{equation}
\frac{z\D h_l^{(1)}(z)}{h_l^{(1)}(z)} = l-\frac{zh_{l+1}^{(1)}(z)}{h_l^{(1)}(z)}=l-\frac{p_{l+1}(z)}{p_l(z)}\
 \to\ l-(2l+1)\ =\ -(l+1),\ \ z\to0\label{eq:useful-limit}
\nn\end{equation}
\section{Asymptotic Expansions of Special Functions}
\subsection{Airy function asymptotics} \label{app:airy}
The Airy function and its derivative can be approximated for \emph{large argument} by (Olver~\cite{Olver:1954fk})  
\begin{align}
Ai(\tau ) &= \frac{1}{2} \pi^{-1/2}\tau ^{-1/4} e^{-\mu}\left[1+\bigO(\abs{\mu}^{-1})\right], & \abs{\arg \tau }&<\pi 	\label{eq:Airy-large-arg}   \\
\D Ai(\tau ) &= -\frac{1}{2} \pi^{-1/2}\tau ^{1/4} e^{-\mu}\left[1+\bigO(\abs{\mu}^{-1})\right], & \abs{\arg \tau }&<\pi	\label{eq:Airyp-large-arg} \\
Ai(-\tau)  &= \pi^{-1/2} \tau^{-1/4} \cos(\mu-\pi/4) \left[1+ \tan(\mu-\pi/4) \bigO(\mu^{-1})\right], & \abs{\arg \tau }&<2\pi/3	 	\label{eq:Airy-large-minus} \\
Ai'(-\tau)  &= \pi^{-1/2} \tau^{1/4} \cos(\mu-3\pi/4) \left[1+ \tan(\mu-3\pi/4) \bigO(\mu^{-1})\right], & \abs{\arg \tau }&<2\pi/3
\end{align}
where $\mu = \frac{2}{3} \tau ^{3/2}$.
Additionally, we have the inequalities~\cite{Olver:1954fk}
\begin{align}
\abs{Ai(\tau)}  &<  C \left( 1 + \abs{\tau}^{1/4}  \right)^{-1}  \bigabs{e^{-\mu}}, & \abs{\arg \tau }&\le\pi	\label{eq:Ai-bound}  \\
\abs{\D Ai(\tau)}  &<  C \left( 1 + \abs{\tau}^{1/4}  \right)  \bigabs{e^{-\mu}}, & \abs{\arg \tau }&\le\pi.	\label{eq:dAi-bound}
\end{align} 
\subsection{Zeros of Airy Functions}\label{sec:airy0}
All of the zeros of the Airy function and its derivative lie on the negative real axis. The larger zeros have the following approximations~\cite{Olver:1954lr,Tokita:1972lr,Watson:1952mc}.

\emph{Asymptotics for zeros of $Ai(z)$:} 
\beq 
\eta_{s} = -\left[\frac{3}{8} \pi (4s-1)\right]^{2/3} + \bigO(s^{-4/3}), \qquad s=s_{0}, \ldots, \left\lfloor \frac{l+1}{2} \right\rfloor.
\label{eq:Airyzeros}
\end{equation}

\emph{Asymptotics for zeros of $\D Ai(z)$:}
%
\beq 
\eta_{s}' = -\left[\frac{3}{8} \pi (4s-3)\right]^{2/3} + \bigO(s^{-4/3}), \qquad s=s_{0}', \ldots, \left\lfloor \frac{l+1}{2} \right\rfloor + 1.
\label{eq:DAiryzeros}
\eeq

\subsection{Hankel function asymptotics} \label{app:unifhankel}
The Hankel functions can be approximated for large order $\nu$, uniformly in $z$ in the sector $\abs{\arg{z}}<\pi$ by

\begin{align} \label{eq:Hnu-nuz}
H_{\nu}^{(1)}(\nu z) &\sim 2 e^{-\pi i/3} \left(\frac{4\zeta}{1-z^{2}}\right)^{1/4} \frac{Ai(e^{2\pi i/3}\nu^{2/3}\zeta)}{\nu^{1/3}} \\
\D H_{\nu}^{(1)}(\nu z) &\sim -4 e^{-\pi i/3} \frac{1}{z}\left(\frac{1-z^{2}}{4\zeta}\right)^{1/4}  \frac{Ai(e^{2\pi i/3}\nu^{2/3}\zeta)}{\nu^{4/3}}
\label{eq:dHnu-nuz}
\end{align}
where
\begin{equation}
\frac{2}{3} \zeta^{3/2} = \int_{z}^{1}\,\frac{\sqrt{1-t^{2}}}{t}\,dt = \log \frac{1+\sqrt{1-z^{2}}}{z} - \sqrt{1-z^{2}}.  	\label{eq:zeta-def}
\end{equation} We note the useful Taylor series
\begin{align}
\frac{2}{3} \zeta^{3/2} &= (-\log (z)+\log(2)-1)+\frac{z^2}{4}+O\left(z^4\right) && \text{ as } z\to 0,  \label{eq:zeta-z0}    \\
&= -\frac{2}{3} i \sqrt{2} (z-1)^{3/2}+\frac{3 i (z-1)^{5/2}}{5 \sqrt{2}}+O\left((z-1)^{7/2}\right) && \text{ as } z\to 1,	\label{eq:zeta-z1}	\\
&= -iz+\frac{i \pi }{2}-\frac{i}{2 z}-\frac{i}{24} \frac{1}{z^3} + \bigO\!\left(\frac{1}{z^4} \right) && \text{ as } z\to\infty.	\label{eq:zeta-zinf}
\end{align} 
which yield the limits
$
\lim_{z\to 0} \zeta(z) = \infty,\ \ \lim_{z\to 1} \zeta(z) = 0,\ \ 
\lim_{z\to \infty} \zeta(z) = \infty. 
$
For $\zeta$ close to zero,~\ref{eq:zeta-def} is inverted to yield
\beq  	\label{eq:zzeta}
z(\zeta) = 1 - 2^{-1/3} \,\zeta + \tfrac{3}{10} 2^{-2/3}\, \zeta^{2} + \tfrac{1}{700}\, \zeta^{3} + \zeta^{4} \bigO(1).
\eeq
%

\subsection{Zeros of Hankel functions}\label{app:hankel-zeros}
Concerning the zeros of  Hankel functions a consequence of the analysis in~\cite{Olver:1954lr} is the following:
\begin{theorem}
Let $\nu=l+1/2$, where $l\ge 2$ is an integer.
The $l+1$ zeros of $\D H_{\nu}^{(1)}(y)$ and the $l$ zeros  of  $H_{\nu}^{(1)}(y)$  lie near arcs in the lower half plane .
\end{theorem}

The detailed locations of the previous zeros given in the following 
\begin{proposition}
The $l$ zeros of $H_{\nu}^{(1)}(y)$ are
\beq
y_{\nu,k} = \nu \left[ 1 + 2^{-1/3}  e^{-2\pi i/3} \nu^{-2/3}\abs{\eta_{k}} + \tfrac{3}{10} 2^{-2/3} e^{-4\pi i/3} \nu^{-4/3} \bigabs{\eta_{k}}^{2} + \bigO \!\left( \nu^{-4/3} \right) \right], 
\eeq
for $k=1,\dotsc,\floor{(l+1)/2}$, where $\eta_{k}$ is the $k$-th zero of $Ai$.

The $l+1$ zeros of $\D H_{\nu}^{(1)}(y)$ are 
\beq 
y_{\nu,k}' = \nu \left[ 1 + 2^{-1/3}  e^{-2\pi i/3} \nu^{-2/3}\bigabs{\eta_{k}'} + \tfrac{3}{10} 2^{-2/3} e^{-4\pi i/3} \nu^{-4/3} \bigabs{\eta_{k}'}^{2} + \bigO \!\left( \nu^{-4/3} \right) \right], \label{eq:dH-zero}
\eeq
for $k=1,\dotsc,\floor{(l+1)/2}+1$ where $\eta_{k}'$ is the $k$-th zero of $\D Ai$.
\end{proposition}
%

We comment  on  approximating the zeros of $\D H_{\nu}^{(1)}(\nu z)$ by Equation~(\ref {eq:dH-zero}),
which is closely related to the analysis of Olver~\cite{Olver:1954fk}. The proof is analogous for $H_{\nu}^{(1)}(\nu z)$.  Introduce the  variables 
\begin{align}
\zeta &= \zeta(z), \qquad \eta = \eta(\zeta)=e^{2\pi i/3}\nu^{2/3}\zeta;\ \qquad\text{see Equation~(\ref {eq:zeta-def})}.\ \ 
	\label{eq:etazeta}
\end{align}

Using  the uniform asymptotic approximation for $\D H_{\nu}^{(1)}(\nu z)$ in~\eqref{eq:dHnu-nuz} we can rewrite the equation $\D H_\nu^{(1)}(y)=0$, for $\nu=l+1/2$ large as 
\beq
\D Ai (\eta) + \frac{e^{-2\pi i/3}}{\nu^{2/3}}  Ai(\eta)  \frac{c_{0}}{d_{0}} \left[ 1 + \bigO\!\left( \nu^{-2} \right) \right] = 0. 	\label{eq:dAi-gen-eq}
\eeq
For $\nu$ large, the zeros, $\eta$, are well approximated by zeros, $\{\eta'_k\}$, of $\D Ai(\eta)$; see~\eqref{eq:DAiryzeros}. These lie on the negative real axis; see Figure~\ref{fig:airyp-zeros} (left).
\begin{figure}[!hbt] 
\begin{center}
	\includegraphics[width=.8\textwidth]{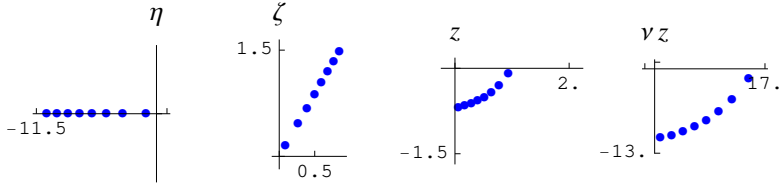}
\caption{\small Complex mappings of some $\D Ai(\eta)$ zeros (left plot) to approximate zeros of $\D H_{\nu}^{(1)}(\nu z)$  (right plot) by transformations $\eta \mapsto \zeta \mapsto z \mapsto \nu z$. The first map, $\eta \mapsto \zeta$, is given by inverting  Equation~(\ref {eq:etazeta}). The second map, $\zeta \mapsto z$, is computed using Equation~(\ref {eq:zzeta}). }
\label{fig:airyp-zeros}
\end{center}
\end{figure}
Then, the asymptotics of Airy functions and the properties of the mapping $\eta\mapsto\zeta\mapsto z\mapsto \nu z$ can be used to show give approximations to the zeros of $\D H_\nu(y)$ and to establish their location near an arc in the lower half plane.
%
%
%
\section{Proof of Proposition~\ref{proposition:tokitish2}}\label{appendix:resbound}
In this section we prove Proposition~\ref{proposition:tokitish2}, which states:
\beq
\Res_{\lambda=\lambda_{l,j}} \frac{\lambda}{\lambda^{2}+\rlhat G_{l}(\epsilon\lambda)}  = \bigO \!\left( \nu^{-4/3} \right).\label{def-res-bound}
\eeq

 The main tool of this analysis is the uniform asymptotics of Hankel functions. 

\begin{proof}
To compute the residues,  in~\eqref{beta-partial}, first  cancel a factor of $\lambda$,  use the differential equation satisfied by spherical Hankel functions and relations in Appendix~\ref{app:hankel}  to obtain: 
\beq \label{eq:lambda-res-bnd} 
\begin{split} 
&\Res_{\lambda=\lambda_{l,j}} \frac{\lambda}{\lambda^{2}+\rlhat G_{l}(\epsilon\lambda)} \ = \Res_{\lambda=\lambda_{l,j}} \frac{1}{\lambda + \epsilon \rlhat \frac{h_{l}^{(1)'}(\epsilon\lambda)}{h_{l}^{(1)}(\epsilon\lambda)}} 
\ = \lim_{\lambda\to\lambda_{l,j}} \frac{1}{\D_{\lambda}\left( \lambda + \epsilon \rlhat \frac{h_{l}^{(1)'}(\epsilon\lambda)}{h_{l}^{(1)}(\epsilon\lambda)} \right)} \\
&= \!\!\lim_{\lambda\to\lambda_{l,j}}\! \!\left[ 1 + \epsilon^{2}\rlhat 
\Biggl( -1 + \frac{l(l+1)}{(\epsilon\lambda)^{2}} 
-\frac{2}{\epsilon\lambda} \biggl[ \frac{-1}{2\epsilon\lambda} + \frac{H_{\nu}^{(1)'}(\epsilon\lambda)}{H_{\nu}^{(1)}(\epsilon\lambda)} \biggr]  - \biggl[ \frac{-1}{2\epsilon\lambda} + \frac{H_{\nu}^{(1)'}(\epsilon\lambda)}{H_{\nu}^{(1)}(\epsilon\lambda)} \biggr] ^{2}
\Biggr) \!\right]^{-1}\!\!\!. 
%
\end{split}
\eeq
In particular, we have used
\beq
\frac{h_{l}^{(1)'}(z)}{h_{l}^{(1)}(z)} = -\frac{1}{2z} + \frac{H_{\nu}^{(1)'}(z)}{H_{\nu}^{(1)}(z)},\ \ \ 
\ \nu=l+1/2. \nn
\eeq

Now we expand the terms in the last line of~\eqref{eq:lambda-res-bnd} about $\lambda_{l,j}$ for large $l$. We recall the approximation for these resonances from Theorem~\ref{thm:big-l}:
\beq
\lambda_{l,s} = \frac{l+1/2}{\epsilon} \left[1 - 2^{-1/3} (l+1/2)^{-2/3} e^{-2\pi i/3} \eta_{s}' + \bigO (l^{-1})\right]\!, \!\quad s=1, 2,\ldots, \left\lfloor \frac{l+1}{2} \right\rfloor+1.
\eeq
There will be two cases to consider: resonances corresponding to Airy prime zeros $\eta_{s}'$ for $s=1,\dotsc,s_{0}$ which are $\bigO (1)$ and for $s\approx l/2$ which we will see are $\bigO (l^{2/3})$. Specifically,
\begin{align*}
\eta_{s}' &= -\left( \frac{3\pi}{8} \right)^{2/3} (4s-3)^{2/3} + \dotsb, & s=s_{0},\dotsc,\left\lfloor \frac{l+1}{2} \right\rfloor+1\\
&= -\left( \frac{3\pi}{2} \right)^{2/3} \nu^{2/3} + \bigO (\nu^{-2/3}), & s\approx l/2 \to \infty.
\end{align*}
Since
\begin{align}
\frac{l(l+1)}{(\epsilon\lambda)^{2}} &= \frac{(\nu-1/2)(\nu+1/2)}{\nu^{2}\left[ 1-e^{-2\pi i/3} 2^{-1/3}\nu^{-2/3}\eta'_{s} \right]} 
\end{align}
we have
\beq
\frac{l(l+1)}{(\epsilon\lambda)^{2}} = \begin{cases}
	1+ \bigO (\nu^{-2/3}), & s\text{ small}, \\
	\bigO (1), & s \text{ large}.
\end{cases} 
\eeq
For the Hankel terms in~\eqref{eq:lambda-res-bnd}, letting $\epsilon\lambda=\nu z$ and using Appendix~\ref {app:unifhankel}, we have
\beq
\frac{H_{\nu}^{(1)'}(\nu z)}{H_{\nu}^{(1)}(\nu z)}  \sim e^{2\pi i/3} \frac{1}{z}\left( \frac{1-z^{2}}{\zeta} \right)^{1/2} \frac{Ai'(\tau)}{Ai(\tau)} \nu^{-1/3},\ \ \tau=e^{2\pi i/3}\nu^{2/3}\zeta.
\eeq
We simplify this ratio for $z\approx 1$ using the Taylor expansion~\eqref{eq:zeta-z1}, which yields
\begin{align}
\zeta &= 2^{1/3} e^{-\pi i /3}(z-1) + \bigO\!\left( (z-1)^{2} \right).
\end{align}
Thus,
$
\zeta^{-1}(1-z^{2})= -(1+z) \frac{z-1}{\zeta} \
=\ -\frac{1+z}{2^{1/3}e^{-\pi i /3}} \
=\  2^{2/3} e^{-2\pi i /3} + \dotsb$ and 
$e^{2\pi i /3}\nu^{2/3} \zeta =  e^{2\pi i /3} \eta_{s} + \dotsb $, 
where the dots indicate higher order terms.
Now for the Airy functions 
$
\frac{\D Ai(\tau)}{Ai(\tau)} = \bigO (1)
$
for small $s$.
From Appendix~\ref {app:airy}
\beq
\frac{\D Ai(\tau)}{Ai(\tau)} \sim - \tau^{1/2}, \qquad \tau \text{ large},\;\abs{\arg{\tau}}<\pi.
\eeq
So, for large $s$
\begin{align*}
\left.\left( \frac{\D Ai(\tau)}{Ai(\tau)} \right) \right|_{\tau=e^{2\pi i /3}\nu^{2/3} \zeta}  &\sim - \left[ e^{-\pi i /3} \left( \frac{3\pi}{2} \right)^{2/3} \nu^{2/3} \right]^{1/2} \!= - \left[ e^{-\pi i /6} \left( \frac{3\pi}{2} \right)^{1/3} \nu^{1/3}   \right].
\end{align*}
Combining results
\beq
\left.\left( \frac{\D Ai(\tau)}{Ai(\tau)} \right) \right|_{\tau=e^{2\pi i /3}\nu^{2/3} \zeta} =
\begin{cases}
	\bigO (1), & s\text{ small}, \\
	e^{5\pi i /6} \left( \frac{3\pi}{2} \right)^{1/3} \nu^{1/3}+\bigO (\nu^{-2/3}), & s\approx l/2 \text{ large}.
\end{cases}
\eeq

We can now conclude the proof. For $s$ large
\beq
-\frac{1}{2\epsilon\lambda} + \frac{H_{\nu}^{(1)'}(\epsilon\lambda)}{H_{\nu}^{(1)}(\epsilon\lambda)}  = -\frac{1}{2\nu} +\bigO (1) = \bigO (1).
\eeq
Thus,
\begin{align}
\Res_{\lambda=\lambda_{l,j}} \frac{\lambda}{\lambda^{2}+\rlhat G_{l}(\epsilon\lambda)}  &= 
\lim_{\lambda\to\lambda_{l,j}} \left[ 1 + \epsilon^{2}\rlhat 
\left( -1 + \bigO (1)
-\frac{2}{\nu} \bigO (1)  - \bigO (1)^{2}
\right) \right]^{-1} \!\! = \bigO (\nu^{-2}).\label{def-res-bound-a} 
\end{align}

For $s$ small,
\begin{align}
\Res_{\lambda=\lambda_{l,j}} \frac{\lambda}{\lambda^{2}+\rlhat G_{l}(\epsilon\lambda)}  &= 
\lim_{\lambda\to\lambda_{l,j}}  \left[ 1 + \epsilon^{2}\rlhat 
\left( \bigO (\nu^{-4/3}) + \bigO (\nu^{-2/3})
\right) \right]^{-1} \
=\ \bigO (\nu^{-4/3}).\label{def-res-bound-b} 
\end{align}
The asymptotic formula~\eqref{def-res-bound} is a consequence of~\eqref{def-res-bound-a} and~\eqref{def-res-bound-b}.
\qquad\end{proof}
%
\bibliographystyle{siam}
\bibliography{texref}
%

\end{document}